\documentclass[10pt,reqno]{amsart}
\usepackage[utf8]{inputenc}
\usepackage[margin=1in]{geometry}
\usepackage[hidelinks]{hyperref}
\usepackage{scalerel,stmaryrd} 
\numberwithin{equation}{section}
\usepackage[normalem]{ulem}
\usepackage{aligned-overset}
\usepackage{amsfonts,amssymb,amsmath,amsthm,tikz,comment,mathtools,setspace,stmaryrd,enumitem,bbm,scalerel}

\allowdisplaybreaks

\newcommand{\N}{\mathbb{N}}
\newcommand{\R}{\mathbb{R}}
\newcommand{\Q}{\mathbb{Q}}
\newcommand{\Z}{\mathbb{Z}}
\newcommand{\F}{\mathcal{F}}

\newcommand{\f}{\frac}
\newcommand{\ZHS}{\Z^2_{\mathrm{HS}}}
\newcommand{\Pp}{\mathbb P}
\newcommand{\Ex}{\mathbb E}
\newcommand{\wt}{\widetilde}
\newcommand{\ind}{\mathbf 1}
\newcommand{\GFS}{G_{\mathrm{FS}}}
\newcommand{\rhoFS}{\rho_{\mathrm{FS}}}
\newcommand{\GLPPc}{\operatorname{GLPP}_c}

\newcommand{\Geo}{\operatorname{Geo}}
\newcommand{\mbf}{\mathbf}

\newcommand{\ve}{\varepsilon}
\newcommand{\deq}{\overset{d}{=}}

\def\tspc{\hspace{1.1pt}}

\newcommand\aabullet{{\tspc\raisebox{1.5pt}{\scaleobj{0.5}{\bullet}}\tspc}}

\DeclareMathOperator*{\argmax}{arg\,max}

\newcommand{\be}{\begin{equation}}
\newcommand{\ee}{\end{equation}}

\newtheorem{theorem}{Theorem}[section]
\newtheorem{proposition}[theorem]{Proposition}
\newtheorem{corollary}[theorem]{Corollary}
\newtheorem{lemma}[theorem]{Lemma}

\theoremstyle{definition}
\newtheorem{definition}[theorem]{Definition}

\newtheorem{remark}[theorem]{Remark}

\theoremstyle{remark}

\begin{document}
\title[Invariant measures for half-space LPP]{Invariant measures for half-space geometric LPP: classification and the one force--one solution principle}
\author{Sayan Das}
\author{Evan Sorensen}
\author{Zongrui Yang}

\begin{abstract}
We prove a complete characterization of the extremal invariant measures for half-space geometric last-passage percolation with an arbitrary boundary parameter. This is the first result of its kind for a model in the KPZ universality class that has boundary effects and an unbounded domain. A description of a class of invariant measures was previously given in a work of Barraquand and Corwin, where it was conjectured that these should comprise all extremal invariant measures. To complete the classification, we prove a one force--one solution principle: when started in the distant past from an arbitrary initial condition with a given asymptotic slope at $\infty$, the recentered solution at time $0$ converges to a process which is distributed as the associated invariant measure with the specified slope. This limiting process is called the Busemann process, the first of its kind constructed for a half-space model. The Busemann process across all slopes is distributed as the joint invariant measure for geometric half-space LPP, recently constructed by Dauvergne and Zhang. There, it was conjectured that the constructed family of jointly invariant measures comprises all extremal jointly invariant measures; our analysis also confirms this conjecture. When the model has a strong (attractive) boundary, the collection of slopes for the invariant measures has a discontinuity, which does not arise in the full-space case. To handle this difficulty, we combine the control of the directions of semi-infinite geodesics with techniques from the theory of half-space Gibbsian line ensembles.  Along the way, we classify the set of directions of semi-infinite geodesics for half-space geometric LPP, confirming a recent conjecture of Dauvergne and Zhang. 
\end{abstract}

\maketitle
\tableofcontents

\section{Introduction and main results}

\subsection{The KPZ universality class and models with boundaries}
The last 25-30 years have seen tremendous progress in the study of the KPZ universality class. This class contains a huge number of random growth models with one space and one time dimension, which exhibit universal scaling exponents and limiting statistics. Major progress was made near the turn of the century in the works \cite{Baik-Deift-Johansson-1999,Johansson-2000}, where it was shown that the large-time one-point distributions of 
Poisson last-passage percolation (LPP) and geometric LPP fluctuate on the order $n^{1/3}$, and the rescaled distribution converges to the Tracy-Widom GUE distribution from random matrix theory. The proofs relied on the remarkable connection between last-passage percolation and algebraic combinatorics via the Robinson-Schendsted-Knuth correspondence. The Poisson and Geometric LPP models are said to be exactly solvable because they have tractable push-forward measures under this correspondence that are amenable to asymptotic analysis. 

Since then, the field has seen remarkable progress in understanding the multi-point correlations and scaling limits of exactly solvable models in the KPZ universality class. For example, the construction of the Airy process \cite{Prahofer-Spohn-02} and the Airy line ensemble \cite{CorwinHammond} have led to a huge number of significant developments. Over the last 10 years, this has culminated in the construction of the full space-time scaling limits of models in the KPZ universality class, namely  the KPZ fixed point \cite{KPZfixed} and directed landscape \cite{directedlandscape}.

The KPZ universality class contains many stochastic PDEs, notably the KPZ equation \cite{Amir-Corwin-Quastel-2011,Wu-2023}. Many other models in this class, including LPP and its variants, can be interpreted as discrete versions of stochastic PDEs. The full space-time scaling limits described above all correspond to the case where the spatial domain is the full line (or some discrete version of it). A substantial amount of work in the field over the last 3-5 years has turned to understanding properties of the models in the KPZ universality class on different domains, most notably a bounded interval and the half-line. This allows us to investigate the behavior of these models in the presence of boundary conditions. On a probabilistic level, these settings present many more difficulties, as the corresponding formulas for the exactly solvable models become more complicated or, in some cases, are unknown. In addition, full-space models enjoy a translation invariance that is key to many proofs. However, in the case of boundaries, spatial shifts change the distance to the boundary,
and the same translation-invariance is not present (see Remark \ref{rmk:arb_shift}). 

Despite these challenges, half-space last-passage percolation with geometric weights admits an interpretation as a Pfaffian Schur process \cite{psp,dy25b}, which in turn is a Pfaffian analogue of the determinantal Schur processes introduced in \cite{dsp}. Earlier works on one-point distributions and fluctuations in geometric half-space LPP and related models can be found in \cite{Baik-Rains-2001a,Baik-Rains-2001b,Baik-Rains-2001c,Rains-2000,Sasamoto-Imamura-2004}. This interpretation as a Pfaffian Schur process provides access to exact formulas and yields a description of the model as a discrete Gibbsian line ensemble. This perspective has been exploited in several works; in particular, scaling limits of the model have been established in various regimes \cite{Zhou25,dy25,Dmitrov-Zhou-2025,ddy26}. Very recently,  the half-space directed landscape was constructed in \cite{Dauvergne-Zhang-2026} using an alternative approach developed in \cite{Dauvergne-Zhang-2024}. A key tool in this work was the use of jointly invariant measures, discussed in the following section.

\subsection{Invariant measures}
For many exactly solvable models in the KPZ universality class, their invariant measures can be described explicitly. In the full-space settings, the invariant measures are of product form (or are Brownian motions with drift in the case of continuum models). In the full-space setting, knowledge of the stationary measures has been used to derive fluctuation exponents for many of these models (see, for example, \cite{Cator-Groeneboom-2006,Balazs-Cator-Seppalainen-2006,Emrah-Janjigian-Seppalainen-2023,Emrah-Georgiou-Ortmann-2025,Landon-Sosoe-2023}), and have been used as key inputs in obtaining alternative characterizations of the universal limiting objects (\cite{Dauvergne-Virag-2024,Dauvergne-Zhang-2024}). 
Invariant measures for full-space particle systems, exclusion processes, and vertex models have also been studied extensively. Examples include the classification of invariant measures for the symmetric \cite{liggett1973characterization,liggett1974characterization} and asymmetric \cite{liggett1976coupling} simple exclusion processes, as well as for the stochastic six-vertex model and higher-spin vertex models \cite{aggarwal2019limit,lin2023classification}. 
Another object of key interest is that of \textit{jointly} invariant measures, which are couplings of invariant measures whose joint distribution is preserved under the evolution of the model with a common noise. For particle systems, these are the multi-type, or colored invariant measures. In  the full-space and periodic contexts, such measures have been studied for particle systems in \cite{Ferrari-Martin-2005,Ferrari-Martin-2007,Ferrari-Martin-2009,Martin-2020,Ayyer-Mandelshtam-Martin-2023,Ayyer-Mandelshtam-Martin-2024,Aggarwal-Nicoletti-Petrov-2025}, LPP models and the directed landscape in \cite{Fan-Seppalainen-20,Seppalainen-Sorensen-21b,Busani-2024,Busa-Sepp-Sore-22a,Bates-Emrah-Martin-Sepp-Sore-2025}, and for polymers and the KPZ equation in  \cite{Dunlap-Graham-2021,Dunlap-Gu-2024,Bates-Fan-Seppalainen-2025,GRASS-2025,Corwin-Gu-Sorensen}. The recent work \cite{Dauvergne-Zhang-2026} constructs the jointly invariant measures in several exactly-solvable half-space models, and uses this as their starting point for the construction of the  directed landscape in half-space. The techniques involving invariant measures give more probabilistic avenues to analyzing models in the KPZ universality class that rely less on exact formulas. Furthermore, some models that do not exhibit exact formulas have tractable invariant measures. For example, the work \cite{Landon-Sosoe-Noack-2023} identified an $O(N^{2/3})$ upper bound for the variance of a class of interacting diffusion processes in full-space, which do not give rise to explicit formulas, but which have product-form invariant measures.

In the half-space setting, families of invariant measures for geometric and exponential LPP, the inverse-gamma polymer, and the KPZ equation have been found in the works \cite{Barraquand-LD-2022,Barraquand-Corwin-22}.
For the asymmetric simple exclusion process (ASEP) in half-space, a family of invariant measures was introduced in~\cite{liggett1975ergodic} and later analyzed in several works; see, for example, \cite{grosskinsky2004phase,sasamoto2012combinatorics,bryc2017asymmetric,yang2024limits}. However, to the best of our knowledge, a complete classification of all invariant measures remains open.  
The joint invariant measures recently discovered in \cite{Dauvergne-Zhang-2026} are a generalization of these measures to multiple functions. In general, the invariant measures are not translation-invariant, but they do have an explicit description in terms of a queuing-like map of independent random walks (see Definition \ref{def:mucs} below). In the works \cite{Barraquand-Corwin-22,Dauvergne-Zhang-2026}, it was conjectured that the set of measures found comprises all extremal invariant measures for these models (see also \cite{Zeng-2025}). The main result of this paper is to prove this conjecture for geometric LPP (Theorem \ref{thm:main_thm} below for the invariant measures and Theorem \ref{thm:joint_invm_class} for the jointly invariant measures). 

\subsection{History of the one force--one solution principle}
In the present paper, we not only classify the invariant measures, but also describe the basins of attraction of the invariant measures started from arbitrary initial conditions with a given asymptotic slope. Each invariant measure is supported on functions having a fixed asymptotic slope, which is a conserved quantity for the system. Specifically, we show in Theorem \ref{thm:1F1S} that if we start the evolution at a time in the distant past, then the recentered solution at finite times converges to a stationary eternal solution, and the limiting solution depends only on the asymptotic slope of the initial condition. This is known as the \textit{one force--one solution principle}.

For interacting particle systems in full space, convergence from broad
classes of initial states to invariant measures has been established in a
number of settings; see, for example,
\cite{liggett1975ergodic,liggett1977ergodic,andjel1988shocks,
bahadoran2006convergence}. In half-space geometries, the corresponding
picture is less complete. For the half-line ASEP under the boundary-rate
relation now often called Liggett's condition,
\cite[Theorem~1.8]{liggett1975ergodic} established convergence from product initial laws with an asymptotic density at infinity:
the process converges to an invariant measure (Bernoulli or non-Bernoulli)
depending on the boundary parameters and on the density at infinity. 

For the Burgers equation in a periodic setting, one force--one solution principles were obtained in \cite{Sinai1991,Sinai1996,Iturriaga-Khanin-2003,EKMS-1997,EKMS-2000,Iturriaga-Khanin-2003,GIKP-2005,TR22}. An ergodicity result for the KPZ equation with open boundaries was obtained in \cite{parekh2023ergodicityresultsopenkpz}. Each of these results relied heavily on the compactness of the model. Significant breakthroughs were made in non-compact settings, including the works \cite{Bakhtin-2013,Bakhtin-Cator-Konstantin-2014,bakhtin2019thermodynamic,Dunlap-Graham-2021}. Results for the full-space KPZ equation and KPZ fixed point were obtained in \cite{Janj-Rass-Sepp-22,Busa-Sepp-Sore-22a}. 

Our result is the first one-force--one solution principle in the half-space setting that describes the asymptotic behavior of solutions for the full phase diagram of slopes and boundary conditions. This half-space setting does not enjoy the advantage of a compact domain nor the translation-invariance that is present in the full-space setting (continuum full-space models also enjoy a shear invariance property which is crucial to many of the works mentioned above).

The general idea in the present work is to employ the coupling technique, developed for queuing systems in \cite{Mountford-Prabhakar-1995,Prabhakar-2003} and adapted to other growth models in the KPZ universality class in \cite{Cator-Groeneboom-2006,Balazs-Cator-Seppalainen-2006,Seppalainen-Valko-2010,Seppalainen-2012,Cator-Lopez-Pimentel-2019,Pimentel-2021,Gu-Tao-2026}. Here, we compare the solution started from a non-stationary initial condition with general slope $\theta$ to two stationary initial conditions: one of these initial conditions has slope $\theta + \ve$, and the other has slope $\theta - \ve$. Knowledge of the shape function and a paths-crossing argument allows us to bound the recentered solution from the generic initial condition above and below by the two stationary solutions. For all but one of the possible slopes of invariant measures, this method works similarly to the full-space case. However, in the half-space setting, there is a minimal allowable slope for the family of invariant measures. When started from an initial condition with slope less than or equal to this minimal slope, we can bound the solution from above by a stationary solution with slightly larger slope, but there is no analogous lower bound. Instead, we get a lower bound by using a dirac initial condition at the diagonal boundary. This lower bound is not stationary, but using estimates and exact formulas for the half-space geometric line ensemble, we show that this lower bound converges to the appropriate stationary measure (Proposition \ref{prop:recentered_conv}).  The method of proof is described in more detail in Section \ref{sec:methods}.
We expect that our techniques can be extended to prove classifications of invariant measures and one force--one solution principles for several other integrable half-space models. We plan to pursue this for the half-space KPZ equation in future work.

\subsection{Preliminaries and definitions} Fix two parameters $q,r$ such that $0 < q < 1$ and $0 \le c < q^{-1}$. Define 
\[
\ZHS = \{(i,j) \in \Z^2: i \ge j\}.
\]
Let $(\omega_{(i,j)})_{(i,j) \in \ZHS}$ be a family of nonnegative random variables.   For $(i,j) \le (m,n)$ (under the coordinate-wise partial ordering), where $(i,j),(m,n) \in \ZHS$, we define the \textbf{passage time} as 
\be \label{eq:G_def}
G\bigl((i,j),(m,n)\bigr) = \max_{\pi:(i,j) \to (m,n)} \sum_{\mbf x \in \pi} \omega_{\mbf x},
\ee
where the sum is taken over all up-right lattice paths $\pi = \Bigl(\mbf x_0,\mbf x_1,\mbf x_2,\ldots,\mbf x_k\Bigr)$ such that
\[
\mbf x_0 = (i,j),\quad\mbf x_k = (m,n), \quad \mbf x_i \in \ZHS
 \text{ for } 0 \le i \le k,\text{ and } \mbf x_i - \mbf x_{i-1} \in \{(0,1),(1,0)\} \text{ for }1 \le i \le k.
 \]
We call this model half-space last-passage percolation (LPP). A  \textbf{geodesic} is a (possibly non-unique) maximizing path. See Figure \ref{fig:half-space_geod}.
    \begin{figure}
        \centering
      \begin{tikzpicture}[scale=0.8, thick]

\def\n{6}

\begin{scope}
  \clip
    (0,0) --
    (\n,\n) --
    (\n,0) --
    (0,0);

  \draw[gray!40] (0,0) grid (\n,\n);
\end{scope}

\draw[line width=1.2pt] (0,0) -- (\n,\n);

\draw[blue, line width=2pt]
  (3,1) -- (4,1) -- (4,4) -- (6,4);

\foreach \p in {(1,0),(2,0),(3,0),(4,0),(5,0),(6,0),(2,1),(3,1),(3,2),(4,1),(4,2),(4,3),(5,1),(5,2),(5,3),(5,4),(6,1),(6,2),(6,3),(6,4),(6,5)} {
  \fill[gray] \p circle (2.2pt);
}

\foreach \p in {(0,0),(1,1),(2,2),(3,3),(4,4),(5,5),(6,6)} {\fill \p circle (2.5pt);}

\node[below left] at (3,1) {$\mathbf x$};
\node[right] at (6,4) {$\mathbf y$};
\end{tikzpicture}

        \caption{\small An up-right path in half-space. Paths are confined to the right of the diagonal line. In the case of geometric half-space LPP, The weights on the black  vertices along the diagonal have the $\Geo(cq)$ distribution, while the weights on the gray vertices in the bulk have the $\Geo(q^2)$ distribution.}
        \label{fig:half-space_geod}
    \end{figure}
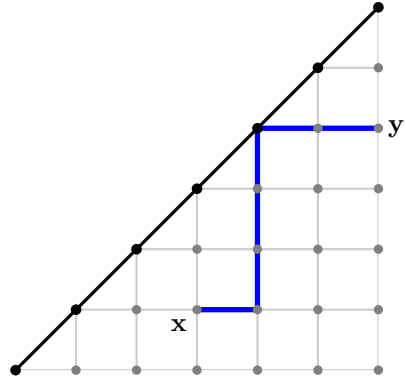
    It follows immediately that $G$ satisfies the following recursion:
    \be \label{eq:G_rec_1}
    G\bigl((i,j),(m,n)\bigr) = \omega_{(m,n)} + \begin{cases}
        G\bigl((i,j),(m-1,n)\bigr) \vee G\bigl((i,j),(m,n-1)\bigr) &m > n > j \\
        G\bigl((i,j),(m,n-1)\bigr) &m = n > j.
    \end{cases}
    \ee
    Here and below, we use $\vee$ to denote the maximum.
For $j \in \Z$ and a function $f:\Z_{\ge j} \to \R$, we define the passage time with initial condition $f$ at time $j - 1$: for integers $m \ge n \ge j$, 
\be \label{eq:G_bdy}
G_{f,j}(m,n) =  \max_{i \in \llbracket j,m \rrbracket}\Bigl[f(i) + G\bigl((i,j),(m,n)\bigr)\Bigr],
\ee
and we define $G_{f,j}(m,j-1) = f(m)$ for $m \ge j$, while $G_{f,j}(j-1,j-1) = 0$. Here, $\llbracket j,m \rrbracket$ denotes the set of integers $\{j,j+1,\ldots,m\}$.

If $j = 1$, we simply write $G_{f,1} = G_f$. We can think of $G_{f,j}(m,n)$ as the passage time from $(j-1,j-1)$ to $(m,n)$, with different weights along the initial horizontal boundary at level $j-1$ as follows: the weight at $(j-1,j-1)$ is $0$ and the weight at $(i,j-1)$ for $i \ge j$ is given by $f(i) - f(i-1)$.  
We could allow for other choices of weight at the vertex $(j-1,j-1)$, but in the present article, we will ultimately be interested in increments of the process $G_{f,j}$, making this initial weight irrelevant.  Here, we think of the vertical coordinate $n$ as time and the horizontal coordinate $m$ as space. We will often refer to `the weights along time level $n$', which is simply the vector of weights $(\omega_{(i,n)})_{i \in \Z_{\ge n}}$.

The focus of this article is the set of invariant measures for this process, which we define now. 

\begin{definition} \label{def:invariant}
    A probability measure $\nu$ on the space of functions $f:\N \to \R$ is called \textbf{invariant} if, whenever $f \sim \nu$, independent of the weights on or above level $1$: $(\omega_{(i,j)})_{i \ge j \ge 1}$, then for all $n \ge 1$, 
    \[
    \Bigl(G_f(n + k,n) - G_f(n,n)\Bigr)_{k \in \N} \sim \nu.
    \]
Here, we take the convention that $\N =\{1,2,\ldots\}$. 
\end{definition}
\begin{remark}
    From the definition, invariance only holds up to the global height shift at $(n,n)$. Since the weights $\omega$ are nonnegative, the uncentered quantity $G_f(n+k,n)$ goes to $\infty$ as $n \to \infty$ (provided the weights have positive mean). Thus, we expect to have no invariant measures without this recentering. Both the uncentered and recentered processes are Markov chains (Proposition \ref{prop:Markov} below). 
\end{remark}
\begin{remark}
    The state space we consider for invariant measures is
\be \label{eq:R0inf}
\R^\N := \{(f(1),f(2),\ldots): f(i) \in \R \;\;\forall i \in \N \},
\ee
together with the product topology. In Definition \ref{def:invariant}, the left-hand side also makes sense for $k = 0$, which gives the value $0$. With this in mind, to ease some of our notation, for any function $f:\N \to \R$, we take the convention that $f(0) = 0$, unless specified otherwise. 
\end{remark}

\subsubsection{Specialization to geometric weights}
Some preliminary results in this paper (specifically, those in Section \ref{sec:gen_LPP}) are stated for a general class of nonnegative weights. However, our main theorems all concern the special case where the weights have the geometric distribution. There are two parameters for this model, namely $q \in (0,1)$ and $c \in [0,q^{-1})$. For each fixed value of $q$, the model undergoes a phase transition at $c = 1$, and the phases are qualitatively the same for each value of $q$ (See Remark \ref{rmk:rc_param} below). In light of this, we shall fix an arbitrary value of $q \in (0,1)$ and make the following definition.  
\begin{definition}\label{def:GLPPc}
    For $c \in [0,q^{-1})$, we define the $\GLPPc$ model to be half-space last-passage percolation where $(\omega_{\mbf x})_{\mbf x \in \ZHS}$ are mutually independent weights with the following distributions:
\begin{itemize}
    \item For $n \in \Z$, $\omega_{(n,n)} \sim \Geo(cq)$. These are the weights along the boundary in Figure \ref{fig:half-space_geod}.
    \item For integers $m > n$, $\omega_{(m,n)} \sim \Geo(q^2)$. These are the bulk weights in Figure \ref{fig:half-space_geod}. 
\end{itemize}
\end{definition}

\begin{remark}
We take the following convention for geometric weights: For $\alpha \in (0,1)$, we say that $X \sim \Geo(\alpha)$ if $P(X = k) = \alpha^k(1-\alpha)$ for $k \in \Z_{\ge 0}$. In the two degenerate cases, $X \sim \Geo(0)$ means $X = 0$ a.s., and $X \sim \Geo(1)$ means $X = \infty$ a.s.
\end{remark}

 \begin{remark} \label{rmk:rc_param}
     Given a parameter $q \in (0,1)$, each boundary parameter $c \in [0,q^{-1})$ gives rise to a different family of invariant measures. For this reason, we shall say that a measure $\nu$ is invariant for $\GLPPc$ if it is an invariant measure with the choice of weights $\omega$ described above. Of course, this model also depends on the parameter $q$, but we will treat this as a fixed parameter throughout the paper. We will also make use of the parameter
     \[
     r_c := c \vee 1.
     \]
     This definition is motivated by a phase transition for the model that occurs at $c = 1$. When $c \le 1$, we say that the boundary is weak, while we say that the boundary is strong when $c > 1$. In the literature, the cases $c < 1$, $c = 1$, and $c > 1$ are sometimes called subcritical, critical, and supercritical regimes, respectively.
 \end{remark}

\begin{definition}
    We say that $S:\N \to \R$ is a $\Geo(\alpha)$ random walk if the increments $(S(k) - S(k-1))_{k \in \N}$ are i.i.d. $\Geo(\alpha)$ random variables. 
\end{definition}
The following set of measures  was shown to be invariant for $\GLPPc$ in \cite[Proposition 3.2]{Barraquand-Corwin-22} (see also \cite{Betea-Ferrari-Occelli-2020}). To aid the exposition, we shall use a slightly different (but equivalent) formulation than that used in \cite{Barraquand-Corwin-22}. We clarify this point later in Appendix \ref{appx:match}.
\begin{definition} \label{def:mucs}
We define the following family of measures.
 For $c \in [0,q^{-1})$ and $r_c \le s < q^{-1}$, let $\mu_{c,s}$ denote the law of the following function $f:\N \to \R$: 
\[
f(k) = S_2(k) + \Bigl(\max_{\ell \in \llbracket 1,k \rrbracket} [S_1(\ell) - S_2(\ell - 1)] - Y\Bigr)^+,
 \]
where $S_1$ is a $\Geo(qs)$ random walk, $S_2$ is a $\Geo\bigl(qs^{-1}\bigr)$ random walk, and $Y \sim \Geo\bigl(cs^{-1}\bigr)$ (a single geometric random variable). Here, we use the notation $x^+$ to denote $x \vee 0$. 
\end{definition}
\begin{remark}
    If $c \ge 1$ and $s = c$, then $Y \sim \Geo(1)$, so $Y = \infty$, and $\mu_{c,c}$ is simply the law of $S_2$: a $\Geo(qc^{-1})$ random walk. If $c < 1$ and $q < s < \min(c^{-1},q^{-1})$, then the above law $\mu_{c,s}$ still makes sense and was shown to be an invariant measure in \cite{Barraquand-Corwin-22}. However, it can be shown using symmetries of Schur processes  that in this case, $\mu_{c,s} = \mu_{c,s^{-1}}$ (see, for example, \cite[Lemma 4.4]{Barraquand-Corwin-22} for a positive-temperature version of this fact). Hence, we do not lose any measures by restricting to $s \ge r_c$. Our results give an alternative proof of this symmetry. In particular, the one force--one solution principle stated in Theorem \ref{thm:1F1S} below implies that any two invariant measures supported on functions with the same asymptotic slope must be the same.  
\end{remark}

\subsection{Classification of invariant measures}
Our first main theorem is the classification of the invariant measures for $\GLPPc$. This result  was conjectured in the context of the inverse-gamma polymer and KPZ equation in \cite{Barraquand-Corwin-22}.  In the result below, an \textit{extremal invariant measure} is an extreme point of the convex set of invariant measures.

\begin{theorem} \label{thm:main_thm}
    For any boundary parameter $c \in [0,q^{-1})$, the set $(\mu_{c,s})_{r_c \le s < q^{-1}}$ comprises the set of all extremal invariant measures for $\GLPPc$.  
\end{theorem}

Theorem \ref{thm:main_thm} is proved in Section \ref{sec:proof_complete}. A key point to the proof is that the invariant measures $\mu_{c,s}$ are each supported on functions having a fixed asymptotic slope. The slopes are given by the function $T_c:[r_c,q^{-1}) \to \R$, which is defined by
 \be  \label{eq:Tfunc}
T_c(s) = \begin{cases}
\f{q}{c-q},& s = c \ge 1 \\
\f{qs}{1-qs}, &s > r_c\quad \text{or}\quad c < 1 \text{ and } s \ge 1.
\end{cases}
\ee 
Later, we also use the function $X_c:\R \to \R$ defined by 
\be \label{eq:K_func}
X_c(\theta) = \begin{cases}
1 &\theta < \f{qr_c}{1 - qr_c}, \\
 \Bigl(\f{q}{\theta(1-q^2) - q^2}\Bigr)^2, &\theta \ge \f{qr_c}{1 - qr_c}.
\end{cases}
\ee

The following lemma is proved at the end of Section \ref{sec:invmeas_slopes}. 
\begin{lemma} \label{lem:inv_meas_slopes}
For $s \in [r_c,q^{-1})$, we have 
\[
\mu_{c,s} \Bigl\{f: \lim_{k \to \infty} \f{f(k)}{k} = T_c(s)\Bigr\} = 1.
\]
\end{lemma}

One of the key inputs to the proof of  Theorem \ref{thm:main_thm} is Proposition \ref{prop:slopes_exist}, which states that for any extremal invariant measure $\nu$, either there exists $\theta > \f{qr}{1-qr}$ so that $\nu$-almost surely,
\be \label{eq:drift_op_1}
\lim_{k \to \infty} \f{f(k)}{k} = \theta,
\ee
or else, $\nu$-almost surely,
\be \label{eq:drift_op_2}
\limsup_{k \to \infty} \f{f(k)}{k} \le \f{qr_c}{1-qr_c}.
\ee
The proof of Theorem \ref{thm:main_thm} comes by combining this fact with our next main result, Theorem \ref{thm:1F1S}, which demonstrates a one force--one solution principle for half-space geometric LPP. 
\subsection{One force--one solution principle}
 The one force--one solution principle states that, when starting with an initial condition in the distant past, the recentered solution converges almost surely to a process which is distributed as one of the measures $\mu_{c,s}$. The object that the process converges to is called the Busemann process. This process is a collection of Busemann functions, indexed by a parameter $\xi$, which is related to the parameters $s$ and $\theta$ via the functions $T_c$ and $X_c$.  The parameter $\xi$ occurs naturally as follows: for a fixed $\xi$, the Busemann function with parameter $\xi$ can be constructed as the following almost sure limit (see Proposition \ref{prop:Buse_limits_1}):
    \[
    W_\xi^c(\mbf x,\mbf y) = \lim_{n \to \infty} G\bigl((-\lfloor \xi n \rfloor,-n),\mbf y) - G\bigl((-\lfloor \xi n \rfloor,-n),\mbf x\bigr).
    \] 

The parameters $\xi, s$ and $\theta$, are related by bijections $T_c$ and $X_c$ defined in \eqref{eq:Tfunc} and \eqref{eq:K_func}. 
The Busemann process, constructed rigorously in Proposition \ref{prop:full_Busemann}, is denoted as
\be \label{Wxi_and_xi_max}
\Bigl(W_{\xi}^c(\mbf x,\mbf y):\mbf x,\mbf y \in \ZHS, \xi \in (0,\xi_{\max}^c]\Bigr),
\quad\text{where}\quad 
\xi_{\max}^c  = X_c\Bigl(\f{qr_c}{1-qr_c}\Bigr) =  \begin{cases}
    1 &c \le 1 \\
    \Bigl(\f{1 - qc}{c-q}\Bigr)< 1 &c < 1.
\end{cases}
\ee
If $0 < \xi < \xi_{\max}^c$, then Proposition \ref{prop:full_Busemann}\ref{it:Buse_dist} states that for each $t \in \Z$, the process 
    $\bigl(W^c_\xi((t,t),(t+k ,t))\bigr)_{k \in \N}$ has law $\mu_{c,s}$, where $s = (X_c \circ T_c)^{-1}(\xi)$. For $\xi = \xi_{\max}^c$, this same process has distribution $\mu_{c,r_c}$.

\begin{theorem} \label{thm:1F1S}
    Let $c \in [0,q^{-1})$, and $f$ be an initial condition $f:\N \to \R$, independent of the weights $(\omega_{\mbf x})_{\mbf x \in \ZHS}$. For such an initial condition and $n \in \Z$, define the function $f_n:\Z_{\ge -n} \to \R$ by $f_n(k) = f(k+n+1)$.
    \begin{enumerate} [label=(\roman*), font=\normalfont]
        \item \label{it:1F1Spt1} If, for some $\theta  > \f{qr_c}{1-qr_c}$, we have
        \[
        \lim_{k \to \infty} \f{f(k)}{k} = \theta,\quad\text{almost surely},
        \]
        then, for each $t \in \Z$ and $k \in \N$,
        \be \label{eq:G_inc}
        \lim_{n \to \infty} [G_{f_n,-n}(t+k,t) - G_{f_n,-n}(t,t)] = W_{\xi}^c\bigl((t,t),(t+k,t)\bigr), \quad\text{almost surely,}
        \ee
        where $\xi = X(\theta)$.
        \item \label{it:1F1Spt2} If we have 
        \[
        \limsup_{k \to \infty} \f{f(k)}{k} \le \f{qr_c}{1-qr_c},\quad\text{almost surely},
        \]
        then for each $k \in \N$,
        \[
        \lim_{n \to \infty}[G_{f_n,-n}(t+k,t) - G_{f_n,-n}(t,t)] = W_{\xi_{\max}^c}^c\bigl((t,t),(t+k,t)\bigr),\quad\text{in probability}.
        \]
    \end{enumerate}
\end{theorem}
Theorem \ref{thm:1F1S} is proved in Section \ref{sec:1F1S proof}.
\begin{remark} \label{rmk:conv_tomu}
    If $f$ is independent of the weights $(\omega_{\mbf x})_{\mbf x \in \ZHS}$, then we can show (see Corollary \ref{cor:shift_to_fn}) that
    \[
    \Bigl(G_{f_{n-1},-n+1}(k,0) - G_{f_{n-1},-n+1}(0,0)\Bigr)_{k \in \N} \deq \Bigl(G_{f}(n+k,n) - G_{f}(n,n)\Bigr)_{k \in \N}.
    \]
    Hence, Theorem \ref{thm:1F1S} along with the discussion of the law of the Busemann function implies that, when started from such an initial condition, the recentered solution
    \[
    \Bigl(G_{f}(n+k,n) - G_{f}(n,n)\Bigr)_{k \in \N}
    \]
    converges in distribution to $\mu_{c,s}$ for some $s \in [r_c,q^{-1})$.
\end{remark}
\begin{remark}
    The body of the paper contains stronger statements than Theorem \ref{thm:1F1S} which are more technical to state.  See Propositions \ref{prop:full_Busemann}\ref{it:Buse_attractive}-\ref{it:boundary_attract} and  \ref{prop:to_Buse_in_prob} for precise statements.  In particular, one can take the initial condition $f$ to depend on $n$ as well. One also gets a joint convergence statement across all initial conditions and all possible slopes simultaneously. However, the Busemann process has some potentially random  discontinuities in the parameter $\xi$, in which case we do not necessarily get convergence but can bound the $\liminf$ and $\limsup$ of the sequence. 
\end{remark}

\subsubsection{Phase diagram}
Figure \ref{fig:phase_diagram} gives a phase diagram for $s$ and $c$ that describes the set of invariant measures. The parameter $s$ relates to the slope of the initial condition in relation to the once force--one solution principle in Theorem \ref{thm:1F1S}. If, for some $s > r_c$, we start with an initial condition $f$ satisfying
\[
\lim_{k \to \infty} \f{f(k)}{k} = T_c(s),
\]
 then since $T_c(s) > T_c(r)$ (Lemma \ref{lem:TX_maps}\ref{it:T_fact}), the recentered solution converges in distribution to $\mu_{c,s}$. On the other hand, if
\[
\limsup_{k \to \infty} \f{f(k)}{k} \le \f{qr_c}{1-qr_c},
\]
then the recentered solution converges in distribution to $\mu_{c,1}$ (when $c < 1$) or $\mu_{c,c}$ (when $c \ge 1$). Thus, all invariant measures live in the gray shaded region of the phase diagram (including the dark boundary). The diagonal line where $c = s > 1$ is called the coexistence line. For each $c \in [0,q^{-1})$, the measures $\mu_{c,s}$ are weakly continuous in $s$ (Lemma \ref{lem:mu_continuous}), but there is a discontinuity in the slope of the invariant measure along the coexistence line.  

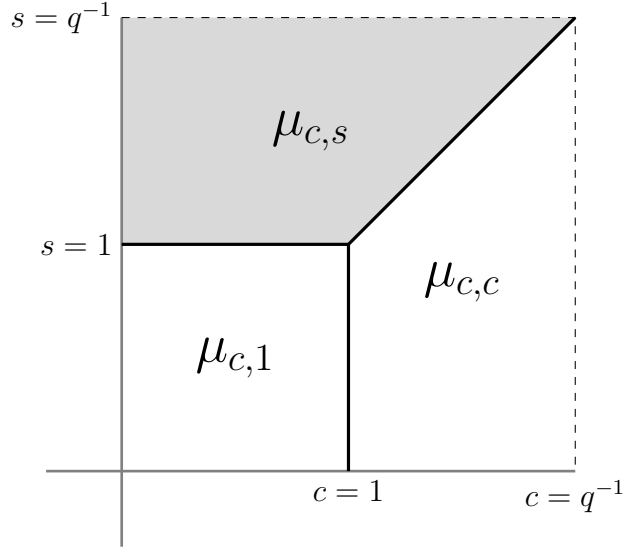
\begin{figure}
    \centering
          \begin{tikzpicture}
\fill[gray!30]
  (0,3) -- (3,3) -- (6,6) --(0,6) -- cycle;

\draw[line width=1pt,color = gray] (-1,0) -- (6,0);

\draw[line width = 1pt,color =gray] (0,-1)--(0,6);

\draw[dashed] (0,6)--(6,6)--(6,0);

\draw[line width = 1.2 pt] (3,0)--(3,3)--(0,3);

\draw[line width = 1.2 pt] (3,3)--(6,6);

\node[left] at (0,6) {\Large{$s = q^{-1}$}};
\node[below] at (6,0) {\Large{$c = q^{-1}$}};
\node[below] at (3,0) {\Large{$c = 1$}};
\node[left] at (0,3) {\Large{$s = 1$}};
\node at (1.5,1.5) {\Huge{$\mu_{c,1}$}};
\node at (2.5,4.5) {\Huge{$\mu_{c,s}$}};
\node at (4.5,2.5) {\Huge{$\mu_{c,c}$}};
\end{tikzpicture}
    \caption{ Phase diagram for the collection of invariant measures in half-space geometric LPP}
    \label{fig:phase_diagram}
\end{figure}

\subsection{Jointly invariant measures} \label{sec:joint_invar}
Theorem \ref{thm:main_thm} is a special case of a more general theorem about the classification of jointly invariant measures. We start with the following definition. 

\begin{definition} \label{def:joint_invariant}
   For $m \in \N$, we say a measure $\nu$ on the space $(\R^\N)^m$ is an $m$-jointly invariant measure if, when $(f_1,\ldots,f_m) \sim \nu$, for all $n \in \N$, we have
   \[
   \Bigl(G_{f_1}(n + k,n) - G_{f_1}(n,n),\ldots, G_{f_m}(n+k,n) - G_{f_m}(n,n)\Bigr)_{ k \in \N} \sim \nu.
   \]
\end{definition}
One readily sees that any $m$-jointly invariant measure is necessarily a coupling of $m$ invariant measures. Very recently, a family of jointly invariant measures was found in \cite{Dauvergne-Zhang-2026}. These measures are described in terms of increments of a last-passage problem in an inhomogeneous environment. 
Specifically, we  choose parameters $q^{-1} > s_1 > s_2 > \cdots > s_m \ge r_c$. Then, let $\ve \in(0,1)$ be  sufficiently small so that $\f{1 - \ve}{s_i} \in (0,1)$ for $i \in \llbracket 1,m \rrbracket$, and extend the vector $(s_i)_{i \in \llbracket 1,m \rrbracket}$ to $(s_i)_{i \in \N}$ by setting
\[
s_{2m + 1-i} = \f{1 - \ve}{s_i}\quad \text{for } i \in \llbracket 1,m \rrbracket,\quad\text{and}\quad s_i = q \quad\text{for } i  \in \Z_{\ge 2m+1}.
\]
Define an environment of independent random variables $(Y^\ve_{(i,j)})_{(i,j) \in \ZHS}$ as follows: for $i \in \N$, let $Y^\ve_{(i,i)} \sim \Geo(qs_i)$, and for integers $i > j$, set $Y^\ve_{(i,j)} \sim \Geo(s_is_j)$. Let $\mathcal G^\ve$ denote half-space last-passage percolation where the variables $\omega$ in \eqref{eq:G_def} are replaced by $Y^\ve$. We use the notation $\mathcal G^\ve$ to distinguish from $G$ because the invariant measure will be described by certain increments of $\mathcal G^\ve$, which then become the random initial conditions for $G$.
\begin{definition} \label{def:mucs_extend}
    For $c \in [0,q^{-1})$ and $q^{-1} > s_1 > s_2 > \cdots > s_m \ge r_c$, let $\mu_{c,s_1,\ldots,s_m}$ denote the law of $(f_1,\ldots,f_m)$, where  
    \be \label{eq:fi_def}
    f_i(k) = \lim_{\ve \searrow 0}[ \mathcal G^\ve\bigl((2m + 1 - i,i),(2m + k,2m)\bigr) - \mathcal G^\ve\bigl((2m + 1 - i,i),(2m,2m)\bigr)], \quad i \in \llbracket 1,m \rrbracket, k \in \N.
    \ee
    We extend the definition of $\mu_{c,s_1,\ldots,s_m}$ for general $(s_i)_{i \in \llbracket 1,m \rrbracket} \subseteq [r_c,q^{-1})$ as follows: If the $s_i$ are all distinct, let $\sigma$ be the permutation of $\llbracket 1,m \rrbracket$ satisfying $s_{\sigma(1)} > s_{\sigma(2)} > \cdots > s_{\sigma(m)}$. Then, define 
    \[\mu_{c,s_1,\ldots,s_m} := \mu_{c,s_{\sigma(1)},\ldots,s_{\sigma(m)}} 
    \circ \sigma^{-1},
    \]
    where $\sigma$ acts on $(f_i)_{i \in \llbracket 1,m \rrbracket}$ as $\sigma (f_i)_{i \in \llbracket 1,m \rrbracket} = (f_{\sigma_i})_{i \in \llbracket 1,m \rrbracket}$.
    If the $s_i$ are not all distinct, then let $(s_{\ell_1},\ldots,s_{\ell_p})$ be a subsequence of maximal length such that the $s_{\ell_i}$ are all distinct. Then, let $\mu_{c,s_1,\ldots,s_m}$ be the law of $(f_1,\ldots,f_m)$, where $(f_{\ell_1},\ldots,f_{\ell_p}) \sim \mu_{c,s_{\ell_1},\ldots,s_{\ell_p}}$, and $f_i = f_j$ whenever $s_i = s_j$. 
\end{definition}
In
\cite[Lemma 2.16]{Dauvergne-Zhang-2026}, it is shown that $\mu_{c,s_1,\ldots,s_m}$ is a jointly invariant measure for any $q^{-1} > s_1 > \cdots > s_m \ge r_c$. Then,  the extension to general $(s_i)_{i \in \llbracket 1,m \rrbracket}$ follows immediately.
In the above  definition, there is a $\Geo(1 - \ve)$ weight at each of the vertices $(2m + 1-i,i)$, which diverges as $\ve \searrow 0$. However, since we are taking a difference in \eqref{eq:fi_def}, this term is canceled out, and the limit still exists. In the case $k = 1$, we show in Appendix \ref{appx:match} that the description of the law in Definition \ref{def:mucs_extend} matches that of $\mu_{c,s}$ from Definition \ref{def:mucs}. In fact, this definition in terms of LPP increments is how the law is introduced in \cite[Proposition 3.2]{Barraquand-Corwin-22}.

The Busemann process is a stationary eternal solution (See Proposition \ref{prop:full_Busemann}\ref{it:Buse_full_eternal}), meaning that it is a bi-infinite stationary process whose time evolution is the same as that of half-space LPP. From this, we get that, for $\xi_1,\ldots,\xi_m \in (0,\xi_{\max}^c]$, the law of 
    \[
    \Bigl(W^c_{\xi_1}\bigl((0,0),(k,0)\bigr),\ldots,W^c_{\xi_m}\bigl((0,0),(k,0)\bigr)\Bigr)_{k \in \N}
    \]
    is an $m$-jointly invariant measure.  We now state the following generalization of Theorem \ref{thm:main_thm}, which is proved in Section \ref{sec:proof_complete}.
\begin{theorem} \label{thm:joint_invm_class}
     Let $c \in [0,q^{-1})$. If $m \in \N$, and $\nu$ is an extremal $m$-invariant measure for $\GLPPc$, then there exist $s_1,\ldots,s_m \in [r_c,q^{-1})$ such that 
     \[
      \nu = \mu_{c,s_1,\ldots,s_m}.
     \]
     Furthermore, for any $\xi_1,\ldots,\xi_m \in (0,\xi_{\max}^c]$,
\[
\Bigl(W^c_{\xi_1}\bigl((0,0),(k,0)\bigr),\ldots,W^c_{\xi_m}\bigl((0,0),(k,0)\bigr)\Bigr)_{k \in \N} \sim \mu_{c,s_1,\ldots,s_m},
\]
where, for $i \in \llbracket 1,m \rrbracket$,
\[
s_i = \begin{cases}
    (X_c \circ T_c)^{-1}(\xi_i) &\xi_i < \xi_{\max}^c \\
    r_c &\xi_i = \xi_{\max}^c.
\end{cases}
\]
\end{theorem}

\begin{remark}
    In the theorem, above, we allow the case where $\xi_i = \xi_j$ for some $i \neq j$. As a function of $\xi$, the process $\xi \mapsto W^c_{\xi}$ is left-continuous with right limits, and each fixed $\xi \in (0,\xi_{\max}^c]$ is a continuity point with probability one (see Proposition \ref{prop:full_Busemann}).  
\end{remark}

\subsection{Directions of semi-infinite geodesics} 
Along the way to proving Theorem \ref{thm:main_thm}, we obtain the following result about the allowable directions of semi-infinite geodesics in geometric LPP. A semi-infinite geodesic is a backwards infinite path $\bigl(\gamma(0),\gamma(-1),\ldots\bigr)$ in $\ZHS$ such any finite restriction of the path is a geodesic. There are two phases with qualitatively different behavior: when $c \le 1$ and the boundary is weak, all possible directions of semi-infinite geodesics in half-space are present. When $c > 1$ and the boundary is strong, there is a gap in the set of allowable directions: there are semi-infinite geodesics near the boundary ($\xi = 1$), then a gap in the allowable directions. This says that infinite geodesics cannot get too close to the boundary without following along the boundary themselves. This was conjectured to be the case in \cite[Remark 2.15]{Dauvergne-Zhang-2026}. This case corresponds to a region of the shape function (between the directions $\xi_{\max}$ and $1$) that is linear (see \eqref{eq:mu_k_def} for the exact description of the shape function). In the full-space setting, it is well-known for a general class of weight distributions (see, for example \cite{Damron-Hanson-2014,Georgiou-Rassoul-Seppalainen-2017a,Emrah-Janjigian-Sepplainen-2025,Groathouse-Janjigian-Rassoul-2025}) that for linear facets of the shape function, semi-infinite geodesics are directed into the facet, without necessarily having a single direction. In our setting, construction of the Busemann process and the result in Proposition \ref{prop:recentered_conv} allow us to show that all semi-infinite geodesics do, in fact, have a direction.

\begin{theorem} \label{thm:geod_direction}
	Let $c \in [0,q^{-1})$. Then, on a single event of probability one, whenever $\gamma$ is any semi-infinite geodesic for $\GLPPc$,  the limit
	\[
	\lim_{k \to \infty} \f{\gamma(-k)\cdot \mbf e_1}{\gamma(-k) \cdot \mbf e_2}
	\]
	exists and lies in the set $[0,\xi_{\max}^c) \cup \{1\}$. Here, $\mbf e_1 = (1,0)$ and $\mbf e_2 = (0,1)$ are the standard basis vectors.

	In particular, if $c \le 1$, then the set of allowable limits is $[0,1]$. If $c > 1$, then the set of allowable limits is  
	\[
	\biggl[0, \Bigl(\f{1-qc}{c-q}\Bigr)^2\biggl) \cup \{1\}.
	\]
    Furthermore, for each $\xi$ in this set and for every lattice point $\mbf x \in \ZHS$, there exists at least one semi-infinite geodesic from $\mbf x$ in that direction. 
\end{theorem}
\begin{remark} \label{rmk:slope_order}
	The quantity $\xi$ represents the inverse asymptotic slope of the geodesics. Thus, a larger value of $\xi$ means that the geodesic is farther to the left as $k \to \infty$. Of course, no semi-infinite geodesic in half-space LPP can have direction greater than $1$ due to the presence of the boundary. 
\end{remark}

Theorem \ref{thm:geod_direction} is proved in Section \ref{sec:directions}.
One can ask many additional interesting questions about semi-infinite geodesics in half-space LPP, such as coalescence, non-existence of bi-infinite geodesics, and exceptional directions with multiple families of semi-infinite geodesics. We leave these questions to be explored in future work.

\subsection{Methods of proof and related literature} \label{sec:methods}
In this section, we describe the necessary inputs to the paper and outline the structure of the proof. Our starting point is the description of the invariant measures from \cite{Barraquand-Corwin-22,Dauvergne-Zhang-2026}. The key fact we need is that these measures are invariant and that they are supported on functions with a given asymptotic slope (Lemma \ref{lem:inv_meas_slopes}). From here, the proof can then be summarized in the following steps:
\begin{itemize}
    \item Establish general facts about half-space LPP and establish necessary symmetries and path-crossing arguments to be used for coupling (Section \ref{sec:gen_LPP}).
    \item Prove a sufficiently strong shape theorem for $\GLPPc$ and use that to control the directions of semi-infinite geodesics from eternal solutions (Sections \ref{sec:shape}-\ref{sec:exit_pt_control}).
    \item Construct the Busemann process and use it along with control of geodesic directions to prove the one force--one solution principle (Section \ref{sec:Busemann}).
    \item Use the machinery of the half-space geometric line ensembles and two-layer Gibbs measures for Schur processes to prove convergence to the stationary measure at the boundary of the phase diagram (Section \ref{sec:line_ensemble}). 
\end{itemize}
The proofs of the main results are then completed in Section \ref{sec:proofs}. For the remainder of this section, we give a sketch of these proofs and discuss some related literature.

\subsubsection{Coupling and sketch of the one force--one solution principle}

As mentioned above, the proof of Theorem \ref{thm:main_thm} (more generally, Theorem \ref{thm:joint_invm_class}) comes in two steps. One is the one force--one solution principle in Theorem \ref{thm:1F1S}, which shows that, when started from an initial condition satisfying one of the conditions \eqref{eq:drift_op_1} or \eqref{eq:drift_op_2}, the recentered solution converges in distribution to $\mu_{c,s}$ for some $s$. We also show in Proposition \ref{prop:slopes_exist} that one of these conditions holds almost surely for any extremal initial condition.  In this section, we will describe the broad ideas in the proof of these two results. 

First, we mention that there are some common ideas that are heavily used in the proofs of both of these items. One of these is what is often called the paths-crossing argument. It is stated as Lemma \ref{lem:exit_pt_comp_ptl} in the present paper and has been used many times in the full-space context. To understand the paths-crossing argument, let $f_1,f_2$ be two initial conditions at time $j-1$. Assume, for some $n \in \Z$ and $k \in \N$, that the maximizer in \eqref{eq:G_bdy} (call it $i_1$) for $m = n+k$ and $f= f_1$ is less than the maximizer for $m = n$ (call it $i_2$). We call the locations the \textit{exit points} from the boundary. Since $i_1 < i_2$, planarity ensures that the geodesic from $(i_1,j)$ to $(k+n,n)$ must cross the geodesic from $(i_2,j)$ to $(n,n)$. One can then exploit this fact (see the proof of Lemma \ref{lem:exit_pt_comp_ptl}) to obtain the inequality
\be \label{eq:G_mont}
G_{f_1,j}(n+k,n) - G_{f_1,j}(n,n) \le G_{f_2,j}(n+k,n) - G_{f_2,j}(n,n).
\ee

When $n-j$ is large, we can also describe the locations of the maximizers from a broad class of initial conditions. To do this, we prove a general shape theorem for half-space geometric LPP in Section \ref{sec:shape} (specifically Theorem \ref{thm:LLN}). To get almost sure limits in the shape theorem, we cannot appeal to the subadditive ergodic theorem. Our techniques and proof to overcome this issue are described in Section \ref{sec:shape}. 

Given the shape function $ \rho_{c}$, if $\lim_{k \to \infty} \f{f(k)}{k} = \theta$, then after a change of variables, the maximizer in \eqref{eq:G_bdy} is approximately
\[
n\Bigl( \argmax_{\xi \in [0,1]}[\theta \xi + \rho_{c}(1-\xi)] - 1\Bigr).
\]
This is proved rigorously in Proposition \ref{prop:converge_to_max}. Lemma \ref{lem:location_of_max} describes the set of maximizers of this function.

Both ingredients in the proof also use the construction of the Busemann process (Proposition \ref{prop:full_Busemann}). This is a collection of bi-infinite stationary solutions for the model (called eternal solutions in the present paper), indexed by a parameter $\xi$, which is in bijection with the set of possible values of $s$ in the known set of invariant measures. Busemann functions, originally used as a tool in geometric topology \cite{Busemann-1955} have been a key tool in random growth models, starting with the works \cite{Newman-1995,Licea-Newman-1996,Hoffman-2005,Hoffman-2008} and used also in, for example, \cite{Cator-Pimentel-2013,Damron-Hanson-2014,geor-rass-sepp-yilm-15,Alberts-Rassoul-Simper-2020,Georgiou-Rassoul-Seppalainen-2017b,Janjigian-Rassoul-20, Seppalainen-Sorensen-2023}.

 To construct the Busemann process, one starts with a countable dense set of values of $\xi$ and considers the increments of the function $G_{f,-n}$ for large $n$ (the model started in the distant past). The recentered solution at each fixed time has the same law by invariance, so the process is tight as $n \to \infty$, and we can extract a subsequential limit. The spatial increments of the Busemann process are then distributed as the known invariant measures, and these invariant measures are supported on functions having a specified asymptotic slope (Lemma \ref{lem:inv_meas_slopes}).

With the Busemann process constructed, we are able to prove the two key ingredients for the proof of the main theorem. The two cases of the one force--one solution principle are handled separately. First, in the case of Item \ref{it:1F1Spt1}, we compare the quantity on the left-hand side of \eqref{eq:G_inc}  to the Busemann functions with parameters $\xi - \ve$ and $\xi + \ve$ for small $\ve > 0$. The paths-crossing monotonicity in \eqref{eq:G_mont}, along with the knowledge of the location of the maximizers from the slopes gives us 
\be \label{eq:W_comp_intro}
\begin{aligned}
W_{\xi +\ve}\bigl((t,t),(t+k,t)\bigr) &\le \liminf_{n \to \infty} [G_{f_n,-n}(t+k,t) - G_{f_n,-n}(t,t)]  \\
&\le \limsup_{n \to \infty} [G_{f_n,-n}(t+k,t) - G_{f_n,-n}(t,t)] \le  W_{\xi - \ve}\bigl((t,t),(t+k,t)\bigr).
\end{aligned}
\ee
The limit in \eqref{eq:G_inc} then comes from sending $\ve \searrow 0$ and using continuity (Proposition \ref{prop:full_Busemann}\ref{it:cts_fixed_xi}).

While the details in half-space require a substantial amount of additional care, the overall coupling technique described here has been used in the full-space context in many works. The key innovation comes in proving Item \ref{it:1F1Spt2}. The proof for the upper bound in \eqref{eq:W_comp_intro} remains the same as in the previous case. However, there is no matching lower bound because there is no invariant measure corresponding to the parameter $\xi - \ve$ (this corresponds to the boundary of the shaded region in the phase diagram of Figure \ref{fig:phase_diagram}). Instead, by using the paths-crossing argument, we get the lower bound
\be \label{eq:G_rec_intro}
G\bigl((-n,-n),(t,t+k)\bigr) - G\bigl((-n,-n),(t,t)\bigr).
\ee
This quantity corresponds to an initial condition whose geodesic is forced to exit the boundary at level $-n - 1$ at the leftmost possible point. As a function of $k$, this process is not stationary, but we show in Proposition \ref{prop:recentered_conv} that as $n \to \infty$, this process converges weakly to $\mu_{c,r_c}$. The proof of this proposition  requires machinery from half-space line ensembles, which is discussed more in the next subsection.

Next, we prove that extremal invariant measures are supported on functions with an asymptotic slope. We let $\nu$ be an arbitrary extremal invariant measure, and from it, we can construct an eternal solution whose increments at any given time level are $\nu$. To any eternal solution, one can construct a family of semi-infinite geodesics--one going downward from each lattice point. In the language of Hamilton-Jacobi equations, these are infinite backwards maximizers. By comparison with the semi-infinite geodesics for the known invariant measures, we show that each of the semi-infinite geodesics has an asymptotic direction. Using the paths-crossing lemma and comparison with the Busemann process, we can show that the asymptotic direction of the semi-infinite geodesics is related to the asymptotic slope of the semi-infinite geodesics. Planarity requires that, as we move to the right along a given time level, the direction of the semi-infinite geodesics is increasing. Then, in Lemma \ref{lem:geodesic_directions_stop}, we show that, as we send the initial point to $\infty$ along the spatial axis, the direction of the semi-infinite geodesic stabilizes. To prove this, we adapt a technique developed for the full-space KPZ fixed point in \cite{Dunlap-Sorensen-25} to the half-space discrete setting. The key idea is that the directions of the semi-infinite geodesics at time level $t$ affect the possible directions of the semi-infinite geodesics at times $s < t$ by planarity. The invariance of the measure $\nu$ then requires these directions to stabilize. In Lemma \ref{lem:lim_exist_1}, we show that this implies that, with probability one, one of the slope conditions in Theorem \ref{thm:1F1S} must hold, although the slope may a. priori be random. We then show that asymptotic slopes are conserved quantities for half-space geometric LPP (Lemma \ref{lem:slopes_pres}), and the assumed extremality of $\nu$ then implies that the measure is supported on exactly one of these slopes.

\subsubsection{Line ensembles} In this subsection, we describe the proof of convergence in law of the recentered process \eqref{eq:G_rec_intro} (Proposition \ref{prop:recentered_conv}). Part of our proof follows a strategy similar to that used in the convergence-to-stationarity result for the half-space inverse-gamma polymer obtained by the first author with Serio in \cite{ds25} using its line ensemble structure \cite{Barraquand-Corwin-Das-2023}. However, there are key differences in the techniques, which we highlight below.

The proof of Proposition \ref{prop:recentered_conv} utilizes the half-space geometric line ensemble structure of the geometric last passage percolation model, which has recently been studied extensively in \cite{Zhou25,Dmitrov-Zhou-2025,dy25,ddy26}. The last passage times $\bigl(G((1,1),(n+k,n)\bigr)_{k\ge 1}$ can be realized as the top curve of a Gibbsian line ensemble $(L_i^N(\cdot))_{i\in \llbracket 1,N\rrbracket}$ consisting of reverse geometric random walks conditioned to interlace and subject to a pairwise pinning interaction at the left boundary. This remarkable structure arises from the RSK correspondence \cite{knuth1970permutations,greene1974extension} together with the Pfaffian Schur process framework \cite{psp,BBCS-2018}; see the recent paper \cite{dy25b} for a short proof.

When $c\in [0,1)$, the pinning occurs between the $(2i-1)$-th and $(2i)$-th curves for $i\ge 1$. When $c=1$, the pinning is absent. When $c>1$, the pinning occurs between the $(2i)$-th and $(2i+1)$-th curves for $i\ge 1$. Thus, from the line ensemble point of view, a curve separation is expected to occur between the second and third curves when $c\in [0,1)$, and between the first and second curves when $c\in [1,q^{-1})$. Given this, the broad steps in the proof of convergence to the stationary measure can be described as follows:

\begin{enumerate}
    \item (Curve separation) When $c\in [0,1)$, the top two curves separate from the remaining curves in the limit, and when $c\in [1,q^{-1})$, the top curve separates from the remaining curves in the limit;
   \item (Local convergence) When $c\in [0,1)$, upon resampling the top two curves using the underlying half-space Gibbs property of the line ensemble, independently of the lower curves, the resulting Gibbs measure converges to the stationary measure in the limit. When $c\in [1,q^{-1})$, upon resampling only the top curve using the underlying half-space Gibbs property of the line ensemble, independently of the lower curves, the resulting Gibbs measure converges to the stationary measure in the limit.
\end{enumerate}

Our proof of Item (1) follows a similar strategy to \cite{ds25} and utilizes recent scaling limit results for the half-space geometric line ensemble proved by the first and third authors together with Dimitrov in \cite{ddy26}. For Item (2), the case $c\in [1,q^{-1})$ is straightforward because the limiting invariant measure is a geometric random walk, which is easily seen to be the limit of the top curve. The case $c\in [0,1)$ is the main challenge. The paper \cite{ds25} proves a general local convergence result, but its techniques do not lead to an explicit description of the limiting law.

Our proof of this local convergence result follows a different approach based on a new idea. Namely, we observe that the Gibbs law of the top two curves is naturally related to the two-layer Gibbs measure introduced by Barraquand, Corwin, and the third author in \cite{Barraquand-Corwin-Yang-2023}. This two-layer Gibbs measure describes the stationary measure for geometric LPP on a strip with two boundary parameters \(c_1\) and \(c_2\). In our setting, the resampling law can be interpreted as replacing the right boundary weight \(c_2^{d}\) in the two-layer Gibbs measure by the indicator \({\bf 1}_{d=k}\), for fixed \(k \in \mathbb{Z}_{\ge 0}\), where \(d\) denotes the gap between the two curves.

The work \cite[Section 2.6]{barraquand2024integral} developed a Markovian description of the two-layer Gibbs measure together with explicit integral formulas for certain functions associated with it. After suitable modifications, analogous Markovian descriptions and formulas can also be derived for the Gibbs law arising in our setting. An asymptotic analysis of these formulas then yields the desired convergence result.

Our framework should naturally extend to positive temperature models. As mentioned above, for the half-space inverse-gamma polymer model, the curve separation strategy has already been developed in \cite{ds25} and \cite{dz24} in the subcritical and supercritical regimes using its line ensemble structure \cite{Barraquand-Corwin-Das-2023}. A similar framework should also be applicable to the intermediate disorder scaling limit, namely the half-space KPZ equation, via its line ensemble structure \cite{ds25b}. For the local convergence, the aforementioned strategy could likewise be applied to these polymer models. In particular, one could potentially exploit the Markovian description and explicit formulas for the Whittaker two-layer Gibbs measures developed in \cite[Section 3.5]{barraquand2024integral}. We leave this for future work.

\subsection{Acknowedgements}
E.S. was partially supported by the Fernholz foundation and by Ivan Corwin’s Simons Investigator Grant \#929852.  
Z.Y. was partially supported by Ivan Corwin's NSF grants DMS:1811143, DMS:2246576, Simons Foundation Grant \#929852, and the Fernholz Foundation's `Summer Minerva Fellows' program.
We thank Guillaume Barraquand, Ivan Corwin, Evgeni Dimitrov, Elnur Emrah, Dominik Schmid, Eyob Tsegaye, and Lingfu Zhang for helpful discussions related to this project.

\section{Preliminary results for general weights} \label{sec:gen_LPP}
This section proves several important algebraic preliminaries that will be needed later on. The results in this section do not require the weights to have the geometric distribution. Instead, we assume the following:
\begin{itemize}
    \item The weights are all nonnegative.
    \item The weights in the bulk, namely $(\omega_{(i,j)})_{i > j}$ are i.i.d. 
    \item The weights along the boundary $(\omega_{(i,i)})_{i \in \Z}$ are i.i.d. and are independent of the weights in the bulk. 
\end{itemize}
 A key feature of this assumption is that the law of the weights is invariant under shifts of the form $\mbf x \mapsto \mbf x + (n,n)$.

\subsection{Dynamic programming principle} 
In what follows, we will refer several times to the dynamic programming principle for half-space LPP, stated in the following lemma. 
\begin{lemma} \label{lem:dynam_prog}
Let $j \le t \le n \le m$ be integers, and let $f:\Z_{\ge j} \to \R$. Then,
\[
G_{f,j}(m,n) = \max_{\ell \in \llbracket t,m \rrbracket} \bigl[G_{f,j}(\ell,t-1) +G\bigl((\ell,t),(m,n)\bigr)\bigr].
\]
\end{lemma}
\begin{proof}
    Since we defined $G_{f,j}(m,j-1) = f(m)$ for $m \ge j$, the case $t = j$ is simply the definition \eqref{eq:G_bdy}. For $j \ge t+1$, we observe that 
    \begin{align*}
        G_{f,j}(m,n) &= \max_{i \in \llbracket j,m \rrbracket}\Bigl[f(i) + G\bigl((i,j),(m,n)\bigr)\Bigr] \\
        &= \max_{i \in \llbracket j,m \rrbracket}\Biggl[f(i) + \max_{\ell \in \llbracket i \vee t, m\rrbracket}\Bigl[G\bigl((i,j),(\ell,t-1)\bigr) + G\bigl((\ell,t),(m,n)\bigr)\Bigr]\Biggr] \\
        &= \max_{\ell \in \llbracket t,m \rrbracket} \Biggl[\max_{i \in \llbracket j,\ell \rrbracket}\Bigl[f(i) + G\bigl((i,j),(\ell,t-1)\bigr)\Bigr] +G\bigl((\ell,t),(m,n)\bigr) \Biggr] \\
        &= \max_{\ell \in \llbracket t,m \rrbracket}\Bigl[G_{f,j}(\ell,t-1) + G\bigl((\ell,t),(m,n)\bigr)\Bigr]. \qedhere
    \end{align*}
\end{proof}

\subsection{Markov property} \label{sec:state_space}

\begin{proposition} \label{prop:Markov}
Let $m \in \N$, and let $f_1,\ldots,f_m:\Z_{\ge j}$ be a coupling of initial conditions, independent of the weights at and above level $j$ : $(\omega_{k,\ell})_{k \ge \ell \ge j}$ (not necessarily independent among the $f_i$). 
Then, the process
\be \label{eq:MC1}
n \mapsto \Bigl(G_{f_i,j}(n + k,n)\Bigr)_{i \in \llbracket 1,m \rrbracket, k \in \Z_{\ge 0}}, \quad n \in \Z_{\ge j-1}
\ee
is a Markov chain on the state space $(\R^{\Z_{\ge 0}})^m$. Furthermore,
the recentered process 
\be \label{eq:MC2}
n \mapsto \Bigl(G_{f_i,j}(n + k,n) - G_{f_i,j}(n,n) \Bigr)_{i \in \llbracket 1,m \rrbracket, k \in \N}, \quad n \in \Z_{\ge j}
\ee
is a Markov chain on the state space $(\R^\N)^m$.
\end{proposition}
\begin{proof}
From the dynamic programming principle (Lemma \ref{lem:dynam_prog}), we have that, for integers $k \ge 0$, $n \ge j$, and $i \in \llbracket 1,m \rrbracket$,
\be
\begin{aligned}
G_{f_i,j}(n+k,n) &=  \max_{u \in \llbracket 1, k+1 \rrbracket}\Biggl[G_{f_i,j}(n-1 + u,n-1)+ G\bigl((n-1 + u,n),(n+k,n)\bigr)\Biggr]  \\
&= \max_{u \in \llbracket 1, k+1 \rrbracket}\Biggl[G_{f_i,j}(n-1 + u,n-1)+ \sum_{\ell = n-1 + u}^{n+k}\omega_{\ell,n}\Biggr].
\end{aligned}
\ee
Hence, we see that $\bigl(G_{f_i,j}(n+k,n)\bigr)_{i \in \llbracket 1,m \rrbracket, k \in \Z_{\ge 0}}$ is a measurable function of $\bigl(G_{f_i,j}(n-1+u,n-1)\bigr)_{i \in \llbracket 1,m \rrbracket, u \in \N}$ and the weights $(\omega_{r,n})_{r \ge n}$, which are independent of the state of the chain up to time $n-1$. This proves the Markov property for the chain \eqref{eq:MC1}.

For \eqref{eq:MC2}, we have 
\begin{align*}
    &\quad \, G_{f_i,j}(n+k,n) - G_{f_i,j}(n,n)  \\
    &= \max_{u \in \llbracket 1, k+1 \rrbracket}\Biggl[G_{f_i,j}(n-1 + u,n-1)+ \sum_{\ell = n-1 + u}^{n+k}\omega_{\ell,n}\Biggr] - \Biggl(G_{f_i,j}(n,n-1)+ \sum_{\ell = n}^{n+k}\omega_{\ell,n}\Biggr) \\
    &= \max_{u \in \llbracket 1, k+1 \rrbracket}\Biggl[G_{f_i,j}(n-1 + u,n-1) - G_{f_i,j}(n-1,n-1)+ \sum_{\ell = n-1 + u}^{n+k}\omega_{\ell,n}\Biggr]  \\
    &\qquad\qquad\qquad - \Biggl(G_{f_i,j}(n,n-1) - G_{f_i,j}(n-1,n-1)+ \sum_{\ell = n}^{n+k}\omega_{\ell,n}\Biggr),
\end{align*}
and the Markov property follows just as for \eqref{eq:MC1}. 
\end{proof}

\subsection{Symmetries and comparisons}
\begin{lemma} \label{lem:shift_sym}
For any $n \in \Z$, we have 
\[
\Biggl(G\bigl(\mbf x + (n,n), \mbf y + (n,n)\bigr):\mbf x \le \mbf y, \mbf x,\mbf y \in \ZHS\Biggr) \deq \Biggl(G\bigl(\mbf x, \mbf y\bigr):\mbf x \le \mbf y, \mbf x,\mbf y \in \ZHS\Biggr)
\]
\end{lemma}
\begin{proof}
    This follows immediately by shifting the weights $\omega_{\mbf x}$ to $\omega_{\mbf x + (n,n)}$ for each $\mbf x \in \ZHS$. 
\end{proof}
\begin{remark} \label{rmk:arb_shift}
    In full-space last-passage percolation, we can replace the shift $(n,n)$ by an arbitrary shift $(m,n)$ and have the same equality in distribution. This does not hold in the half-space setting because of the presence of the boundary. 
\end{remark}

\begin{corollary} \label{cor:shift_to_fn}
    Let $(f_1,\ldots,f_m) \in (\R^\N)^m$ be any coupling of initial conditions, jointly independent of the weights $(\omega_{\mbf x})_{\mbf x \in \ZHS}$ \rm{(}but $(f_1,\ldots,f_m)$ are not necessarily independent\rm{)}. For $n \in \Z$ and $i \in \llbracket 1,m \rrbracket$, define $f_{n,i}:\Z_{ \ge -n} \to \R$ by
    \[
    f_{n,i}(k) = f_i(n + 1+k),\quad\text{for }k \in \Z_{\ge -n}.
    \]
    Then, for $n \in \Z$ and $t \in \Z_{> -n}$, we have 
    \[
    \bigl(G_{f_{n-1,i},-n+1}(t+k,t) \bigr)_{i \in \llbracket 1,m \rrbracket, k \in \Z_{\ge 0}} \deq  \bigl(G_{f,i}(t+n+k,t+n)\bigr)_{i \in \llbracket 1,m \rrbracket, k \in \Z_{\ge 0}}
    \]
\end{corollary}
\begin{proof}
    We have 
    \begin{align*}
        G_{f_{n-1,i},-n + 1}(t+k,t) &= \max_{j \in \llbracket -n + 1, t+k \rrbracket}[f_{n-1,i}(j) + G\bigl((j,-n+1),(t+k,t)\bigr)]\\
        &= \max_{j \in \llbracket 1, t+n +k \rrbracket} [f_i(j) + G\bigl((j -n,-n+1),(t+k,t)\bigr)] \\
        &\deq \max_{j \in \llbracket 1, t+ n +k \rrbracket} [f_i(j) + G\bigl((j,1),(t+n+k,t+n)\bigr)] = G_{f}(t+n+k,t+n),
    \end{align*}
    where the distributional equality follows by Lemma \ref{lem:shift_sym} and holds jointly over $k \in \Z_{\ge 0}$ and $i \in \llbracket 1,m \rrbracket$ since it only involves a shift of $G$, which is independent of $(f_1,\ldots,f_m)$. 
\end{proof}

While we don't have the general translation invariance in half-space LPP, the following key symmetry does hold and is crucial to the proof of the shape theorem in the following section.  
\begin{lemma} \label{lem:refl_sym}
    For any integers $0 \le j \le i \le n$, we have  
    \[
    G\bigl((i,j),(n,n)\bigr) \deq G\bigl((0,0),(n-j,n-i)\bigr).
    \]
\end{lemma}
\begin{proof}
    Given the weights $\omega$, define a new set of weights that are equal in distribution,  $\bigl(\omega_{(m,\ell)}'\bigr)_{0 \le \ell \le m \le n}$, by reflecting the weights across the line $y = n - x$. That is, 
    \[
    \omega_{(m,\ell)}' = \omega_{(n - \ell,n - m)}.
    \]
    For any up-right path $\pi$ in $\ZHS$ connecting $(i,j)$ and $(n,n)$, we define its image under this reflection as follows: denote $\pi = \bigl((i_0,j_0),(i_1,j_1),\ldots,(i_k,j_k)\bigr)$, where $(i_0,j_0) = (i,j)$ and $(i_k,j_k) = (n,n)$. Then, we define the reflected path as 
    \[
    \pi' = \bigl((n - j_k,n - i_k),(n - j_{k-1},n - i_{k-1}),\ldots,(n - j_0,n - i_0)\bigr).
    \]
    We see that the path $\pi'$ starts at $(0,0)$ and ends at $(n-j,i-j)$, remains in the set $\ZHS$, and its steps lie in the set $\{(0,1),(1,0)\}$. See Figure \ref{fig:reflection}.
    \begin{figure}
        \centering
      \begin{tikzpicture}[scale=0.8, thick]

\def\n{6}

\begin{scope}
  \clip
    (0,0) --
    (\n,\n) --
    (\n,0) --
    (0,0);

  \draw[gray!40] (0,0) grid (\n,\n);
\end{scope}

\draw[line width=1.2pt] (0,0) -- (\n,\n);

\draw[orange, line width=1pt]
  (0,0) -- (2,0) -- (2,2) -- (5,2)--(5,3);

\draw[blue, line width=2pt]
  (3,1) -- (4,1) -- (4,4) -- (6,4)--(6,6);

\draw[black, dotted, line width=2pt]
  (3,3) -- (6,0);

\foreach \p in {(0,0),(3,1),(5,3),(6,6)} {
  \fill \p circle (2.2pt);
}

\node[below left] at (0,0) {$(0,0)$};
\node[below] at (3,1) {$(m,\ell)$};
\node[above right] at (5,3) {$(n-\ell,n-m)$};
\node[above] at (6,6) {$(n,n)$};

\end{tikzpicture}

        \caption{\small A path from $(m,\ell)$ to $(n,n)$ (blue, thick) and its image (orange, thin) across the dotted line $y = n-x$.}
        \label{fig:reflection}
    \end{figure}
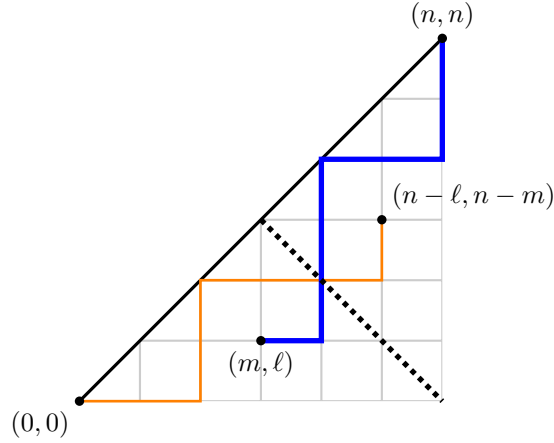
    This defines a bijection between the set of paths $(i,j) \to (n,n)$ and the set of paths $(0,0) \to (n - j,n - i)$; we reflect again to obtain the inverse. Furthermore, for each such path $\pi$, we have 
    \[
        \sum_{\mbf x \in \pi} \omega_{\mbf x} = \sum_{\mbf x \in \pi'} \omega_{\mbf x}',    
    \]
    Taking the maximum over all paths $\pi$ and all paths $\pi'$, the proof is complete since the weights $\omega$ and $\omega'$ are equal in distribution. 
\end{proof}

We shall also make use of the following monotonicity, which follows immediately from the fact that the weights $\omega_{\mbf x}$ are all nonnegative. 
\begin{lemma} \label{lem:G_monotonicity}
    For $(i,j) \le (m,n)$ in $\ZHS$, if $a,b,c,d \in \Z_{\ge 0}$ are such that $(i-a,j-b),(m+c,n + d) \in \ZHS$, then
    \[
    G\bigl((i,j),(m,n)\bigr) \le G\bigl((i-a,j-b),(m +c,n + d)\bigr).
    \]
\end{lemma}

\subsection{Eternal solutions}
We say that a function $b:\ZHS \to \R$ is an \textbf{eternal solution} if, for all $m,n \in \Z$,
\be \label{eq:b_recur}
b(m,n) = \begin{cases} \omega_{(m,n)} + b(m,n-1) \vee b(m-1,n),&m > n \\
\omega_{(n,n)} + b(n,n-1), &m = n .
\end{cases}
\ee
For any initial condition $f:\Z_{\ge j}$, the recursion for $G$ \eqref{eq:G_rec_1} the function $G_{f,j}$ satisfies this recursion for all $n \ge j$. This restriction is removed in the definition of an eternal solution. In other words, we say that the solution is eternal because it is bi-infinite in time. We obtain the following immediately from the definition of an eternal solution.
\begin{lemma} \label{lem:omega_upbd}
    If $b$ is an eternal solution, then for $(m,n) \in \ZHS$,
    \[
    \omega_{(m,n)} = \begin{cases}
    \bigl(b(m,n) - b(m-1,n)\bigr)
 \wedge  \bigl(b(m,n) - b(m,n-1)\bigr), &m > n \\
 b(m,n) - b(m-1,n), &m = n,\end{cases}
    \]
    where $\wedge$ denotes minimum.
\end{lemma}

\begin{lemma} \label{lem:eternal_construction}
    Let $\nu_1,\nu_2,\ldots,$ be a sequence of invariant measures for the corresponding half-space LPP problem on the space $\R^\N$. Then, on some probability space $(\wt \Omega,\wt{\F},\wt{\mathbb P})$, there exists weights $(\wt \omega_{\mbf x})_{\mbf x \in \ZHS}$, equal in distribution to the original weights, and a sequence of eternal solutions $b_1,b_2,\ldots,$ such that, for each $p \in \N$ and $n \in \Z$,
    \be \label{eq:b_law}
    \bigl(b_p(n+k,n) - b_p(n,n)\bigr)_{k \in \N} \sim \nu_p,
    \ee
    and the process $\bigl(b_p(i,j)\bigr)_{i\ge j, j \le n, p\in \N}$ is independent of $(\wt \omega_{i,j})_{i \ge j \ge n+1}$ \rm{(}the weights on levels $n + 1$ and above\rm{)}. 
\end{lemma}
\begin{proof}
    Let $f_1,f_2,\ldots$ be independent functions $\N \to \R$, with $f_p \sim \nu_p$ for each $p\ge 1$, independent of the entire field of weights $(\omega_{(i,j)})_{i \ge j}$. For $p,T \in \N$, define $f_p^T:\Z_{\ge -T} \to \R$ $f_p^T(k) = f_p(k+T+1)$, and define $(b_p^T(m,n))_{m \ge n \ge -T}$ by 
    \[
    b_p^T(m,n) := G_{f_p^T,-T}(m,n) - G_{f_p^T,-T}(0,0).
    \]
    Fix $s < 0$. By Corollary \ref{cor:shift_to_fn} and the assumed invariance of each measure $\nu_p$, we have, for each $p \in \N$ and $n \in \Z_{\ge -T}$ that 
\be \label{eq:bT_law}
\bigl(b_p^T(n+k,n) - b_p^T(n,n)\bigr)_{k \in \N} \deq \bigl(G_{f_p}(n+T+1 + k,n+T+1) - G_{f_p}(n+T+1,n+T+1)\bigr)_{k \in \N} \sim \nu_p,
\ee
and, by definition, we have that $\bigl(b^T_p(i,j)\bigr)_{i \ge j \ge -T, j \le n, p \in \N}$ is independent of $(\omega_{(i,j)})_{i \ge j\ge n+1}$ for each $n \in \Z_{\ge - T}$.
    
     By the dynamic programming principle (Lemma \ref{lem:dynam_prog}), for all $p \in \N$, $T \in \Z_{> - t}$, and integers $m\ge n \ge t$, 
    \begin{align*}
    b^T_p(m,n) &= \max_{r \in \llbracket t+1,m \rrbracket}\bigl[G_{f_p,-T}(r,t) + G\bigl((r,t+1),(m,n)\bigr)\bigr] - \max_{r \in \llbracket t+1,0 \rrbracket}\bigl[G_{f_p,-T}(r,t) + G\bigl((r,t+1),(m,n)\bigr)\bigr] \\
    &=\max_{r \in \llbracket t+1,m \rrbracket}\bigl[G_{f_p,-T}(r,t) - G_{f_p,-T}(t,t) + G\bigl((r,t+1),(m,n)\bigr)\bigr] \\&\qquad\qquad\qquad - \max_{r \in \llbracket t+1,0 \rrbracket}\bigl[G_{f_p,-T}(r,t) - G_{f_p,-T}(t,t) + G\bigl((r,t+1),(m,n)\bigr)\bigr] \\
    &= \max_{k \in \llbracket 1, m-t \rrbracket}\bigl[b_T^p(t+k,t) - b_T^p(t,t) + G\bigl((t+k,t+1),(m,n)\bigr)\bigr] \\
    &\qquad\qquad\qquad- \max_{k \in \llbracket 1, -t\rrbracket}\bigl[b^T_p(t+k,t) - b^T_p(t,t) + G\bigl((t+k,t+1),(m,n)\bigr)\bigr].
    \end{align*}
    Thus, by \eqref{eq:bT_law} with $n = t$, for each $t \in \Z_{< 0}$ and $p \in \N$ the law of $\bigl(b^T_p(m,n)\bigr)_{m \ge n \ge t}$ does not depend on $T$ (as long as $t > -T$) and is therefore tight as $T \to \infty$. Then also, the law of the coupled process across all $p$, namely $\bigl(b^T_p(m,n)\bigr)_{m \ge n \ge t, p \in \N}$ is tight. We can also couple this with the field of weights $\omega$, which do not depend on $T$ and is therefore jointly tight. 

    Additionally, each $b^T_p$ inherits the recursion for an eternal solution \eqref{eq:b_recur} from the recursion for $G$ \eqref{eq:G_rec_1}. The result then follows by sending $T \to \infty$ and choosing a subsequential limit. 
\end{proof}

\subsection{Semi-infinite geodesics}
A \textbf{semi-infinite geodesic } rooted at $\gamma(0) \in \ZHS$ is a backwards infinite path $(\ldots, \gamma(-2),\gamma(-1),\gamma(0))$ such that, for any $k \in \N$, the path $(\gamma(-k),\gamma(-(k - 1)),\ldots,\gamma(0))$ is a geodesic.

We prove an intermediate lemma.
\begin{lemma} \label{lem:G_leb}
    If $b:\ZHS \to \R$ is an eternal solution, and $(i,j) \le (m,n)$ in $\ZHS$, we have
    \be \label{eq:G_le_b_bd}
    G\bigl((i,j),(m,n)\bigr) \le \begin{cases}
    \bigl(b(m,n) - b(i-1,j)\bigr) \wedge \bigl(b(m,n) - b(i,j-1)\bigr) &i > j \\
    b(m,n) - b(i,i-1) &i = j.
    \end{cases}
    \ee
\end{lemma}
\begin{proof}
    Let $\pi = \bigl(\mbf x_0,\ldots,\mbf x_k\bigr)$ be any up-right path in $\ZHS$ from $(i,j)$ to $(m,n)$. Let $\mbf x_{-1}$ to be one of the points in the set $\{(i-1,j),(i,j-1)\}$ (choosing $(i,j-1)$ if $i = j$). From Lemma \ref{lem:omega_upbd}, we have that, for $\ell \in \llbracket 0,k \rrbracket$, 
    \[
    \omega_{\mbf x_\ell} \le b(\mbf x_\ell) - b(\mbf x_{\ell - 1}).
    \]
    Hence, 
    \[
    \sum_{\ell = 0}^k \omega_{\mbf x_\ell} \le \sum_{\ell = 0}^k \bigl(b(\mbf x_\ell) - b(\mbf x_{\ell - 1})\bigr) = b(\mbf x_k) -b(\mbf x_{-1}) = b(m,n) - b(\mbf x_{-1}).
    \]
    Since this holds for any path $\pi$ and either choice of $\mbf x_{-1} \in \{(i-1,j),(i,j-1)\}$, we obtain \eqref{eq:G_le_b_bd}, as desired.  
\end{proof}

Our next result in this section shows we can obtain semi-infinite geodesics from eternal solutions. This is the half-space analogue of the result in \cite[Lemma 4.1]{Georgiou-Rassoul-Seppalainen-2017a}.
\begin{lemma} \label{lem:hs_geodesics}
    Assume that $b$ is an eternal solution. For each $(m,n)$, define the following backwards infinite path $\gamma_{(m,n)} = (\gamma_{(m,n)}(-k): k \in \Z_{\ge 0})$: start at $\gamma_{(m,n)}(0) = (m,n)$, and define the path inductively as follows: if $\gamma_{(m,n)}(-k) = (i,j)$, then 
    \be \label{eq:gamma_path}
    \gamma_{(m,n)}(-(k+1)) = \begin{cases}
        (i,j-1),&i = j, \text{ or } i > j \text{ and }b(i,j-1) \ge b(i-1,j), \\
    (i-1,j), &i > j \text{ and } b(i-1,j) > b(i,j-1).
    \end{cases}
    \ee
    Then, the following hold:
    \begin{enumerate} [label=(\roman*), font=\normalfont]
        \item \label{it:geo}For each $(m,n) \in \ZHS$, the path $\gamma_{(m,n)}$ is a semi-infinite geodesic rooted at $(m,n)$.
         \item \label{it:Buse=G} For $p \in \Z_{\ge 0}$, define $(m_{-p},n_{-p}) := \gamma_{(m,n)}(-p)$. Then, for every $k \in \Z_{\ge 0}$,
        \[
        G\bigl((m_k,n_k),(m,n)\bigr) = b(m,n) - b(m_{k+1},n_{k+1}).
        \]
        \item \label{it:rightmost} For every $k \in \Z_{\ge 0}$, the path $\bigl(\gamma_{(m,n)}(-k),\gamma_{(m,n)}(-(k-1)),\ldots, \gamma_{(m,n)}(0)\bigr)$ is the rightmost geodesic between the points $\gamma_{(m,n)}(-k)$ and $(m,n)$. 
        \item\label{it:meet} If $(m,n),(m',n') \in \ZHS$ are such that $\gamma_{(m,n)}(-k) =\gamma_{(m',n')}(-k')$ for some $k,k' \in \Z_{\ge 0}$, then for every $\ell \in \N$,
        \[
        \gamma_{(m,n)}(-k - \ell) =  \gamma_{(m',n')}(-k' - \ell).
        \]
        That is, if the paths from two different points ever meet, they stay together \rm{(}when moving backwards\rm{)}.
    \end{enumerate}
\end{lemma}
\begin{remark}
    We have taken the convention to move downward when there is a tie between $b(i,j-1)$ and $b(i-1,j)$. One can also choose to move leftward and still obtain a semi-infinite geodesic. In this case, a similar proof would show that the path would be the leftmost geodesic between any two of its points. 
\end{remark}
\begin{proof}
    \textbf{Items \ref{it:geo}-\ref{it:Buse=G}:}  Observe that $(m_0,n_0) = (m,n)$. By Lemma \ref{lem:omega_upbd} and definition of the path $\gamma_{(m,n)}$ \eqref{eq:gamma_path}, we have
    \be \label{eq:b_gamma_eq}
        \omega_{(m_{-p},n_{-p})} = b(m_{-p},n_{-p}) - b(m_{-(p+1)},n_{-(p+1)}),\quad \text{for all } p \in \Z_{\ge 0}.
    \ee
    Then, for every $k \in \Z_{\ge 0}$,
    \[
    \sum_{p  =0}^k \omega_{(m_{-p},n_{-p})} = b(m,n) - b(m_{-(k+1)},n_{-(k+1)}).
    \]
    On the other hand, Lemma \ref{lem:G_leb} implies that 
    \[
    b(m,n) - b(m_{-(k+1)},n_{-(k+1)}) \ge G\bigl((m_{-k},n_{-k}),(m,n)\bigr)
    \]
    and thus
    \[
    G\bigl((m_{-k},n_{-k}),(m,n)\bigr) = \sum_{p = 0}^l\omega_{(m_{-p},n_{-p})} = b(m,n) - b(m_{-(k+1)},n_{-(k+1)}),
    \]
    and the path $\bigl(\gamma_{(m,n)}(-k),\ldots,\gamma_{(m,n)}(0)\bigr)$ is a geodesic between $(m_{-k},n_{-k})$ and $(m,n)$.

    \medskip \noindent \textbf{Item \ref{it:rightmost}:}
Assume, by way of contradiction that a different path $\pi = (\mbf x_{-k},\mbf x_{-(k-1)},\ldots,\mbf x_0)$ is the rightmost geodesic. Since $\pi$ is also a geodesic, Item \ref{it:Buse=G} implies that 
 \be \label{eq:x_passage_time}
\sum_{p = 0}^k \omega_{\mbf x_{-p}} =  b(m,n) - b(m_{-(k+1)},n_{-(k+1)}).
\ee
We will also define $\mbf x_{-(k+1)} = (m_{-(k+1)},n_{-(k+1)})$. From Lemma \ref{lem:omega_upbd}, we have $\omega_{\mbf x_{-p}} \le b(\mbf x_{-p}) - b(\mbf x_{-(p+1)})$ for $p \in \llbracket \ell, k \rrbracket$, so \eqref{eq:x_passage_time} implies that, in fact, 
\be \label{eq:omegaxp_eq}
\omega_{\mbf x_{-p}} = b(\mbf x_{-p}) - b(\mbf x_{-(p+1)}),\quad\text{ for } p \in \llbracket 0,k \rrbracket.
\ee

 Since the paths $\bigl(\gamma_{(m,n)}(-k),\ldots,\gamma_{(m,n)}(0)\bigr)$ and $(\mbf x_{-k},\ldots,\mbf x_0)$ terminate at the same point $(m_0,n_0) = (m,n)$,  there would exist a minimal index $p \in \Z_{\ge 0}$ such that $(m_{-p},n_{-p}) = \mbf x_{-p}$, while 
 \be \label{eq:x_gamma_switch}
 \mbf x_{-(p+1)} = (m_{-p},n_{-p} - 1)\quad\text{and}\quad \gamma_{(m,n)}(-(p+1)) = (m_{-p} - 1,n_{-p}).
 \ee
 That is, when going backwards from $(m_0,n_0)$, the path $\pi$ moves downwards from $(m_{-p},n_{-p})$ and the path $\gamma_{(m,n)}$ moves to the left from $(m_{-p},n_{-p})$.  See Figure \ref{fig:splitting_geod}.

     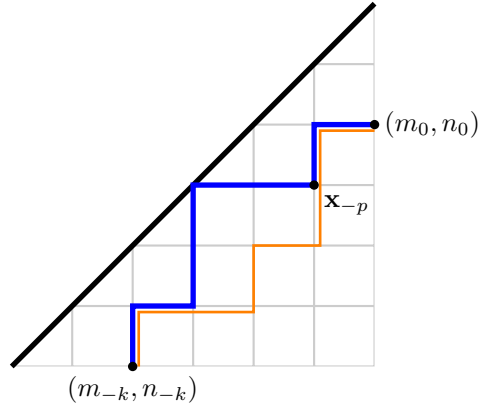
\begin{figure}
        \centering
      \begin{tikzpicture}[scale=0.8, thick]

\def\n{6}

\begin{scope}
  \clip
    (0,0) --
    (\n,\n) --
    (\n,0) --
    (0,0);

  \draw[gray!40] (0,0) grid (\n,\n);
\end{scope}

\draw[line width=2pt] (0,0) -- (\n,\n);

\draw[orange, line width=1pt]
  (2.1,0)--(2.1,0.9)-- (4,0.9) -- (4,2)--(5.1,2)-- (5.1,3.9)--(6,3.9);

\draw[blue, line width=2pt]
  (2,0)--(2,1)--(3,1)--(3,3) -- (5,3) -- (5,4) -- (6,4);

\foreach \p in {(2,0),(5,3),(6,4)} {
  \fill \p circle (2.2pt);
}

\node[below ] at (2,0) {$(m_{-k},n_{-k})$};
\node[below right] at (5,3) {$\mbf x_{-p}$};
\node[right] at (6,4) {$(m_0,n_0)$};

\end{tikzpicture}
\caption{\small The geodesics $\gamma$ (blue, thick) and $\pi$ (orange, thin) both terminate at $(m_0,n_0)$, and when moving downwards, there is a first place $\mbf x_{-p} = (m_{-p},n_{-p})$ where the two geodesics split.}
\label{fig:splitting_geod}
\end{figure}

 From \eqref{eq:omegaxp_eq}, we have 
    \[
    \omega_{(m_{-p},n_{-p})} = b(\mbf x_{-p}) - b(\mbf x_{-(p+1)}) = b(m_{-p},n_{-p}) -  b(m_{-p},n_{-p} -1),
    \]
    so Lemma \ref{lem:omega_upbd} implies that $b(m_{-p},n_{-p} - 1) \ge b(m_{-p}-1,n_{-p})$. 
Then, by definition of the path $\gamma_{(m,n)}$ \eqref{eq:gamma_path}, we have
    $
    \gamma_{(m,n)}(-(p+1)) = (m_{-p},n_{-p} - 1),$
    a contradiction to \eqref{eq:x_gamma_switch}. 

    \medskip \noindent \textbf{Item \ref{it:meet}:} This is immediate from the definition. 
\end{proof}

The following is a partial dual problem to Lemma \ref{lem:hs_geodesics}: there we construct semi-infinite geodesics from an eternal solution. In the following lemma, given a semi-infinite geodesic, we construct a backwards infinite solution $b(\mbf x)$ defined for all $\mbf x \le \gamma(0)$. This follows a similar procedure introduced for first-passage percolation by Hoffman \cite{Hoffman-2008}.

\begin{lemma} \label{lem:Buse_from_geod}
	Let $\gamma$ be any semi-infinite geodesic rooted at $\gamma(0) = (m,n) \in \ZHS$ and such  that $\gamma(-k)\cdot \mbf e_1 \to -\infty$ as $k \to \infty$. Then, for every $\mbf x \le \mbf \gamma(0)$ in $\ZHS$, the limit
	\be \label{eq:b_from_G}
	b(\mbf x) := \lim_{k \to \infty} G(\gamma(-k),\mbf x) - G(\gamma(-k),\gamma(0))
	\ee
	exists. Furthermore, for any $j \in \Z_{\le n}$,
	\be \label{eq:mxgbz}
	\begin{aligned}
		&b(\gamma(0)) = 0  = \max_{i \in \llbracket j,m \rrbracket} [b(i,j-1) + G\bigl((i,j),(m,n)\bigr)],  \quad\text{and}\\
		&i_j^\star := \max\{i \in \Z: \gamma(-k) = (i,j-1) \text{ for some } k \in \N \} \in \argmax_{i \in \llbracket j,m \rrbracket}[b(i,j-1) + G\bigl((i,j),(m,n)\bigr)].
	\end{aligned}
	\ee 
\end{lemma}
\begin{proof}
	To establish convergence, we show that the sequence is monotone and bounded. Let $1 \le \ell \le k$ be integers. Since the geodesic $\gamma$ is rooted at $\mbf z$, for all large enough $k$ so that $\gamma(-k) \le \mbf x$, we have
	\begin{align*}
		G(\gamma(-k),\gamma(0)) &= G(\gamma(-k),\gamma(-\ell)) + G(\gamma(-\ell),\gamma(0)) - \omega_{\gamma(-\ell)},\quad\text{and} \\
		G(\gamma(-k),\mbf x) &\ge G(\gamma(-k),\gamma(-\ell)) + G(\gamma(-\ell),\mbf x) - \omega_{\gamma(-\ell)}.
	\end{align*}
	Therefore,
	\begin{align*}
		G(\gamma(-k),\mbf x) - G(\gamma(-k),\gamma(0)) \ge G(\gamma(-\ell),\mbf x) - G(\gamma(-\ell),\gamma(0)),
	\end{align*}
	and so this sequence is increasing. We show it is bounded from above. Since $\mbf x \le \mbf z$, we get
	\[
	G(\gamma(-k),\mbf x) - G(\gamma(-k),\mbf z) \le 0
	\] 
	since the geodesic from $\gamma(-k)$ to $\mbf z$ has the option to pass through $\mbf x$. Thus, the limit \eqref{eq:b_from_G} exists, as desired. 
	
	Now, we prove both equations in \eqref{eq:mxgbz} together. It is immediate that $b(\gamma(0)) = 0$. By definition of $i_j^\star$, we have $(i_j^\star,j-1) = \gamma(-\ell)$ and $(i_j^\star,j) = \gamma(-\ell+1)$ for some $\ell \in \N$. Thus, 
	\begin{align*}
		&\quad \; b(i_j^\star,j-1) + G\bigl((i_j^\star,j),(m,n)\bigr) =b(\gamma(-\ell)) + G(\gamma(-\ell + 1),\gamma(0)) \\
		&= \lim_{k \to \infty} [G(\gamma(-k),\gamma(-\ell)) + G(\gamma(-\ell + 1),\gamma(0)) - G(\gamma(-k),\gamma(0))] =  \\
		&\lim_{k \to \infty}[G(\gamma(-k),\gamma(0)) -G(\gamma(-k),\gamma(0)) ] = 0.
	\end{align*}
	On the other hand, for general $i \in \llbracket j,m \rrbracket$, 
	\begin{align*}
		&\quad \; b(i,j-1) + G\bigl((i,j),(m,n)\bigr) \\
		& = \lim_{k \to \infty} [G\bigl(\gamma(-k),(i,j-1)\bigr) +G\bigl((i,j),(m,n)\bigr)- G(\gamma(-k),\gamma(0)) ] \\
		&\le \lim_{k \to \infty} [G(\gamma(-k),\gamma(0))- G(\gamma(-k),\gamma(0)) ] = 0. \qedhere
	\end{align*}
\end{proof}

\subsection{Exit points}
We introduce a notion of the \textbf{exit point} from an initial boundary. For $j \in \Z$ and a function $f:\Z_{\ge j} \to \R$, we define the exit point as follows: for integers $m \ge n \ge j$, 
\be \label{eq:exit_pt}
Z_{f,j}(m,n) = \max\Biggl\{ \argmax_{i \in \llbracket j,m \rrbracket} \Bigl[f(i) + G\bigl((i,j),(m,n)\bigr)\Bigr]\Biggr\}.
\ee

\begin{lemma} \label{lem:geodesics_are_maximizers}
    Assume that $b$ is an eternal solution. For $(m,n) \in \ZHS$, define the path $\gamma_{(m,n)}$ by \eqref{eq:gamma_path}. For $j \in \Z_{< n}$, define 
    \[
    i_j^* := \max\Bigl\{i \in \Z: \gamma_{(m,n)}(-k) = (i,j-1) \text{ for some }k \in \N \Bigr\}.
    \]
    That is, $i_j^\star$ is the rightmost location on horizontal level $j-1$ that the path $\gamma_{(m,n)}$ crosses. Then,
    \be \label{eq:i=argmax}
    i_j^\star = Z_{b(\aabullet,j-1),j}(m,n) = \max \Biggl\{\argmax_{j \le i \le m} \Bigl[b(i,j-1) + G\bigl((i,j),(m,n)\bigr)\Bigr]\Biggr\}.
    \ee
\end{lemma}

\begin{proof}

We first prove that that $i_j^\star$ lies in the argmax set in \eqref{eq:i=argmax}. By definition of $i_j^\star$, we know that $\gamma_{(m,n)}(-k) = (i_j^\star,j-1)$ for some $k \in \N$. Since $i_j^\star$ is the rightmost such value of $i$, it must be that $\gamma_{(m,n)}(-k + 1) = (i_j^\star,j)$. Hence, by Lemma \ref{lem:hs_geodesics}\ref{it:Buse=G}, we have that
    \be \label{eq:b+G_equal}
    b(i_j^\star,j-1) + G\bigl((i_j^\star,j),(m,n)\bigr) = b(i_j^\star,j-1) + b(m,n) - b(i_j^\star,j-1) = b(m,n).
    \ee
    On the other hand, Lemma \ref{lem:G_leb} gives us that for $i \in \llbracket j,m \rrbracket$,
    \be \label{eq:b+G_ineq}
    b(i,j-1) + G\bigl((i,j),(m,n)\bigr) \le b(i,j-1) + b(m,n) - b(i,j-1) \le b(m,n).
    \ee

    Now to show that $i_j^\star$ is the rightmost maximizer as in \eqref{eq:i=argmax}, we
    assume, by way of contradiction, that $i_j^\star <  z_j := Z_{b(\aabullet,j),j}(m,n)$.   
    Let $\pi = (\mbf x_{-\ell},\mbf x_{-(\ell - 1)},\ldots,\mbf x_0)$ be any geodesic from $(z_j,j)$ to $(m,n)$. Since $z_j > i_j^\star$, but both $(\mbf x_{-\ell},\mbf x_{-(\ell - 1)},\ldots,\mbf x_0)$ and $\gamma_{(m,n)}$ have terminal point $(m,n)$,  there must be some indices $k \ge 0$ and $p \in \llbracket 0, \ell - 1\rrbracket$ so that 
    \be \label{eq:xgammauv}
    \mbf x_{-p} = \gamma_{(m,n)}(-k) =:(u,v),\quad\text{and } \mbf x_{-(p+1)} = (u,v-1),\text{ while } \mbf \gamma_{(m,n)}(-(k+1)) = (u-1,v).
    \ee
    That is, when moving backwards from $(m,n)$, the path $\pi$ moves downward from $(u,v)$ and the path $\gamma_{(m,n)}$ moves leftward from $(u,v)$.

    Since $z_j$ is a maximizer, by \eqref{eq:b+G_equal}-\eqref{eq:b+G_ineq}, we must have \be \label{eq:G=bzi}
    G\bigl((z_j,j),(m,n)\bigr) = b(m,n) - b(z_j,j-1),
    \ee
    and furthermore, since $\pi$ is a geodesic containing the point $(u,v)$, Lemma \ref{lem:omega_upbd} implies that 
    \[
    \omega_{(u,v)} = b(\mbf x_{-p}) - b(\mbf x_{-(p+1)}) = b(u,v) - b(u,v-1).
    \]
    because otherwise, we would have inequality in \eqref{eq:G=bzi}. But by Lemma \ref{lem:omega_upbd} again, this implies $b(u,v-1) \ge b(u-1,v)$, so the definition of the path $\gamma_{(m,n)}$ \eqref{eq:gamma_path} implies $\gamma_{(m,n)}(-(k+1)) = (u,v-1)$, a contradiction to \eqref{eq:xgammauv}. 
\end{proof}

We obtain the following corollary. 
\begin{corollary} \label{cor:eternal_soln_recursion}
    If $b$ is an eternal solution, then for $(m,n) \in \ZHS$ and $j \in \Z_{\le n}$,
    \[
    b(m,n) = G_{b(\aabullet,j-1),j}(m,n) = \max_{i \in \llbracket j,m \rrbracket}\bigl[b(i,j-1) + G\bigl((i,j),(m,n)\bigr)\bigr].
    \]
\end{corollary}
\begin{proof}
Let $\gamma_{(m,n)}$ be the semi-infinite geodesic constructed in Lemma \ref{lem:hs_geodesics}. Since the path lies in $\ZHS$, when moving backwards from the point $(m,n)$, it cannot move infinitely far to the left. Hence, for every $j \le n$, there exists $i' \ge j$ such that $\gamma_{(m,n)}(-k) = (i',j-1)$. We take the rightmost such index $i'$. Then, by Lemma \ref{lem:hs_geodesics}\ref{it:Buse=G} followed by Lemma \ref{lem:geodesics_are_maximizers},
\[
b(m,n) = b(i',j-1) + G\bigl((i',j),(m,n)\bigr) = \max_{i \in \llbracket j,m \rrbracket}\bigl[b(i,j-1) + G\bigl((i,j),(m,n)\bigr)\bigr]. \qedhere
\]
\end{proof}

\subsection{Bounding solutions via exit point comparisons}
The following two results are called paths-crossing lemmas. They have been proved many times in the full-space context, and there is a similar result for half-space in \cite[Proposition 3.9 and Corollary 3.10]{Dauvergne-Zhang-2026}. We prove the precise results we need here for completeness. 
\begin{lemma} \label{lem:exit_pt_comp_ptl}
    Let $j \in \Z$ and $f_1,f_2:\Z_{\ge j} \to \R$ be two initial conditions. Assume, for some integers $\ell \ge m \ge n \ge j$, that 
    $
    Z_{f_1,j}(\ell,n) \le Z_{f_2,j}(m,n).
    $ 
    Then,
    \[
    G_{f_1,j}(\ell,n) - G_{f_1,j}(m,n) \le G_{f_2,j}(\ell,n) - G_{f_2,j}(m,n).
    \]
\end{lemma}
\begin{proof}
For notational simplicity, set $i_1 = Z_{f_1,j}(\ell,n)$ and $i_2 = Z_{f,j}(m,n)$. Since $i_1 \le i_2$ and $m \le \ell$, planarity implies that the rightmost geodesic from $(i_1,j)$ to $(\ell,n)$ must cross the rightmost geodesic from $(i_2,j)$ to $(m,n)$ at some point $\mbf x \in \ZHS$. Then, we observe that 
\be \label{eq:Gsplit1}
\begin{aligned}
G_{f_1,j}(\ell,n) &= f_1(i_1) + G\bigl((i_1,j),(\ell,n)\bigr)= f_1(i_1) + G\bigl((i_1,j),\mbf x\bigr) + G\bigl(\mbf x,(\ell,n)\bigr) - \omega_{\mbf x},
\end{aligned}
\ee
where the subtraction of $\omega_{\mbf x}$ comes to remove the double-counting of this weight. 
Similarly, we have 
\be \label{eq:Gsplit2}
G_{f_2,j}(m,n) = f_2(i_2) + G\bigl((i_2,j),\mbf x\bigr) + G\bigl(\mbf x,(m,n)\bigr) - \omega_{\mbf x}.
\ee
On the other hand, we have 
\be \label{eq:Gsplit3}
\begin{aligned}
G_{f_1,j}(m,n) &\ge f_1(i_1) + G\bigl((i_1,j),\mbf x\bigr) + G\bigl(\mbf x,(m,n)\bigr) - \omega_{\mbf x},\quad\text{and} \\
G_{f_2,j}(\ell,n) &\ge f_2(i_2) + G\bigl((i_2,j),\mbf x\bigr) + G\bigl(\mbf x,(\ell,n)\bigr) - \omega_{\mbf x}.
\end{aligned}
\ee
Combining \eqref{eq:Gsplit1}-\eqref{eq:Gsplit3}, we obtain
\begin{align*}
    &\quad \, G_{f_1,j}(\ell,n) - G_{f_1,j}(m,n)   \\
    &\le  f_1(i_1) + G\bigl((i_1,j),\mbf x\bigr) + G\bigl(\mbf x,(\ell,n)\bigr) - \Bigl(f_1(i_1) + G\bigl((i_1,j),\mbf x\bigr) + G\bigl(\mbf x,(m,n)\bigr)  \Bigr) \\
    &= f_2(i_2) + G\bigl((i_2,j),\mbf x\bigr) + G\bigl(\mbf x,(\ell,n)\bigr)  - \Bigl(f_2(i_2) + G\bigl((i_2,j),\mbf x\bigr) + G\bigl(\mbf x,(m,n)\bigr) \Bigr) \\
    &\le G_{f_2,j}(\ell,n) - G_{f_2,j}(m,n). \qedhere
\end{align*}
\end{proof}

We now prove the following corollary to Lemma \ref{lem:exit_pt_comp_ptl}.
\begin{lemma} \label{lem:exit_pt_comp_ptp}
 Let $j \in \Z$, and let $f:\Z_{\ge j} \to \R$ be an initial condition. Let $\ell \ge m \ge n \ge j$ be integers. Then, for all integers $p \ge Z_{f,j}(\ell,n)$, we have 
 \be \label{eq:Zptl_ptp_1}
 G_{f,j}(\ell,n) - G_{f,j}(m,n) \le G\bigl((p,j),(\ell,n)\bigr) - G\bigl((p,j),(m,n)\bigr).
 \ee
 Furthermore,  for all $p \in \llbracket j, Z_{f,j}(m,n) \rrbracket$, we have 
 \be \label{eq:Zptl_ptp_2}
 G_{f,j}(\ell,n) - G_{f,j}(m,n) \ge G\bigl((p,j),(\ell,n)\bigr) - G\bigl((p,j),(m,n)\bigr).
 \ee
 Lastly, if $p \in \Z_{\le r}$, then
 \be \label{eq:G_inc_comp}
 G\bigl((p,j),(\ell,n)\bigr) - G\bigl((p,j),(m,n)\bigr) \le G\bigl((r,j),(\ell,n)\bigr) - G\bigl((r,j),(m,n)\bigr).
 \ee
\end{lemma}
\begin{proof}
    We prove \eqref{eq:Zptl_ptp_1}, and the others follow a similar proof. We apply Lemma \ref{lem:exit_pt_comp_ptl} with $f_1 = f$, and define a new function $f_2$ as follows: First, we set $f_2(p)= 1$ and for all $i \ge j+1$ with $i \neq p$, we set $f_2(i)$ to some large negative value to guarantee that 
    \[
    Z_{f_2,j}(\ell,n) = Z_{f_2,j}(m,n) = p.
    \]
    This is possible because the inner max in \eqref{eq:exit_pt} is over a finite set of choices $i$. Then, we have $Z_{f_1,j}(\ell,n) \le p = Z_{f_2,j}(m,n)$, and by Lemma \ref{lem:exit_pt_comp_ptl},
    \begin{align*}
    G_{f,j}(\ell,n) - G_{f,j}(m,n) &\le G_{f_2,j}(\ell,n) - G_{f_2,j}(m,n) \\
    &= f_2(p) + G\bigl((p,j),(\ell,n)\bigr) - \Biggl(f_2(p) + G\bigl((p,j),(m,n)\bigr)\Biggr) \\
    &= G\bigl((p,j),(\ell,n)\bigr) - G\bigl((p,j),(m,n)\bigr). \qedhere
    \end{align*}
\end{proof}

\section{Shape function for half-space geometric LPP} \label{sec:shape}
For the remainder of the paper, we shall work with the $\GLPPc$ model defined in Definition \ref{def:GLPPc}. Each of the statements that follows holds for a choice of $q \in (0,1)$ and $c \in [0,q^{-1})$, which we will not explicitly mention in the statement of the result. We will say that the weights are defined on the probability space $(\Omega,\mathcal F,\Pp)$.

 Let $ \rho_{c}:[0,1] \to \R$ be the continuous function
    \begin{equation} \label{eq:mu_k_def}
    \rho_{c}(\kappa) = \begin{cases}
        \f{q^2(1 + \kappa) +2q\sqrt \kappa}{1 - q^2},& c \le 1, \text{ or } c > 1 \text{ and } \kappa < \Bigl(\f{1 - cq}{c -q}\Bigr)^2 \\
        \f{q(1 - qc + c^2 \kappa - qc \kappa)}{(c-q)(1 -cq)}, &c > 1 \text{ and } \kappa \ge \Bigl(\f{1 - cq}{c -q}\Bigr)^2.
    \end{cases}
    \end{equation}
The purpose of this section is to prove the following almost sure shape function. 
\begin{theorem} \label{thm:LLN}
    For each $\xi \in [0,1]$ and $k \ge t$ (requiring also $t \ge 0$ for $\xi = 1$), if $\xi_n$ is a deterministic sequence in $\llbracket 0,n \rrbracket$ such that $\lim_{n \to \infty} \f{\xi_n}{n} = \xi$, then
   \[
    \lim_{n \to \infty} \f{1}{n}G\bigl((-n +  \xi_n, -n ),(k,t)\bigr) = \rho_{c}(1 - \xi)\quad\text{a.s.},
    \]
\end{theorem}
\begin{remark}
By using symmetries, it was previously known from fluctuation results in \cite{Dmitrov-Zhou-2025,Zhou25,ddy26} that the convergence to $ \rho_{c}(1-\xi)$ holds in probability. In the full-space setting, this would suffice to give the almost sure convergence by the subadditive ergodic theorem. Except for the case $\xi = 0$, this cannot be employed here, so we must use other techniques. 
\end{remark}

Along the boundary, the existence of the almost sure limit follows from the subadditive ergodic theorem. This is recorded in the following result.  
\begin{proposition} \label{prop:boundary_shape}
	The following limit holds almost surely:
	\[
	\lim_{n \to \infty} \f{1}{n} G\bigl((-n,-n),(k,k)\bigr) = \rho_{c}(1),\quad\text{a.s.}.
	\]
\end{proposition}
\begin{proof}
	The existence of the limit is the subadditive ergodic theorem, using the diagonal translation-invariance in Lemma \ref{lem:shift_sym}. The constant is shown in \cite[Theorems 1.5 and 1.11 (see also Remark 1.9)]{ddy26}.
\end{proof}

\begin{proposition} \label{prop:upper_tail}
    Let $\kappa \in [0,1)$. Then, for each $\ve > 0$, there exist constants $A,\alpha > 0$ (depending on $c,q,\kappa$) so that for all $n \ge 1$,
    \[
    \Pp\Biggl(G\bigl((1,1),(n,\lfloor\kappa n\rfloor \vee 1)\bigr) - n \rho_{c}(\kappa) > n\ve\Biggr) \le A e^{-\alpha n^{1/2}}.
    \]
\end{proposition}
\begin{remark} \label{rmk:any_seq11}
	If we replace $\kappa n$ by a sequence $\kappa_n \in \N$ such that $\f{\kappa_n}{n} \to \kappa$, then for $\delta > 0$ and $n$ sufficiently large, we have
	\[
	G\bigl((1,1),(n, \lfloor n(\kappa - \delta) \rfloor \bigr) \le G\bigl((1,1),(n, \kappa_n \bigr) \le G\bigl((1,1),(n, \lfloor n(\kappa + \delta) \rfloor \bigr),
	\]
	so by continuity of $ \rho_{c}$, the tail bounds in Proposition \ref{prop:upper_tail} still hold with $\lfloor \kappa n \rfloor$ is replaced by $\kappa_n$.
\end{remark}

\begin{proof} [Proof of Proposition \ref{prop:upper_tail}] Since there is only one up-right path from $(1,1)$ to $(n,1)$, the $\kappa = 0 $ case just follows from tail bounds for geometric random walk (e.g., Lemma \ref{lem:Geo_RW_LD}). For this, we observe that 
	$
	 \rho_{c}(0) = \f{q^2}{1-q^2},
	$
	which is the mean of the $\Geo(q^2)$ weights $\omega_{(i,1)}$ for $i \ge 2$.

	 Thus, we may now fix $\kappa\in (0,1)$ and set $z_\kappa=\frac{1+q\sqrt{\kappa}}{q+\sqrt{\kappa}}$. When $c\in (0,z_\kappa)$, it follows from the analysis in Step 3 of proof of \cite[Proposition 4.2]{Zhou25}  that there exists a constant $C>0$ such that for all $a\ge 0$ and $n\in \N$,
\begin{align}\label{ew1}
\mathbb{P}\left(G\left((1,1),(n,\lfloor \kappa n\rfloor\right)-n \rho_{c}(\kappa)\ge a n^{1/3}\right) \le Cn^{2/3}e^{-n^{3/4}/C}+Ce^{-a/C}.
\end{align}
The deviations are of the order $n^{1/3}$, as in this regime, the joint last passage times converge to the Airy$_2$ process under KPZ scaling \cite{Zhou25}.  As $ \rho_{c}$ does not depend on $c$ when $c\in [0,z_{\kappa}]$, and $G\left((1,1),(n,\lfloor \kappa n\rfloor\right)$ is stochastically increasing in $c$, the above estimate remains true when $c=0$. 

When $c \in (z_\kappa,q^{-1})$ it follows from the analysis Step 3 of proof of \cite[Proposition 7.3]{Dmitrov-Zhou-2025}  that there exists a constant $C>0$ such that for all $a\ge 0$ and $n\in \N$,
\begin{align}\label{ew2}
\mathbb{P}\left(G\left((1,1),(n,\lfloor \kappa n\rfloor\right)-n \rho_{c}(\kappa)\ge a n^{1/2}\right) \le Ce^{-n/C}+Ce^{-a/C}.
\end{align}
Here the deviations are of the order $n^{1/2}$ as in this regime, the joint last passage times converge to Brownian motion under diffusive scaling \cite{Dmitrov-Zhou-2025}. Taking $a=n^{2/3}\varepsilon$ and $a=n^{1/2}\varepsilon$ in \eqref{ew1} and \eqref{ew2}, respectively, we arrive at desired tail bounds for all $c\neq z_{\kappa}$.  The case $c=z_\kappa$ can be obtained by taking $c>z_\kappa$ sufficiently close to $z_\kappa$ by noting that $ \rho_{c}(\kappa)$ is strictly increasing on $[z_\kappa,q^{-1})$ and that $G((1,1),(n,\lfloor\kappa n\rfloor))$ is stochastically increasing in $c$.
\end{proof}

Proposition \ref{prop:upper_tail}, combined with the Borel-Cantelli lemma, gives us an upper bound for the shape function. For the lower bound, it would suffice to have an analogous lower tail bound. However, the available formulas are not well-suited to lower tails. Instead, we use the fact that the function $\rho_c$ has two phases. In one phase (the top line in \eqref{eq:mu_k_def}), the shape function agrees with the  shape function for full-space geometric LPP. In the other phase, the shape function is linear. On a macroscopic scale, this indicates that the geodesic from $(k,t)$ (when going backwards to $(-n + \xi_n,-n)$) follows the boundary until it finds an optimal location to leave and go into the bulk, where it macroscopically is equivalent to full-space LPP.

With this in mind, we couple the half-space last-passage values $G$ with full-space last-passage values $\GFS$ as follows: Given the weights $\bigl(\omega_{(i,j)}\bigr)_{(i,j) \in \ZHS}$ that define the values of $G$, define weights $\bigl(\omega'_{(i,j)}\bigr)_{(i,j) \in \Z^2}$ as follows. For $i > j$, set $\omega'_{(i,j)} = \omega_{(i,j)}$, and let $\bigl(\omega_{(i,j)}'\bigr)_{i \le j}$ be a collection of i.i.d. $\Geo(q^2)$ random variables, independent of the weights $\omega$ (and assumed to live on the same probability space).  Then,  $\bigl(\omega_{(i,j)}'\bigr)_{(i,j) \in \Z^2}$  are i.i.d. $\Geo(q^2)$ weights, and for $(i,j) \le (m,n)$, let $\GFS((i,j),(m,n))$ be the maximal sum of weights along all up-right paths from $(i,j)$ to $(m,n)$ (not just confined to the half-space $\ZHS$). We define
\be \label{rhoFS}
\rhoFS(\kappa) = \f{q^2(1 + \kappa) +2q\sqrt \kappa}{1 - q^2}.
\ee

\begin{proposition} \cite[Theorem 1.1]{Johansson-2000} \label{prop:Johansson_lT}
For each $\kappa \in [0,1]$ and $\ve > 0$, there exists constants $A,\alpha > 0$ so that, for all $n \ge 1$,
\[
\Pp\biggl(\Bigl|\GFS\bigl((1,1),(n,\lfloor \kappa n \rfloor\vee 1)\bigr) - n  \rhoFS(\kappa)\Bigr| \ge n \ve \biggr) \le A e^{-\alpha n}.
\]   
\end{proposition}

\begin{remark} \label{rmk:any_seq}
     In \cite{Johansson-2000}, it is shown that the lower tail can be bounded by $Ae^{-\alpha n^2}$, but this is not necessary for our purposes. Just as in Remark \ref{rmk:any_seq11}, we may replace $\lfloor \kappa n \rfloor$ with any sequence $\kappa_n \in \N$ with $\lim_{n \to \infty} \f{\kappa_n}{n} = \kappa$.   
\end{remark}

\begin{proof}[Proof of Theorem \ref{thm:LLN}]
   By Lemma \ref{lem:G_monotonicity}, we have that 
   \begin{align*}
   &G\bigl((-n + \xi_n, -n ),(t,t)\bigr) \le G\bigl((-n + \xi_n, -n ),(k,t)\bigr)  \le G\bigl((-n  + \xi_n, -n ),(k,k)\bigr),
   \end{align*}
   so it suffices to prove the statement for $k = t$ (here, $k$ is fixed while $n \to \infty$). First, note that if $\xi = 0$, then for $\ve > 0$ and all sufficiently large $n$, Lemma \ref{lem:G_monotonicity} implies that
   \[
   G\bigl((-n + \ve n, -n ),(k,k)\bigr) \le G\bigl((-n  + \xi_n, -n ),(k,k)\bigr) \le G\bigl((-n, -n ),(k,k)\bigr),
   \]
   so by
 Proposition \ref{prop:boundary_shape} and continuity of $\rho_c$, we may assume $\xi > 0$. Similarly if $\xi = 1$ and $k \in \Z_{\ge 0}$, then for large $n$ and $\ve \in (0,1)$, we have  
   \[
   G\bigl((0,-n),(0,0)\bigr) \le G\bigl((-n  + \xi_n, -n ),(k,k)\bigr) \le G\bigl((-\ve n, -n ),(k,k)\bigr),
   \]
   and the lower bound is simply a $\Geo(q^2)$ random walk with mean $\f{q^2}{1-q^2} = \rho_{c}(0)$. Hence, we may also assume $\xi < 1$. We prove upper and lower bounds.

   \medskip \noindent \textbf{Upper bound:} By applying Lemma \ref{lem:shift_sym} followed by Lemma \ref{lem:refl_sym}, then another application of Lemma \ref{lem:shift_sym}, we have 
   \be \label{eq:G_deq_string}
   \begin{aligned}
   &\quad \;G\bigl((-n + \xi_n,-n),(k,k)\bigr)\\
   &\deq G\bigl((\xi_n,0),(n+k,n+k)\bigr)  \\
   &\deq G\bigl((0,0),(n+k, n + k  - \xi_n)\bigr) \\
   &\deq G\bigl((1,1),(n + k + 1,n + k + 1 - \xi_n)\bigr).
   \end{aligned}
   \ee
   Then, by Proposition \ref{prop:upper_tail}, for any $\ve > 0$, we have 
   \[
   \Pp\Bigl(G\bigl((-n + \xi_n,-n),(k,k)\bigr) - n \rho_{c}(1-\xi) > \ve n \Bigr) \le Ae^{-\alpha n^{1/2}}
   \]
   for some $A,\alpha > 0$. By the Borel-Cantelli lemma, we have \
      \begin{equation} \label{eq:shape_upbd}
   \limsup_{n \to \infty} \f{1}{n}G\bigl((-n + \lfloor \xi n \rfloor,-n),(k,k)\bigr) \le \rho_{c}(1 - \xi),\quad\text{a.s.}
   \end{equation}
   
   \medskip \noindent \textbf{Lower bound:} Define the real number $x = x_{c,q,\xi}$ by 
\[
x = \begin{cases}
    0 &c \le 1 \quad \text{or } \xi \ge \f{1}{1 - \bigl(\f{1-cq}{c-1}\bigr)^2}, \\
    1 - \f{\xi}{1 - \bigl(\f{1-cq}{c - q}\bigr)^2} &\text{otherwise}.
\end{cases}
\]
For this choice of $x$, we observe that
\be \label{eq:rho_linear}
 \rho_{c}(1 - \xi) = x \rho_{c}(1) + (1 - x) \rho_{c}\Bigl(\f{1 - x - \xi}{1-x}\Bigr).
\ee
To see this, note the statement is trivially true for $x = 0$. When $x > 0$, we have that 
\be \label{eq:equality_to_lim}
\f{1 - x - \xi}{1 - x} = \Bigl(\f{1-cq}{c - q}\Bigr)^2,
\ee
and so by the definition of $ \rho_{c}$ in \eqref{eq:mu_k_def}, the function $\kappa \mapsto \rho_{c}(\kappa)$ is linear on the interval $\Bigl[\f{1 - x - \xi}{1 - x},1\Bigr]$, giving us \eqref{eq:rho_linear}. Furthermore, by \eqref{eq:equality_to_lim} and definition of the functions $ \rho_{c}$ \eqref{eq:mu_k_def} and $\rhoFS$ \eqref{rhoFS}. we observe that 
   \be \label{rho=rhoFS}
    \rho_{c}\Bigl(\f{1 - x - \xi}{1 -x}\Bigr) = \rhoFS\Bigl(\f{1 - x - \xi}{1 -x}\Bigr).
   \ee
The argument that follows can be interpreted in the following way: The geodesic going backwards from $(k,k)$ to $(-n + \xi_n,-n)$ will travel approximately along a diagonal path to the point $(-xn,-xn)$, then  move to the point $(-n + \xi_n,-n)$. From here, the point $(-n + \xi_n,-n)$ is far enough from the boundary relative to the vertical distance to $(-xn,-xn)$ that the geodesic will use almost only weights in the bulk. Thus, the portion of the passage time from $(-n + \xi_n,-n)$ to $(-xn,-xn)$ will have the same leading-order term as for the full-space model. See Figure \ref{fig:boundary_split}. 

\begin{figure}
    \centering
          \begin{tikzpicture}[scale=0.8, thick]

\def\n{6}

\begin{scope}
  \clip
    (0,0) --
    (\n,\n) --
    (\n,0) --
    (0,0);

\end{scope}
\fill[gray!30]
  (0,0) -- (\n+0.1 ,0) -- (\n+.1,\n+.1) -- cycle;
\draw[line width=1pt] (0,0) -- (\n,\n);


\draw[blue, line width=2pt]
  (2,0) -- (3.1,2.9)--(6.1,5.9);

\foreach \p in {(2,0),(3.1,2.9),(6.1,5.9)} {
  \fill \p circle (2.2pt);
}

\node[below] at (2,0) {$(-n + \xi_n,-n)$};
\node[below right] at (3.1,2.9) {$(-xn,-xn)$};
\node[right] at (6,6) {$(k,k)$};

\end{tikzpicture}
    
    \caption{The geodesic (blue/thick) from $(-n + \xi_n,-n)$ to $(k,k)$ will first travel to the boundary (black, thin) at approximately the point $(-xn,-xn)$, then move along the diagonal to $(k,k)$. The blue path is shown to be slightly offset from the boundary for visual purposes. The gray shading indicates the half-space $\ZHS$. }
    \label{fig:boundary_split}
\end{figure}
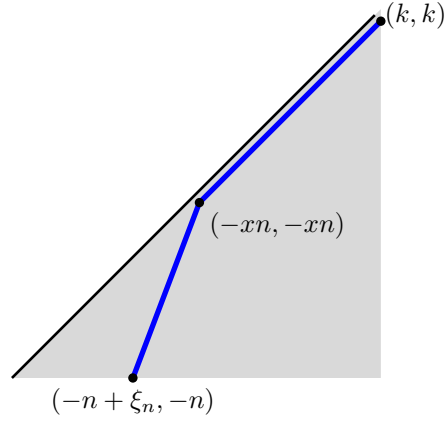

Since $G$ is the maximal passage time, we have the lower bound
\be \label{eq:Glbx}
   \begin{aligned}
       &\quad \;G\bigl((-n + \lfloor \xi n \rfloor,-n),(k,k)\bigr) + \omega_{(-\lfloor xn \rfloor - 1,-\lfloor xn \rfloor - 1) }  \\
       &\ge G\bigl((-n + \lfloor \xi n \rfloor,-n),(-\lfloor xn \rfloor - 1, -\lfloor xn \rfloor - 1) \bigr) + G\bigl((-\lfloor xn \rfloor - 1, -\lfloor xn \rfloor - 1),(k,k)\bigr) .
   \end{aligned}
   \ee
   By Proposition \ref{prop:boundary_shape}, we have
   \[
   \lim_{n \to \infty}  \f{1}{n}G\bigl((\lfloor xn \rfloor - 1, \lfloor xn \rfloor-1),(k,k)\bigr) = x \rho_{c}(1),\quad\text{a.s.}
   \]
   Then, by \eqref{eq:Glbx} and the exponential tails of the geometric weights $\omega$, Equation \eqref{eq:rho_linear} implies that it suffices to show 
   \be \label{eq:Glwbd}
   \liminf_{n \to \infty} \f{1}{n}G\bigl((-n + \lfloor \xi n \rfloor,-n),(-\lfloor xn \rfloor - 1, -\lfloor xn \rfloor - 1) \bigr) \ge (1 - x) \rho_{c}\Bigl(\f{1 - x - \xi}{1 -x}\Bigr),\quad \text{a.s.}
   \ee
   
   Then, by \eqref{rho=rhoFS} and by continuity of $ \rho_{c}$ and $\rhoFS$, to prove \eqref{eq:Glwbd}, it suffices to show that for sufficiently small $\delta > 0$,
\be \label{eq:Glwbd2}
\liminf_{n \to \infty} \f{1}{n}G\bigl((-n + \lfloor \xi n \rfloor,-n),(-\lfloor xn \rfloor - 1, -\lfloor xn \rfloor - 1) \bigr) \ge L(x,\delta,\xi),\quad \text{a.s.},
\ee
where 
\[
L(x,\delta,\xi) := (1 - x - \delta) \rho_{FS}\Bigl(\f{1 - x - \xi}{1 -x-\delta}\Bigr)
\]

Since $\xi > 0$, we may choose $\delta  \in (0,\xi)$. Note that $\xi < 1-x$, so $\delta < 1-x$ as well. By applying Lemma \ref{lem:shift_sym}, then Lemma \ref{lem:refl_sym}, followed by another application of Lemma \ref{lem:shift_sym}, then the monotonicity in Lemma \ref{lem:G_monotonicity}, we have 
\be \label{eq:Glbdelta}
\begin{aligned}
    &\quad \; G\bigl((-n + \lfloor \xi n \rfloor,-n),(-\lfloor xn \rfloor - 1, -\lfloor xn \rfloor - 1) \bigr)\\
    &\deq  G\bigl((1,1),(n  - \lfloor xn \rfloor, n - \lfloor xn \rfloor - \lfloor \xi n])\bigr) \\
    &\ge G\bigl((1 + \lfloor \delta n \rfloor,1),(n  - \lfloor xn \rfloor, n - \lfloor xn \rfloor - \lfloor \xi n])\bigr).
\end{aligned}
\ee

Let $B_{x,\delta,\xi,n}$ be the event on which the rightmost geodesic for $\GFS$ from $(1 + \lfloor \delta n \rfloor,1)$ to $(n  - \lfloor xn \rfloor, n - \lfloor xn \rfloor - \lfloor \xi n])$ never passes through a vertex $(i,j)$ with $i \le j$. On  this event, our choice of coupling between half-space and full-space LPP implies that 
\[
G\bigl((1 + \lfloor \delta n \rfloor,1),(n  - \lfloor xn \rfloor, n - \lfloor xn \rfloor - \lfloor \xi n])\bigr) \ge \GFS\bigl((1 + \lfloor \delta n \rfloor,1),(n  - \lfloor xn \rfloor, n - \lfloor xn \rfloor - \lfloor \xi n])\bigr).
\]
 Then, by \eqref{eq:Glbdelta}, for any $\ve > 0$, we have 
\begin{align*}
    &\quad \; \Pp\Bigl(G\bigl((-n + \lfloor \xi n \rfloor,-n),(-\lfloor xn \rfloor - 1, -\lfloor xn \rfloor - 1) \bigr) - n L(x,\delta,\xi)\Bigr) \le -\ve n \Bigr) \\
    &\le\Pp\Bigl(\GFS\bigl((1 + \lfloor \delta n \rfloor,1),(n  - \lfloor xn \rfloor, n - \lfloor xn \rfloor - \lfloor \xi n])\bigr)- n L(x,\delta,\xi) \le -\ve n \Bigr)+  \Pp(B_{x,\delta,\xi,n}^C).
\end{align*}
In the full-space setting, the law of passage times are preserved by any shift, so by shifting both initial and terminal points by $(-\lfloor \delta n \rfloor,0)$ and using Proposition \ref{prop:Johansson_lT} (see also Remark \ref{rmk:any_seq}), we have that the first term on the right-hand side above is bounded by $A e^{-\alpha n}$ for some $A,\alpha > 0$. By the Borel-Cantelli Lemma, to prove \eqref{eq:Glwbd}, it then suffices to show that, for any $x,\delta,\xi$, $\Pp(B_{x,\delta,\xi,n}^C)$ decays exponentially in $n$. This is shown in the following lemma. 
\end{proof}
\begin{lemma}
    Fix $q \in (0,1)$, $\delta \in (0,1)$, and let $z_n$ be a sequence in $\Z$ with $\lim_{n \to \infty} \f{z_n}{n} = z > \delta + 1$. Let $B_{n}$ be the event that the rightmost geodesic for $\GFS$ (in full-space) from $(1+ \lfloor \delta n \rfloor,1)$ to $(z_n,n)$ never uses a vertex $\omega'_{(i,j)}$ with $i \le j$. That is, $B_n$ is the event where the geodesic stays strictly to the right of the main diagonal line $i = j$. Then, there exist constants $A,\alpha > 0$ so  that $\Pp(B_n^C) \le A e^{-\alpha n}$. 
\end{lemma}
\begin{proof}
For $j \in \llbracket 1,n \rrbracket$,  define $i_j$ to be the minimal integer $i$ such that $(i,j)$ lies on the rightmost geodesic between $(1+ \lfloor \delta n \rfloor,1)$ to $(z_n,n)$. Then, $B_n$ holds if and only if $i_j > j$ for all $1 \le j \le n$. Since paths move up and to the right, we have $i_{j} \le i_{j+1}$ for all $j \in \llbracket 1,n-1\rrbracket$. For $n \in \N$, define
    \be \label{eq:mnKn}
    m_n = \left\lfloor \f{\delta}{3(z-\delta)}n\right\rfloor,\quad\text{and}\quad K_n = \max\{k\in \N: (k+1)m_n \le n\}, 
    \ee
    and note that $K_n$ is bounded in $n$. Since $z-\delta > 1$, we also note that $m_n \le \f{1}{3}n$, so $K_n \ge 3$. Also, define 
    \be \label{frakm}
    \mathfrak m := \lim_{n \to \infty} \f{m_n}{n} = \f{\delta}{3(z - \delta)}.
    \ee
Fix $a \in (1,z-\delta)$, and let $E_n$ be the event on which    
    \[
    i_{k m_n} > a(k+2) m_n,\quad\text{for all}\quad k \in \llbracket 2, K_n \rrbracket.
    \]
    We claim that $E_n \subseteq B_n$ for all sufficiently large $n$. Indeed, since the geodesic starts at the point $(1 + \lfloor \delta n \rfloor,1)$, we see that $i_j \ge 1 + \lfloor \delta n \rfloor$ for all $j$. Hence, since $z-\delta > a$,  for all sufficiently large $n$,
    \[
    i_j \ge 1 + \lfloor \delta n \rfloor > 3am_n  > j \quad\text{for all }j \in \llbracket 1, m_n \rrbracket.
    \]
    In particular, $i_{m_n} > 3am_n$. 
    
Now, if $i_{km_n} > a(k+2)m_n$ for some $k \in \llbracket 1, K_n - 1 \rrbracket$, then since $a > 1$, for $j \in \llbracket km_n, (k+2) m_n \rrbracket$, we have 
\[
i_j \ge i_{km_n} > a(k+2)m_n  > j,
\]
Furthermore, if $i_{K_n m_n} > a(K_n + 2)m_n$, then since $n < (K_n + 2)m_n$ by definition of $K_n$,  we  have 
\[
i_j \ge a(K_n+2) m_n > (K_n + 2)m_n > j\quad\text{for all }j \in \llbracket K_n, n \rrbracket.
\]
Hence, $E_n \subseteq B_n$, as desired. Now, define
\[
x_{k,n} := \Bigl\lfloor \delta n\Bigl(\f{k}{3} + 1\Bigr)\Bigr \rfloor,
\]
and in fact, since $a < z-\delta$, $x_{k,n} \gg a(k+1)m_n$ when $n$ is large. 
The point $(x_{k,n},k m_n)$ is chosen to be approximately on the straight line path between $(1+ \lfloor \delta n \rfloor,1)$ and $(z_n,n)$. In words, if $E_n^C$ occurs, then the geodesic must take some large detour from this straight-line path.

Now, since we showed that $i_{m_n} > 3a m_n$ automatically for all sufficiently large $n$, we get 
\be\label{eq:ikbd}
\begin{aligned}
\Pp(B_n^C) &\le \Pp(E_n^C) \le \sum_{k = 2}^{K_n}\Pp( i_{km_n} \le a(k+2) m_n) \le \sum_{k = 2}^{K_n} \sum_{i = 1 + \lfloor \delta n \rfloor }^{a(k+2)m_n} \Pp(D_{i,k,n}),
\end{aligned}
\ee
where $D_{i,k,n}$ is the event that the rightmost geodesic passes through $(i,km_n)$. 

Let $\eta > 0$ be a small number to be chosen later. For $r\in \Z_{\ge 0}$, let $y_{r,n} = \lfloor \delta n \rfloor + r \lfloor \eta n \rfloor$,  and let 
\be \label{eq:Rvek}
R_{\eta,k,n} = \max\{r \in \Z_{\ge 0}: y_{r,n} \le a(k+2) m_n\}.
\ee
Note that $R_{\eta,k,n}$ is bounded in $n$ for each fixed $k  \in \Z_{\ge 0}$ and $\eta > 0$. Then, we have
\be \label{eq:sum_break2}
\sum_{k = 2}^{K_n} \sum_{i = \lfloor \delta n \rfloor + 1}^{(k+2)m_n} \Pp(D_{i,k,n}) \le \sum_{k = 2}^{K_n} \sum_{r = 0}^{R_{\eta,k,n}} \sum_{i = y_{r,n} + 1}^{y_{r+1,n}} \Pp(D_{i,k,n})
\ee
On the event $D_{i,k,n}$, we have that the passage time through the vertex $(i,k m_n)$, namely
\[
\GFS\bigl((1 + \lfloor \delta n \rfloor,1),(i,km_n)\bigr) + \GFS\bigl((i,km_n),(z_n,n)\bigr) - \omega_{(i,k_m)}'
\]
is greater than the passage time through $(x_{k,n},km_n)$, namely
\be \label{eq:straight_line}
T_{k,n} := \GFS\bigl((1 + \lfloor \delta n \rfloor,1),(x_{k,n},km_n)\bigr) + \GFS\bigl((x_{k,n},k m_n),(z_n,n)\bigr) - \omega_{(x_{k,n},km_n)}',
\ee
Now, if $i \in \llbracket y_{r,n} - 1, y_{r+1,n} \rrbracket$, then we have 
\[
\begin{aligned}
&\quad \;\GFS\bigl((1 + \lfloor \delta n \rfloor,1),(i,km_n)\bigr) + \GFS\bigl((i,km_n),(z_n,n)\bigr)  \\
&\le \GFS\bigl((1 + \lfloor \delta n \rfloor,1),(y_{r+1,n},km_n)\bigr) + \GFS\bigl((y_{r,n},km_n),(z_n,n)\bigr) =: T_{\eta,r,k,n}.
\end{aligned}
\]
Thus, we see that 
\be \label{eq:Dikn_bd}
\Pp(D_{i,k,n}) \le \Pp\Bigl(T_{\eta,r,k,n} - \omega_{(i,k_m)}  > T_{k,n}\Bigr),\quad\text{for }i \in \llbracket y_{r,n} - 1, y_{r+1,n}\rrbracket.
\ee
By Proposition \ref{prop:Johansson_lT} and Remark \ref{rmk:any_seq}, the quantity $T_{\eta,r,k,n}$ is concentrated around
\be \label{eq:conc1}
 \Biggl( k \mathfrak m \rhoFS\Bigl(\f{r\eta + \eta}{k \mathfrak m}\Bigr) + (1 - k \mathfrak m) \rhoFS\Bigl(\f{z - \delta - r\eta}{1 - k \mathfrak m}\Bigr) \Biggr),
\ee
where we recall the definition of $\mathfrak m$ from \eqref{frakm} and use the fact that
\[
\rhoFS(\kappa) := \kappa \rhoFS(\kappa^{-1}).
\]
On the other hand, since $(x_{k,n},km_n)$ lies approximately on the straight-line path from $(1+\lfloor \delta n \rfloor)$ to $(z_n,n)$, the quantity $T_{k,n}$ in \eqref{eq:straight_line} is concentrated around 
\be \label{eq:conc2}
n \Bigl( k\mathfrak m \rhoFS(z - \delta)+ (1- k \mathfrak m)\rhoFS(z -\delta) \Bigr) = n \rhoFS(z - \delta).
\ee
By definition of $K_n$, we have $(K_n + 1)m_n \le n$, so 
\[
\mathfrak m \le k \mathfrak m  \le K_n \mathfrak m \le 1- \mathfrak m \quad\text{for all }k \in \llbracket 1,K_n \rrbracket.
\]
By taking the second  derivative, we can readily see that the function $\rhoFS$ is strictly concave. Then, for each $k \in \llbracket 1, K_n \rrbracket$, the function
\be \label{x_concave}
x \mapsto k \mathfrak m \rhoFS\Bigl(\f{x}{k \mathfrak m}\Bigr) + (1- k \mathfrak m) \rhoFS\Bigl(\f{z - \delta - x}{1 - k \mathfrak m}\Bigr)
\ee
is uniquely maximized at $x = z -\delta$ and is strictly increasing for $x < z - \delta$.  Then, for $k \in \llbracket 1,K_n \rrbracket$ and $r \in \llbracket 0, R_{\eta,k,n} \rrbracket$, we have
\[
\delta + r\eta \le \delta + R_{\eta,k,\ve} \eta \le a(k+2) \mathfrak  m \le a(K_n + 2) \mathfrak m \le a + a \mathfrak m < z - \delta + \f{\delta}{3}.
\]
Above, in the second inequality, we used $\lfloor \delta n \rfloor + R_{\eta,k,n} \lfloor \eta n \rfloor  \le a(k+2)m_n$ by definition of $R_{\eta,k,n}$ \eqref{eq:Rvek}, in the penultimate inequality, we used $(K_n+1)m_n \le n$ by definition of $K_n$ \eqref{eq:mnKn}, and in the last inequality, we used $a < z-\delta$ and the definition of $\mathfrak m$ \eqref{frakm}.
Subtracting $\delta$, we get $
r\eta < (z - \delta) - \f{2\delta}{3}$, and since the function \eqref{x_concave} is strictly increasing for $x < z-\delta$, we have that, for $k \in \llbracket 1,K_n \rrbracket$ and $r \in \llbracket 0, R_{\eta,k,n} \rrbracket$,
\be \label{eq:rhoFSineq1}
\begin{aligned}
&\quad \; k \mathfrak m \rhoFS\Bigl(\f{r\eta }{k \mathfrak m}\Bigr) + (1 - k \mathfrak m) \rhoFS\Bigl(\f{z - \delta - r\eta}{1 - k \mathfrak m}\Bigr) \\
&< k \mathfrak m \rhoFS\Bigl(\f{z - \delta - 2\delta/3 }{k \mathfrak m}\Bigr) + (1 - k \mathfrak m) \rhoFS\Bigl(\f{2\delta/3}{1 - k \mathfrak m}\Bigr)  < \rhoFS(z - \delta).
\end{aligned}
\ee
We now claim that we may then choose $\eta > 0$ sufficiently small so that, for $k \in \llbracket 1,K_n \rrbracket$ and $r \in \llbracket 0, R_{\eta,k,n} \rrbracket$,
\begin{align} \label{eq:all_eta_small}
&\quad \; t_{\eta,r,k,n}:= k \mathfrak m \rhoFS\Bigl(\f{r\eta + \eta }{k \mathfrak m}\Bigr) + (1 - k \mathfrak m) \rhoFS\Bigl(\f{z - \delta - r\eta}{1 - k \mathfrak m}\Bigr) < \rhoFS(z - \delta).
\end{align}
To see this, we see that we choose such an $\eta > 0$ sufficiently small that works for all $k \in \llbracket 1,K_n \rrbracket$ when $r = 0$ by noting that $K_n$ is bounded. On the other hand,   if $r \ge 1$, by an explicit computation using the definition \eqref{rhoFS}, we can see that $\rhoFS'(x) = C_1 +  C_2x^{-1/2}$ for all $x \in (0,1]$ and constants $C_1,C_2$ depending only on $q$. Thus, we have
\[
k \mathfrak m \rhoFS\Bigl(\f{r\eta + \eta }{k \mathfrak m}\Bigr) = k \mathfrak m\Biggl( \rhoFS\Bigl(\f{r\eta }{k \mathfrak m}\Bigr) + (C_1 + C_2(x^\star)^{-1/2})\f{\eta}{k\mathfrak m}\Biggr),
\]
where $x^\star$ is some number strictly between $\f{r\eta}{k\mathfrak m}$ and $\f{r\eta + \eta}{k\mathfrak m}$. We can then further bound
\[
k \mathfrak m \rhoFS\Bigl(\f{r\eta + \eta }{k \mathfrak m}\Bigr) < k \mathfrak m\rhoFS\Bigl(\f{r\eta }{k \mathfrak m}\Bigr) + C_1\eta + C_2\Bigl(\f{r\eta}{k\mathfrak m}\Bigr)^{-1/2}\eta \le  k \mathfrak m\rhoFS\Bigl(\f{r\eta }{k \mathfrak m}\Bigr) + C\eta^{1/2},
\]
where the last step follows for a new constant $C$ because $r \ge 1$ and $\mathfrak m < k\mathfrak m < 1 - \mathfrak m$.
Then, since $K_n$ is bounded, we may $\eta$ sufficiently small so that, for all $k \in \llbracket 1,K_n \rrbracket$,
\[
k \mathfrak m \rhoFS\Bigl(\f{z - \delta - 2\delta/3 }{k \mathfrak m}\Bigr) + (1 - k \mathfrak m) \rhoFS\Bigl(\f{2\delta/3}{1 - k \mathfrak m}\Bigr) + C\eta^{1/2}  < \rhoFS(z - \delta).
\]
Equation \eqref{eq:rhoFSineq1} then gives us that for all $k \in \llbracket 1,K_n \rrbracket$ and $r \in \llbracket 0, R_{\eta,k,n} \rrbracket$,
\begin{align*}
     &\quad \; k \mathfrak m \rhoFS\Bigl(\f{r\eta  + \eta}{k \mathfrak m}\Bigr) + (1 - k \mathfrak m) \rhoFS\Bigl(\f{z - \delta - r\eta}{1 - k \mathfrak m}\Bigr)  \\
     &<  k \mathfrak m \rhoFS\Bigl(\f{z - \delta - 2\delta/3 }{k \mathfrak m}\Bigr) + (1 - k \mathfrak m) \rhoFS\Bigl(\f{2\delta/3}{1 - k \mathfrak m}\Bigr) + C\eta^{1/2}< \rhoFS(z - \delta),
\end{align*}
giving \eqref{eq:all_eta_small}, as desired. 

Now, to complete the proof, for our particular choice of $\eta$ let $\ve$ be an arbitrary positive number less  than $\rhoFS(z-\delta) - t_{\eta,r,k,n}$ for all $k \in \llbracket 1,K_n \rrbracket$ and $r \in \llbracket 0, R_{\eta,k,n} \rrbracket$. This is possible to choose by \eqref{eq:all_eta_small} since $K_n$ and $R_{\eta,k,n}$ are each bounded by definition (for the fixed choice of $\eta$). 

Then, combining \eqref{eq:ikbd}, \eqref{eq:sum_break2}, and \eqref{eq:Dikn_bd}, we have 
\begin{align*}
    &\Pp(B_n^c) \le \sum_{k = 2}^{K_n} \sum_{r = 0}^{R_{\eta,k,n}} \sum_{i = y_{r,n} + 1}^{y_{r+1,n}}\Pp\Bigl(T_{\eta,r,k,n} - \omega_{(i,k_m)}'  > T_{k,n}\Bigr) \\
    &\le Cn \sum_{k = 2}^{K_n} \sum_{r = 0}^{R_{\eta,k,n}}\Bigl[\Pp\Bigl(T_{\eta,r,k,n} - nt_{\eta,r,k,n} > \f{\ve}{3}n\Bigr) + \Pp\Bigl(T_{k,n} - n\rhoFS(z-\delta) < - \f{\ve}{3}n\Bigr) + \Pp\Bigl(\omega_{(0,0)}' < -\f{\ve}{3}n\Bigr) \Bigr] \\
    &\le Cne^{-\alpha n} \le C'e^{-\alpha'n},
\end{align*}
where $\alpha,\alpha',C > 0$. In the second equality, we used the fact that the weights $\omega'$ are i.i.d., and in the last step, we used Proposition \ref{prop:Johansson_lT} (and the discussion around \eqref{eq:conc1} and \eqref{eq:conc2} on concentration) and the fact that $K_n$ and $R_{\eta,k,n}$ are bounded.      
\end{proof}

\section{Controlling the exit points} \label{sec:exit_pt_control}
In this section, we use the shape theorem from the previous section to control the exit points from sloped initial conditions. This is accomplished in Proposition \ref{prop:converge_to_max}. This result requires a uniform slope condition as the initial time goes to $-\infty$. Then, in Lemma \ref{lem:stat_obeys_slope}, we show that the known invariant measures satisfy this condition. Just as in the previous section, the results of this section are exclusive to the geometric weights described in Definition \ref{def:GLPPc}. 

\subsection{Preliminaries}
Recall  the maps $T_c$ and $X_c$ defined in \eqref{eq:Tfunc}-\eqref{eq:K_func}.
We record the following facts about these maps in the following lemma.
\begin{lemma} \label{lem:TX_maps}
    The following hold.
    \begin{enumerate} [label=(\roman*), font=\normalfont]
        \item \label{it:T_fact} The restriction $T_c:(r_c,q^{-1}) \to \bigl(\f{qr_c}{1-qr_c},\infty\bigr)$ is a strictly increasing bijection, with inverse
\[
T_c^{-1}(\theta) =  \f{\theta}{q(1 + \theta)}.
\]
\item \label{it:X_fact} $X(\theta) \in (0,1]$ for all $\theta \in \R$. Furthermore, the restriction $X:\bigl[\f{qr_c}{1 -qr_c},\infty \bigr) \to (0,\xi_{\max}^c]$ is a strictly decreasing bijection, with inverse
\[
X_c^{-1}(\xi) = \f{q + q^2\sqrt{\xi}}{(1-q^2)\sqrt{\xi}}.
\]
    \end{enumerate}
\end{lemma}
\begin{proof}
    We first note that the function $x \mapsto \f{x}{1-x}$ is strictly increasing for $x \in (0,1)$, immediately giving \ref{it:T_fact}, along with a straightforward computation of the inverse.  

    To see that $X_c(\theta) \in (0,1]$, we have  $\f{qr_c}{1-qr_c} \ge \f{q}{1-q}$ because $r_c \ge 1$. Then, for $\theta \ge \f{qr_c}{1-qr_c}$, we have 
    \[
    \theta(1-q^2) - q^2 \ge \f{q(1-q^2)}{1-q} - q^2 = q(1+q) - q^2 = q,
    \]
    and so $0 < X_c(\theta) \le 1$ since $q \in (0,1)$. From here, it also follows that $X_c$ is strictly increasing on $\bigl[\f{qr_c}{1 -qr_c},\infty \bigr)$, and the computation of the inverse is straightforward. 
\end{proof}

\subsection{Maximizers and exit points}

The proof of the following is an elementary calculus exercise.
\begin{lemma} \label{lem:location_of_max}
    For $\theta \in \R$, consider the function $F_{c,\theta}:[0,1] \to \R$ defined by 
    \[
    F_{c,\theta}(\xi) = \theta \xi + \rho_{c}(1- \xi),
    \]
    where $ \rho_{c}$ is the function defined in \eqref{eq:mu_k_def}. Then, the following hold:
    \begin{enumerate} [label=(\roman*), font=\normalfont]
\item \label{it:unique_max} If $c \le 1$ or $\theta \neq \f{qc}{1-qc}$, then $F_{c,\theta}$ has a unique maximizer at $\xi = 1 - X_c(\theta)$.
\item \label{it:non_unique_max} If $c > 1$ and $\theta = \f{qc}{1-qc}$, then $\xi$ is a maximizer of $F_{c,\theta}$ if and only if $
\xi \in [0,1 - X_c(\theta)]$.
In particular, $F_{c,\theta}$ is constant on the interval $[0,1 - X_c(\theta)]$. 
\end{enumerate}
\end{lemma}

We now prove a result that tells us about the locations of maximizers from sloped initial conditions. 
\begin{proposition} \label{prop:converge_to_max}
    Let $\bigl(f_n\bigr)_{n \in \N}$ be a sequence of \rm{(}possibly random\rm{)} functions $f_n:\N \to \R$. Then,   
    \begin{enumerate} [label=(\roman*), font=\normalfont]
        \item \label{it:lbsup} If there exists $\theta \in \R$ such that, for each fixed $\xi \in [0,1]$, with probability one,
    \be \label{eq:liminf_ge}
    \liminf_{n \to \infty} \f{f_n(1 + \lfloor \xi n\rfloor)}{n} \ge \theta \xi,
    \ee
    then, for fixed real numbers $0 \le x< y \le 1$ and integers $k \ge t$, the following holds almost surely:
    \be \label{eq:liminf_sup}
    \liminf_{n \to \infty} \f{1}{n} \max_{\ell \in \llbracket 1 + \lfloor xn \rfloor,  \lfloor yn \rfloor + k + 1\rrbracket }\bigl[f_n(\ell) + G\bigl((-n - 1 + \ell,-n),(k,t))\bigr)\bigr] \ge \sup_{\xi \in [x,y]}[\theta \xi + \rho_{c}(1 - \xi)],
    \ee
     where the function $ \rho_{c}:[0,1] \to \R$ is defined in \eqref{eq:mu_k_def}. 
    \item \label{it:upsup} If, for some $\theta \ge 0$, the following holds almost surely:
    \be \label{eq:limsup_le}
    \limsup_{n \to \infty} \f{1}{n} \max_{\ell \in \llbracket 1,n \rrbracket }[f_n(\ell) - \theta \ell] = 0,
    \ee
    then for all real numbers $0 \le x < y \le 1$ and integers $k \ge t$, the following holds almost surely:
    \be \label{eq:limsup_sup}
    \limsup_{n \to \infty} \f{1}{n} \max_{\ell \in \llbracket 1 + \lfloor xn \rfloor, \lfloor yn \rfloor + k + 1 \rrbracket }\bigl[f_n(\ell) + G\bigl((-n - 1 + \ell,-n ),(k,t))\bigr)\bigr] \le \sup_{\xi \in [x,y]}[\theta \xi + \rho_{c}(1 - \xi)].
    \ee
    \item \label{it:mx_location} Assume $c \le 1$ or $\theta \neq \f{qc}{1-qc}$ and the sequence $f_n$ satisfies both conditions \eqref{eq:liminf_ge} and  \eqref{eq:limsup_le} almost surely. Define the new function $f_n':\Z_{\ge -n} \to \R$ by $f_n'(\ell) = f_n(\ell + n + 1)$. Then, for each $k \ge t$ and $\ve > 0$, there exists a random $N = N(k,t,\ve) \in \Z$ such that for all $n \in \Z_{\ge N}$, all maximizers of $\ell \mapsto f_n'(\ell) + G\bigl((\ell,-n),(k,t)\bigr)$ over $\ell \in \llbracket -n,k\rrbracket$ lie in the set 
    \be \label{eq:max_loc}
    \Bigl[\bigl(-X_c(\theta) - \ve)n, (-X_c(\theta)- \ve)n\bigr)\Bigr],
    \ee
    where the function $X$ is defined in \eqref{eq:K_func}.  In particular, for all $n \in \Z_{\ge N}$,
    \be \label{eq:Z_max_loc}
    Z_{f_n',-n}(k,t) \in  \Bigl[\bigl(-X_c(\theta)  - \ve)n, (-X_c(\theta) + \ve)n\bigr)\Bigr].
    \ee
    \end{enumerate}
\end{proposition}
\begin{proof}
By the diagonal shift invariance in Lemma \ref{lem:shift_sym}, it suffices to prove the statement for $t = 0$.

\medskip \noindent \textbf{Item \ref{it:lbsup}:} Observe that, for every $\xi \in [x,y]$ and $k \ge 0$,
\begin{align*}
\max_{\ell \in \llbracket 1 + \lfloor xn \rfloor, \lfloor yn \rfloor + k + 1 \rrbracket}&\bigl[f_n(\ell) + G\bigl((-n + \ell - 1,-n ),(k,0))\bigr)\bigr] \\
&\ge f_n(1 + \lfloor \xi n \rfloor) + G\bigl((-n + \lfloor \xi n \rfloor, - n )\bigr).
\end{align*}
Then, by the assumption \eqref{eq:liminf_ge} and Theorem \ref{thm:LLN}, with probability one, for all $\xi \in [x,y]\cap \Q$,
\[
\liminf_{n \to \infty} \f{1}{n} \max_{1 + \lfloor xn \rfloor \le \ell \le \lfloor yn \rfloor + k+1}\bigl[f_n(\ell) + G\bigl((-n + \ell - 1,-n),(k,0))\bigr)\bigr] \ge \theta \xi + \rho_{c}(1 - \xi),
\]
and therefore, using continuity of $ \rho_{c}$ in the last step below, we have
\begin{align*}
&\quad \, \liminf_{n \to \infty} \f{1}{n} \max_{\ell \in \llbracket 1 + \lfloor xn \rfloor, \lfloor yn \rfloor + k + 1 \rrbracket}\bigl[f_n(\ell) + G\bigl((-n + \ell - 1,-n + 1),(k,0))\bigr)\bigr]  \\
&\ge \sup_{\xi \in [x,y] \cap \Q}[\theta \xi + \rho_{c}(1 - \xi)] = \sup_{\xi \in [x,y]}[\theta \xi + \rho_{c}(1 - \xi)],
\end{align*}
thus proving \eqref{eq:liminf_sup}. 

\medskip \noindent \textbf{Item \ref{it:upsup}:} Choose $m \in \N$, and for $m \in \llbracket 0,m \rrbracket$, define 
\[
b_i = b_i(n,x,y,m) := \Biggl\lfloor \lfloor nx \rfloor + i \Bigl(\f{\lfloor ny \rfloor + k - \lfloor nx \rfloor}{m}\Bigr) \Biggr\rfloor + 1.
\]
Then,
\be \label{eq:fsup_upbd}
\begin{aligned}
&\quad \,\max _{\ell \in \llbracket 1 + \lfloor xn \rfloor, \lfloor yn \rfloor + k + 1 \rrbracket} [f_n(\ell) +  G\bigl((-n + \ell-1,-n + 1),(k,0))\bigr)] \\
&=\max_{i \in \llbracket 1,m \rrbracket} \max_{\ell \in \llbracket b_{i-1}, b_i \rrbracket}\bigl[f_n(\ell) - \theta \ell + \theta \ell + G\bigl((-n + \ell-1,-n),(k,0))\bigr)\bigr] \\
&\le \max_{i \in \llbracket 1,m\rrbracket }\Bigl[G\bigl((-n + b_{i-1} -1,-n),(k,0)\bigr) + \theta b_i + \max_{\ell \in \llbracket b_{i-1},b_i\rrbracket }[f_n(\ell) - \theta \ell] \Bigr] \\
&\le \max_{\ell \in \llbracket 1 + \lfloor nx \rfloor,\lfloor ny \rfloor + k + 1 \rrbracket}[f_n(\ell) - \theta \ell] + \max_{i \in \llbracket 1,m\rrbracket}\Bigl[G\bigl((-n + b_{i-1}-1,-n),(k,0)\bigr) + \theta b_i  \Bigr].
\end{aligned}
\ee
We handle each of the two terms on the last line above separately. First, the assumption \eqref{eq:limsup_le} gives us
\be \label{eq:term1}
\begin{aligned}
\limsup_{n \to \infty} \f{1}{n} \max_{\ell \in \llbracket 1 + \lfloor xn \rfloor, \lfloor yn \rfloor + k + 1 \rrbracket}[f(\ell) - \theta \ell]
&\le \limsup_{n \to \infty} \f{1}{n} \max_{\ell \in \llbracket 1, n+k+1 \rrbracket}[f(\ell) - \theta \ell]   \\\
&\le  2\limsup_{n \to \infty} \f{1}{2n} \max_{\ell \in \llbracket 1, 2n\rrbracket }[f(\ell) - \theta \ell] \le 0.
\end{aligned}
\ee
Next, for $i \in \llbracket 0,m \rrbracket$, define
\[
\xi_{i} := \lim_{n \to \infty} \f{b_{i}(n,x,y,m)}{n} = x + \f{y-x}{m}i.
\]
Then, by Theorem \ref{thm:LLN}, for $i \in \llbracket 0,m \rrbracket$, we have  
\[
\lim_{n \to \infty} \f{1}{n}G\bigl((-n + b_{i-1}-1,-n),(k,0)\bigr) = \rho_{c}(1 - \xi_{i-1}), \quad \text{a.s.}
\]
Combined with \eqref{eq:term1} and \eqref{eq:fsup_upbd}, we have 
\begin{align*}
&\quad \, \limsup_{n \to \infty} \f{1}{n}\max _{\ell \in \llbracket 1 + \lfloor xn \rfloor, \lfloor yn \rfloor + k+1 \rrbracket} [f_n(\ell) +  G\bigl((-n + \ell-1,-n),(k,0))\bigr)] \\ 
&\le \max_{i \in \llbracket 1,m \rrbracket }\bigl[\theta \xi_i + \rho_{c}(1 - \xi_i)\bigr] \\ 
&= \f{y-x}{m}\theta + \max_{i \in \llbracket 1,m \rrbracket}\bigl[\theta \xi_{i-1} + \rho_{c}(1 - \xi_{i-1})\bigr] \\
&\le \f{y-x}{m}\theta + \sup_{x \le \xi \le y}[\theta \xi+ \rho_{c}(1 - \xi)].
\end{align*}
This holds for all $m \in \N$, so sending $m \to \infty$ completes the proof of Item \ref{it:upsup}.

\medskip \noindent \textbf{Item \ref{it:mx_location}:} If the sequence $f_n$ satisfies both \eqref{eq:liminf_ge} and \eqref{eq:limsup_le}, then  the previous items show that for each $0 \le x < y \le 1$,
\be \label{eq:fG_sup_lim}
\lim_{n \to \infty} \f{1}{n} \max_{1 +\lfloor xn \rfloor \le \ell \le \lfloor yn \rfloor + k+1}\bigl[f_n(\ell) + G\bigl((-n + \ell - 1,-n),(k,0))\bigr)\bigr] = \sup_{\xi \in [x,y]}[\theta \xi + \rho_{c}(1 - \xi)],\quad\text{a.s.}
\ee
Since we assumed $c \le 1$ or $\theta \neq \f{qc}{1-qc}$, Lemma \ref{lem:location_of_max} implies that the function $\xi \mapsto \theta \xi + \rho_{c}(1- \xi)$ has a unique maximizer at $\xi = 1 - X_c(\theta)$, so by \eqref{eq:fG_sup_lim}, for each $\ve > 0$, there exists $N = N(k,\ve)$ so that when $n \ge N$,
\[
\max_{\ell \in I_n(\ve)}\bigl[f_n(\ell) + G\bigl((-n + \ell-1,-n),(k,0))\bigr)\bigr] > \max_{\ell \notin I_n(\ve)}\bigl[f_n(\ell) + G\bigl((-n + \ell-1,-n),(k,0))\bigr)\bigr],
\]
where $I_n(\ve) =  \Bigl[1 + \bigl(1-X_c(\theta)  - \ve)n,1 + (1 - X_c(\theta) + \ve)n\bigr)\Bigr] \cap \Z$. Then, \eqref{eq:max_loc} (and consequently \eqref{eq:Z_max_loc}) comes by shifting the index $\ell$ by $n+1$.
\end{proof}

\subsection{Verifying the slope assumptions for the known invariant measures} \label{sec:invmeas_slopes}

We now prove a lemma that shows that the assumptions in Proposition \ref{prop:converge_to_max} hold for sequences of functions sampled from the known invariant measures $\mu_{c,s}$.
\begin{lemma} \label{lem:stat_obeys_slope}
     Assume that $s \ge r_c$, but $s \neq 1$. Let $(f_n)_{n \ge 1}$ be any sequence of functions $\N \to \R$ such that $f_n \sim \mu_{c,s}$ for all $n \ge 1$. Set $\theta = T_c(s)$. Then, with probability one,
    \be \label{eq:lim123}
    \lim_{n \to \infty} \f{1}{n} \max_{k \in \llbracket 1,n \rrbracket}[|f_n(k) -\theta k|] = 0.
    \ee
    In particular, for each $\xi \in [0,1]$, we have 
    \be \label{eq:lim124}
    \lim_{n \to \infty} \f{f_n(1 +\lfloor \xi n \rfloor)}{n} = \theta\xi.
    \ee
\end{lemma}
\begin{proof}
We get \eqref{eq:lim124} from \eqref{eq:lim123} because
\[
\f{1}{n}|f_n(1 + \lfloor \xi n \rfloor) - \theta (\lfloor \xi n \rfloor + 1)| \le \f{1}{n}\max_{k \in \llbracket 1,n \rrbracket}[|f_n(k) -\theta k|],
\]
so we turn to proving \eqref{eq:lim123}. We claim that, for each $\ve > 0$, there exist constants $A > 0$ and $0 < \alpha < \beta$ (depending on $\ve$) such that, for each $k \in \llbracket 1,n \rrbracket$,
\be \label{eq:154}
\Pp\Bigl(|f_n(k) - \theta k| \ge \ve n\Bigr) \le Ae^{\alpha k - \beta n}.
\ee
Before proving \eqref{eq:154}, we show why this implies the almost sure convergence in \eqref{eq:lim123}. We use a union bound to get 
\[
\Pp\Bigl( \max_{k \in \llbracket 1,n \rrbracket}[|f_n(k) -\theta k|] \ge \ve n  \Bigr) \le \sum_{k = 1}^n \Pp(|f_n(k) - \theta k| > \ve n) \le An e^{(\alpha - \beta)n}.
\] 
Since $\alpha < \beta$, this bound is summable, and a Borel-Cantelli argument completes the proof of \eqref{eq:lim123}.

We turn to proving the bound \eqref{eq:154}. By definition of the measure $\mu_{c,s}$ (Definition \ref{def:mucs}), each $f_n$ has the law of
\be \label{eq:fk_def_repeat}
k \mapsto  f(k) := S_2(k) + \Bigl(\max_{\ell \in \llbracket 1,k \rrbracket} [S_1(\ell) - S_2(\ell - 1)] - Y\Bigr)^+,
 \ee
where $S_1$ is a $\Geo(qs)$ random walk, $S_1$ is a $\Geo(qs^{-1})$ random walk, and $Y \sim \Geo(cs^{-1})$ (a single geometric random variable). For the rest of the proof, since we are seeking to prove a probability bound \eqref{eq:154}, we use $f$ generically instead of $f_n$. If $s = c > 1$, we have $f(k) = S_2(k)$, and the bound follows immediately from Lemma \ref{lem:Geo_RW_LD}. We may thus assume that $s > c\vee 1$.  We now prove an intermediate lemma.

\begin{lemma} \label{lem:f_fromS_bound}
Let $k \in \N$, $s > 1$, and $\ve > 0$. Let $n \in \Z_{\ge k}$ be chosen to be sufficiently large so that $\f{q}{s-q} \le \f{\ve}{2}n$, and assume that the following hold:
\begin{enumerate} [label=(\roman*), font=\normalfont]
    \item $|S_1(\ell) - \f{qs}{1-qs}\ell| \le \f{\ve}{4}n,\quad\text{for all }\ell \in \llbracket 1,k \rrbracket$.
    \item $|S_2(\ell) - \f{q}{s-q}\ell| \le \f{\ve}{4}n,\quad\text{for all } \ell \in \llbracket 1,k \rrbracket$.
    \item $Y \le \f{\ve}{4}n$. 
\end{enumerate}
Then, we have
\[
\Bigl|f(k) - \f{qs}{1-qs}k\Bigr| \le \ve n,
\]
where $f$ is defined in \eqref{eq:fk_def_repeat}. 
\end{lemma}
\begin{proof}
    We note that since $s > 1$, we have 
    \[
    T_c(s) = \f{qs}{1-qs} > \f{q}{s-q}.
    \]
    Then, under the assumptions on $S_1$ and $S_2$ and since $\f{q}{s-q} \le \f{\ve}{2}n$, for each $1 \le k \le n$, we have 
\begin{align*}
f(k) &\le \f{q}{s-q}k + \max_{1 \le \ell \le k}\Bigl[\f{qs}{1-qs}\ell - \f{q}{s-q}(\ell - 1)\Bigr] + \f{\ve}{2}n \\
&= \f{q}{s-q} + \f{qs}{s-q}k + \f{\ve}{2}n \le \f{qs}{s-q}k + \ve n.
\end{align*}
We turn to proving a lower bound. Under the assumptions on $S_1,S_2,Y$, we have 
\be \label{eq:fklb}
\begin{aligned}
    f(k)&=S_2(k) + \Bigl(\max_{1 \le \ell \le k} [S_1(\ell) - S_2(\ell - 1)] - Y\Bigr)^+ \\&\ge \f{q}{s-q}k - \f{\ve}{4}n +  \Biggl(\max_{1 \le \ell \le k}\Bigl[\f{qs}{1-qs}\ell - \f{q}{s-q}(\ell - 1)\Bigr] -\f{3\ve}{4}n \Biggr)^+ \\
    &\ge \f{q}{s-q}k - \f{\ve}{4}n +\Biggl(\Bigl(\f{qs}{1-qs} - \f{q}{s-q}\Bigr)k - \f{3\ve}{4}n \Biggr)^+,
\end{aligned}
\ee
where in the last step, we have discarded the term $\f{q}{s-q}$ at the cost of an inequality.  We now consider two cases:

\medskip \noindent \textbf{Case 1:} $\Bigl(\f{qs}{1-qs} - \f{q}{s-q}\Bigr)k > \f{3\ve}{4}n$. Then, from \eqref{eq:fklb}, we have $
f(k)  \ge \f{qs}{1-qs}k - \ve n$. 

\medskip \noindent \textbf{Case 2:} $\Bigl(\f{qs}{1-qs} - \f{q}{s-q}\Bigr)k \le \f{3\ve}{4}n$. Then, from \eqref{eq:fklb}, we have 
\[
    f(k) \ge \f{q}{s-q}k - \f{\ve}{4}n = \f{qs}{1-qs}k +\Bigl(\f{q}{s-q} - \f{qs}{1-qs}\Bigr)k - \f{\ve}{4}n  \ge  \f{qs}{1-qs}k - \ve n. \qedhere
\]
\end{proof}
Returning to the proof of \eqref{eq:154} to complete the proof of Lemma \ref{lem:stat_obeys_slope}, recall that $T_c(s) = \f{qs}{1-qs}$ since $s > 1$ by definition  of $T_c$ \eqref{eq:Tfunc}. Now, by Lemma \ref{lem:f_fromS_bound} and a union bound, for all $n$ such that $\f{q}{s-q} \le \f{\ve}{2}n$, and all $k \in \llbracket 1,n \rrbracket$,
\begin{align} \label{eq:fk_union_bd}
    \Pp(|f(k) - \theta k| \ge \ve n) \le \sum_{i = 1}^2\sum_{\ell = 1}^k \Pp\Bigl(\Bigl|S_i(\ell) - \f{qs}{1-qs}\ell\Bigr| \ge \f{\ve}{4}n \Bigr) + \Pp\Bigl(Y \ge \f{\ve}{4}n\Bigr).
\end{align}
By Lemma \ref{lem:Geo_RW_LD} and the exponential tails of the geometric distribution, there exist constants $0 < \alpha < \beta$ and $A > 0$ so that the probability in \eqref{eq:fk_union_bd} is bounded from above by 
\[
2\sum_{\ell = 1}^k e^{\alpha\ell - \beta n} + e^{-\beta n} \le A e^{\alpha k - \beta n},
\]
thus completing the proof of \eqref{eq:154}. 
\end{proof}

We now complete the proof of Lemma \ref{lem:inv_meas_slopes}.
\begin{proof}[Proof of Lemma \ref{lem:inv_meas_slopes}]
    Recall that we are showing that, if $f \sim \mu_{c,s}$, then 
    \[
    \lim_{k \to \infty} \f{f(k)}{k} = T_c(s).
    \]
    If $s \ge r_c$, but $s \neq 1$, then Lemma \ref{lem:stat_obeys_slope} (specifically Equation \ref{eq:lim124} in the case $\xi = 1$) gives a stronger statement.

    The last case to consider is where $c \le 1$ and $s =1$. Using Definition \ref{def:mucs}, $f \sim \mu_{c,1}$ has the law of 
    \[
    f(k) = S_2(k) + \Bigl(\max_{\ell \in \llbracket 1,k \rrbracket} [S_1(\ell) - S_2(\ell - 1)] - Y\Bigr)^+,
    \]
    where $S_1$ and $S_2$ are independent $\Geo(q)$ random walks, independent of $W \sim \Geo(c)$. Since $T_c(1) = \f{q}{1-q}$ in this case (recall \eqref{eq:Tfunc}), and this is the same as the mean of the increments of $S_2$, it suffices to show that 
    \[
    \lim_{k \to \infty} \f{f(k) - S_2(k)}{k}  = 0.
    \]
    For any $\ve > 0$, there exists a constant $A > 0$ so that $|S_i(\ell) - T_c(1)\ell| \le A + \ve \ell$ for all $\ell \in \N$ and $i \in \{1,2\}$. Then, 
    \[
    \max_{\ell \in \llbracket 1,k \rrbracket} [S_1(\ell) - S_2(\ell - 1)] \le 2A + \ve +\max_{\ell \in \llbracket 1,k \rrbracket}[2\ve \ell ] = 2A + \ve + 2\ve k.
    \]
    Thus, for any $\ve > 0$,
    \[
    0 \le \liminf_{k \to \infty} \f{f(k) - S_2(k)}{k} \le \limsup_{k\to \infty} \f{f(k) - S_2(k)}{k} \le   \limsup_{k \to\infty} \f{1}{k}\max_{\ell \in \llbracket 1,k \rrbracket} [S_1(\ell) - S_2(\ell - 1)] \le 2\ve.
    \]
    The result follows by sending $\ve \searrow 0$. 
\end{proof}

\section{Busemann limits and the Busemann process} \label{sec:Busemann}
In this section, we construct the Busemann process, culminating in Proposition \ref{prop:full_Busemann}. We first prove some intermediate lemmas.

\begin{lemma} \label{lem:mu_continuous}
    The function $s \mapsto \mu_{c,s}$ over $s \in [r_c,q^{-1})$ is weakly continuous in the space of probability measures on $\R^{\N}$. Furthermore, the measure escapes to $\infty$ as $s \to q^{-1}$. That is, for any $A \in (0,\infty)$ and any integers $k > m \ge 0$, 
    \be \label{eq:to_infty}
    \lim_{s \nearrow q^{-1}} \mu_{c,s}\Bigl(f(k) - f(m) \ge A \Bigr) = 1.
    \ee
\end{lemma}
\begin{proof}
    Couple the random walks $S_{1,s}$ and $S_{2,s}$ and a random variable $Y_s$  for $r_c \le s < q^{-1}$   so that $S_2$ is a $\Geo(\f{q}{s})$ random walk, $S_1$ is a $\Geo(qs)$ random walk, and $Y_s \sim \Geo(\f{c}{s})$ (with $Y_c = \infty$). Couple these in such a way that $\lim_{u \to s} S_{j,u}(\ell) = S_{j,s}(\ell)$ a.s. for $j = \{1,2\}$ and $\ell \in \N$  and so that $\lim_{u \to s}Y_u = Y_s$ a.s. For $s\ge 1$,  $s > c$,  and $k \in \Z_{\ge 1}$, define
    \[
    f_{s}(k) = S_{2,s}(k) + \Bigl(\max_{\ell \in \llbracket 1,k \rrbracket} [S_{1,s}(\ell) - S_{2,s}(\ell - 1)] - Y_s\Bigr)^+.
    \]
    Then, we see that $f_u(k) \to f_s(k)$ a.s. as $u \to s$ when $s \ge 1$ and $s > c$. On the other hand, if $c \ge 1$, as $s \searrow c$, since $Y_s \to \infty$, we have $f_s(k) \to S_{2,s}(k)$ a.s. Since $S_2 \sim \mu_{c,c}$, this completes the proof of the weak continuity. 
    
    To show \eqref{eq:to_infty}, we observe that as $s \nearrow q^{-1}$, $S_{1,s}$ is a geometric random walk with diverging mean. Thus, for $s$ close to $q^{-1}$, with high probability, we have
    \begin{align*}
    &\quad \; f_s(k) - f_s(m) \\
    &= S_{2,s}(k) - S_{2,s}(m) + \max_{\ell \in \llbracket 1,k \rrbracket} [S_{1,s}(\ell) - S_{2,s}(\ell - 1)] - \max_{\ell \in \llbracket 1,k \rrbracket} [S_{1,s}(\ell) - S_{2,s}(\ell - 1)],
    \end{align*}
    and with high probability, the maximums above will be achieved at $\ell = k$ and $\ell = m$, respectively, so this is 
    \begin{align*}
    &\quad \;S_{2,s}(k) - S_{2,s}(m) + S_{1,s}(k) - S_{2,s}(k-1) - (S_{1,s}(m) - S_{2,s}(m-1)) \\
    &=S_{1,s}(k) - S_{1,s}(m) + \Bigl(S_{2,s}(k) - S_{2,s}(k-1)\Bigr) - \Bigl(S_{2,s}(m) - S_{2,s}(m-1)\Bigr).
    \end{align*}
   Each of the terms in parenthesis are $\Geo(qs^{-1})$ random variables, which are converging in distribution to $\Geo(q^2)$ random variables as $s \nearrow q^{-1}$. However, the $S_{1,s}(k) - S_{1,s}(m)$ term is large with high probability, completing the proof.  
\end{proof}

The next intermediate result is a precursor to the one force--one solution principle. Given eternal solutions with the known invariant distributions, it allows us to compare to the increments of passage times with a general class of initial conditions. 
\begin{lemma} \label{lem:comp_to_eternal}
    Let $\zeta < \xi < \eta$ each lie in the interval $(0,\xi_{\max}^c)$. Let $b_\zeta$ and $b_\eta$ be eternal solutions for $\GLPPc$ such that for $\varphi \in \{\zeta,\eta\}$ and each $n \in \Z$,
    \be \label{eq:distribution_assumption}
    \bigl(b_\varphi(n + k,n) - b_\varphi(n,n)\bigr)_{k \in \N} \sim \mu_{c,s_{\varphi}},\quad\text{where}\quad s_\varphi = (X \circ T)^{-1}(\varphi).
    \ee
    Then, the following hold with probability one:
    \begin{enumerate} [label=(\roman*), font=\normalfont]
        \item \label{it:upbd1} If $\mbf v_p = (i_p,j_p)$ for $p \in \N$ is any sequence satisfying $j_p \to -\infty$ and $\liminf_{p \to \infty}\f{i_p}{j_p} \ge \xi$, then for any integers $\ell \ge m \ge t$, there exists $P \in \N$ so that for all integers $p \ge P$,
        \be \label{eq:bd1}
        G\bigl(\mbf v_p,(\ell,t)\bigr) - G\bigl(\mbf v_p,(m,t)\bigr) \le b_\zeta(\ell,t) - b_\zeta(m,t).
        \ee
        Moreover, for all $n \ge -t$,
        \be \label{eq:bd2}
        G\bigl((-n,-n),(\ell,t)\bigr) - G\bigl((-n,-n),(m,t)\bigr) \le b_\zeta(\ell,t) - b_\zeta(m,t).
        \ee
        \item \label{it:lbd1} If $\mbf v_p = (i_p,j_p)$ for $p \in \N$ is any sequence satisfying $j_p \to -\infty$ and $\limsup_{p \to \infty}\f{i_p}{j_p} \le \xi$, then for any integers $\ell \ge m \ge t$, there exists $P \in \N$ so that for $p \ge P$,
        \[
        G\bigl(\mbf v_p,(\ell,t)\bigr) - G\bigl(\mbf v_p,(m,t)\bigr) \ge b_\eta(\ell,t) - b_\eta(m,t).
        \]
        \item \label{it:both_bd1} If $(f_n)_{n \in \N}$ is a sequence of functions $\Z_{\ge -n} \to \R$ satisfying 
        \be \label{eq:bd_assumption_1}
        \lim_{n \to \infty} \f{1}{n} \max_{k \in \llbracket 1,n \rrbracket}[|f_n(k - n - 1) - X_c^{-1}(\xi)k|] = 0,
        \ee
        then for any integers $\ell \ge m \ge t$, there exists $N \in \Z_{\ge -t}$ so that for all $n \ge N$,
        \[
        b_\eta(\ell,t) - b_\eta(m,t) \le G_{f_n,-n}(\ell,t) - G_{f_n,-n}(m,t) \le b_\zeta(\ell,t) - b_\zeta(m,t).
        \]
        \item \label{it:lbd2} If $(f_n)_{n \in \N}$ is a sequence of functions $f_n:\Z_{\ge -n} \to \R$  such that $\liminf_{n \to \infty} \f{f_n(-n)}{n} \ge 0$, and
        \be \label{eq:bd_assumption_2}
        \limsup_{n \to \infty} \f{1}{n} \max_{k \in \llbracket 1,n \rrbracket}[f_n(k - n-1) - X_c^{-1}(\xi_{\max}^c)k] \le 0,
        \ee
        then for any integers $\ell \ge m \ge t$, there exists $N \ge -t$ so that  for all $n \ge N$,
        \[
        G_{f_n,-n}(\ell,t) - G_{f_n,-n}(m,t) \le b_\eta(\ell,t) - b_\eta(m,t).
        \]
    \end{enumerate}
\end{lemma}
\begin{proof}
    \noindent Since $b_\varphi$ is an eternal solution for $\varphi \in \{\zeta,\eta\}$, Corollary \ref{cor:eternal_soln_recursion} implies that for all integers $k \ge t$ and $n \ge -t$, we have 
\[
 b_\varphi(k,t) = G_{b_\varphi(\aabullet,-n-1),-n} = \max_{i \in \llbracket -n,k \rrbracket} \bigl[b_\varphi(i,-n-1) + G\bigl((i,-n),(k,t)\bigr)\bigr].
\]
and thus, by applying this for $k \in \{m,\ell\}$ and adding and subtracting $b_\varphi(-n-1,-n-1)$, we have
\be \label{eq:b_zeta_diff}
b_\varphi(\ell,t) - b_\varphi(m,t) =  G_{b_\varphi(\aabullet,-n-1)- b_\varphi(-n-1,-n-1),-n}(\ell,t) -G_{b_\varphi(\aabullet,-n-1)- b_\varphi(-n-1,-n-1),-n}(m,t).
\ee
Since $\varphi  < \xi_{\max}^c$, Lemma \ref{lem:TX_maps} imply that $s_\varphi = (X_c \circ T_c)^{-1}(\varphi) > r_c$, so the assumption \eqref{eq:distribution_assumption} combined with Lemma \ref{lem:stat_obeys_slope} implies that the sequence of initial conditions 
\[
b_\varphi(-n-1 + \aabullet,-n-1) - b_\varphi(-n-1,-n-1)
\]
satisfies the conditions \eqref{eq:liminf_ge} and \eqref{eq:limsup_le} for $\theta = X_c^{-1}(\varphi)$. Then, by Proposition \ref{prop:converge_to_max}\ref{it:mx_location}, for all integers $\ell \ge m \ge t$, there exists $N = N(m,\ell,t,\ve) \ge -t$ so that for all $n \ge N$, $k \in \{m,\ell\}$, and $\varphi \in \{\zeta,\eta\}$, 
\be \label{578}
Z_{\wt b_\varphi(\aabullet,-n-1)- \wt b_\varphi(-n-1,-n-1),-n}(k,t) \in \bigl[(-\varphi -\ve)n,(-\varphi + \ve)n\bigr] \cap \llbracket -n, k\rrbracket.
\ee
Then, Items \ref{it:upbd1}-\ref{it:lbd1} follow from the assumption $j_p \to -\infty$, Equation \eqref{eq:b_zeta_diff}, and Lemma \ref{lem:exit_pt_comp_ptp}. 

For Item \ref{it:both_bd1}, the assumption \eqref{eq:bd_assumption_1} and Lemma \ref{prop:converge_to_max}\ref{it:mx_location} imply that for all integers $\ell \ge m \ge t$, there exists $N' = N'(m,\ell,t,\ve) \ge -t$ so that for $n \ge N$ and $k \in \{m,\ell\}$,
\be \label{579}
Z_{f_n,-n}(k,t) \in \bigl[(-\xi -\ve)n,(-\xi + \ve)n\bigr] \cap \llbracket-n,k\rrbracket.
\ee
Then, by choosing $\ve > 0$ sufficiently small, we have for all sufficiently large $n$,
\begin{align*}
Z_{\wt b_\eta(\aabullet,-n-1)- \wt b_\eta(-n-1,-n-1),-n}(\ell,t)  &< Z_{f_n,-n}(m,t) \\
&\le Z_{f_n,-n}(\ell,t) < Z_{\wt b_\zeta(\aabullet,-n-1)- \wt b_\zeta(-n-1,-n-1),-n}(m,t).
\end{align*}
Then, using Equation \eqref{eq:b_zeta_diff} and Lemma \ref{lem:exit_pt_comp_ptp}, we obtain Item \ref{it:both_bd1}.

The proof of Item \ref{it:lbd2} will follow similarly by comparison with \eqref{578} once we show that for each $\ve > 0$, and all sufficiently large $n$, we have 
\be \label{642}
Z_{f_n,-n}(\ell,t) \in \bigl[-n, (-\eta - \ve)n\bigr] \cap \llbracket -n, k\rrbracket.
\ee
By the assumption \eqref{eq:bd_assumption_2}, and Proposition \ref{prop:converge_to_max}\ref{it:upsup},
\be \label{598}
\limsup_{n \to \infty}  \f{1}{n} \max_{\ell \in \llbracket \lfloor(-\eta - \ve)n\rfloor,k\rrbracket }\bigl[f_n(\ell) + G\bigl(( \ell,-n),(k,t)\bigr)\bigr] \le \sup_{\xi \in [1 - \eta - \ve,1]}[X^{-1}_c(\xi_{\max}^c) \xi + \rho_{c}(1-\xi)].
\ee
On the other hand, we have 
\be \label{eq:rho1_big}
\begin{aligned}
&\quad \; \liminf_{n \to \infty} \f{1}{n} \max_{\ell \in \llbracket -n, \lfloor(-\eta - \ve)\rfloor n \rrbracket}\bigl[f_n(\ell) + G\bigl(( \ell,-n),(k,t)\bigr)\bigr] \\
&\ge \liminf_{n \to \infty}\f{1}{n} \bigl[f_n(-n) + G\bigl((-n,-n),(k,t)\bigr)\bigr]  \\
&\ge \liminf_{n \to \infty}\f{1}{n} \bigl[f_n(-n) + G\bigl((-n,-n),(t,t)\bigr)\bigr]  \\
&\ge \rho_{c}(1) > \sup_{\xi \in [1 - \eta - \ve,1]}[X^{-1}_c(\xi_{\max}^c) \xi + \rho_{c}(1-\xi)].
\end{aligned}
\ee
The second inequality came by Lemma \ref{lem:G_monotonicity}, the third inequality follows by the assumption on $f_n(-n)$ and Proposition \ref{prop:boundary_shape},
and the final inequality holds for sufficiently small $\ve > 0$ by Lemma \ref{lem:location_of_max}. Specifically, Lemma \ref{lem:TX_maps}\ref{it:X_fact} implies that $X_c^{-1}(\xi_{\max}^c) = \f{qr}{1-qr}$, so in the case that $c \le 1$, we have $r = 1$, and Lemma \ref{lem:location_of_max}\ref{it:unique_max} implies that  the function $\xi \mapsto X_c^{-1}(\xi_{\max}^c)\xi + \rho_{c}(1-\xi)$ has a unique maximizer at $\xi = 1 - \xi_{\max}^c = 0$ (the last equality by \eqref{Wxi_and_xi_max}). Since $\eta < \xi_{\max}^c = 1$ by assumption, we have $1 - \eta - \ve > 0$ for sufficiently small $\ve > 0$, so we get the final strict inequality of \eqref{eq:rho1_big} by setting $\xi = 0$. On the other hand, if $c > 1$, then Lemma \ref{lem:location_of_max}\ref{it:non_unique_max} implies that the maximizers of $\xi \mapsto X_c^{-1}(\xi_{\max}^c)\xi + \rho_{c}(1-\xi)$ are exactly the $\xi$ lying in the interval $[0,1 - \xi_{\max}^c]$, and since $\eta < \xi_{\max}^c$, we have that $1-\eta - \ve > 1-\xi_{\max}^c$ for sufficiently small $\ve > 0$, giving the final strict inequality in \eqref{eq:rho1_big} in this case. Comparison of \eqref{eq:rho1_big} with \eqref{598} yields \eqref{642}, as desired. 
\end{proof}

 Our final preliminary result shows the existence of Busemann limits for a fixed direction parameter $\xi$ away from the boundary. 
\begin{proposition} \label{prop:Buse_limits_1}
     Fix  $\xi \in (0,\xi_{\max}^c)$. Then, as $n \to \infty$, when $\mbf x,\mbf y \in \ZHS$, the following limit exists on a $\xi$-dependent event of probability one:
\be \label{eq:def_of_W}
W_\xi^c(\mbf x,\mbf y) := \lim_{n \to \infty} G\bigl((-\lfloor \xi n\rfloor,-n),\mbf y\bigr) - G\bigl((-\lfloor \xi n\rfloor,-n),\mbf x\bigr)
\ee
Furthermore, the family of limits $W_\xi^c(\mbf x,\mbf y)$  has the following properties.  
\begin{enumerate} [label=(\roman*), font=\normalfont]
\item \label{it:Buse_eternal} For any fixed $\mbf x \in \ZHS$, $(m,n) \mapsto W_\xi^c\bigl(\mbf x,(m,n)\bigr)$ is an eternal solution.
\item \label{it:Buse_h_law} For each $t \in \Z$, the process 
    $\bigl(W_\xi^c((t,t),(k + t,t))\bigr)_{k \in \N}$ has law $\mu_{c,s}$, where $s = (X_c \circ T_c)^{-1}(\xi)$.
    \item \label{it:prelim_Buse_monotonicity} For $\zeta < \xi$, $t \in \Z$, and $k \in \Z_{\ge t}$, we have 
    \begin{align}
    &W_\xi^c\bigl((k,t),(k+1,t)\bigr) \le  W_\zeta^c\bigl((k,t)(k+1,t)\bigr),\qquad\text{and} \label{eq:h_mont} \\
    &W_\xi^c\bigl((k,t),(k,t+1)\bigr) \ge W_\zeta^c\bigl((k,t),(k,t+1)\bigr). \label{eq:v_mont}
    \end{align}
\end{enumerate}
\end{proposition}
\begin{proof}
Given the existence of the limits, Item \ref{it:Buse_eternal} follows immediately from the recursive relations for $G$ \eqref{eq:G_rec_1}. We now prove the existence of the limits together with Items \ref{it:Buse_h_law}-\ref{it:prelim_Buse_monotonicity}.

We introduce the shorthand notation 
\[
\mbf x_{n}^\xi = (-\lfloor \xi n \rfloor,-n)
\]
We first handle the case where $\mbf y = (k,t)$ and $\mbf x = (m,t)$ for $t \in \Z$ then later prove it for general $\mbf x$ and $\mbf y$.  

In this case, by the diagonal shift invariance in Lemma \ref{lem:shift_sym}, it suffices to take $t = 0$. 
Furthermore, by writing
    \begin{align*}
    &\quad \;G(\mbf x_n^\xi,(k,0)) - G(\mbf x_n^\xi,(m,0))= G(\mbf x_n^\xi,(k,0)) - G(\mbf x_n^\xi,(0,0)) - \Bigl(G(\mbf x_n^\xi,(m,0)) - G(\mbf x_n^\xi,(0,0))\Bigr),
    \end{align*}
    we see that it suffices to prove existence of the limit for the case $m = 0$.

    By Lemma \ref{lem:eternal_construction}, on an appropriate probability space $(\wt \Omega, \wt{\mathcal F},\wt \Pp)$ with weights $(\wt \omega_{\mbf x})_{\mbf x \in \ZHS}$, there exists a family of eternal solutions  \[
    (\wt b_{\zeta})_{\zeta \in (0,\xi_{\max}^c)  \cap \Q}
    \]
    such that, for each $\zeta \in (0,\xi_{\max}^c)$ and $j \in \Z$, we have 
\be \label{eq:bs_law}
\bigl(\wt b_\zeta(j+\ell,j) - \wt b_\zeta(j,j)\bigr)_{\ell \in \N} \sim \mu_{c,s_\zeta},
\ee
where we define $
s_\zeta = (X_c \circ T_c)^{-1}(\zeta)$. By Lemma \ref{lem:TX_maps}, $s_\zeta \in (r_c,\infty)$.
Let $\wt G$ denote last-passage percolation with respect to these weights $\wt \omega$. Since the weights $\wt \omega$ are equal in distribution to the original weights $\omega$, it suffices to prove the existence of the almost sure limits on the right-hand side in \eqref{eq:def_of_W} for $\wt G$ on this new probability space.

Then, by Lemma \ref{lem:comp_to_eternal}\ref{it:upbd1}-\ref{it:lbd1}, for $\zeta < \xi < \eta$ with $\zeta,\eta \in (0,\xi_{\max}^c) \cap \Q$ and all sufficiently large $n$,
\be \label{eq:Buse_bounds}
\wt b_\eta(k+1,0) - \wt b_\eta(k,0) \le \wt G(\mbf x_n^\xi,(k+1,0))  - \wt G(\mbf x_n^\xi,(k,0))\le \wt b_\zeta(k+1,0) - \wt b_\zeta(k,0).
\ee
In particular, \eqref{eq:Buse_bounds} implies that $\wt b_\eta(k+1,0) - \wt b_\eta(k,0) \le \wt b_\zeta(k+1,0) - \wt b_\zeta(k,0)$ whenever $\eta > \zeta$. We use this monotonicity and a telescoping sum to create a process
\[
\Bigl(\wt W_\zeta^c(k): k \in \N,  \zeta \in (0,\xi_{\max}^c) \Bigr)
\]
defined by 
\[
\wt W_\zeta^c(k) := \lim_{\Q \ni \eta \nearrow \zeta} [\wt b_\eta(k,0) - \wt b_\zeta(0,0)].
\]
By weak continuity of the measures $\mu_{c,s}$ in $s$ (Lemma \ref{lem:mu_continuous}) and continuity of the map $(X_c \circ T_c)^{-1}$ on $(0,\xi_{\max}^c)$ (Lemma \ref{lem:TX_maps}),  for $\zeta \in (0,\xi_{\max}^c)$, we have $(\wt W_\zeta^c(k): k \ge 0) \sim \mu_{c,s_\zeta}$. Furthermore, by monotonicity, we have
\[
\lim_{\eta \searrow \zeta} \wt W_\eta^c(k) \le \wt W_\zeta^c(k),
\]
but again by continuity of the measures $\mu_{c,s}$, both sides above have the same law and thus are equal on a $\zeta$-dependent event of probability one. 

Hence, by first taking a limit in $n$ in \eqref{eq:Buse_bounds}, then a limit as $\eta \searrow \xi$ and $\zeta \nearrow \xi$, we have the following $\wt \Pp$-almost sure limit on a $\xi$-dependent full-probability event:
\[
\lim_{n \to \infty} [\wt G\bigl(\mbf x_n^\xi,(k,0)\bigr) - \wt G\bigl(\mbf x_n^\xi,(0,0)\bigr)] = \wt W_\xi^c(k).
\]
Thus, the existence of the $\Pp$-almost sure limit also holds for $G$. Along the way, we have also established the law in Item \ref{it:Buse_h_law}, and the monotonicity in Equation \eqref{eq:h_mont} of Item \ref{it:prelim_Buse_monotonicity}. 

Now, for general $\mbf x,\mbf y$, since we have shown that the limits hold for $\mbf x = (m,t)$ and $\mbf y = (k,t)$ when $k,m$ are integers greater than an integer $t$, by writing the prelimiting difference in \eqref{eq:def_of_W}, it suffices to show that the limits exist for $\mbf y = (k,t+1)$ and $\mbf x = (k,t)$ when $k \ge t+1$ are integers. We do this for a fixed $t \in \Z$ by induction on $k$, starting from the $k = t+1$ case. In this case, by the recursion for $G$ \eqref{eq:G_rec_1}, for all $n \in \Z_{\ge - t}$,
\[
G(\mbf x_n^\xi,(t+1,t+1)) - G(\mbf x_n^\xi,(t+1,t))= \omega_{(t,t)},
\]
so the limit exists (in fact, the difference is constant in $n$). 

Now, assume, by way of induction, that for some integer $k \ge t+1$,
\[
W_\xi^c((k,t),(k,t+1)) := \lim_{n \to \infty} [G(\mbf x_n^\xi,(k,t+1)) - G(\mbf x_n^\xi,(k,t))]
\]
exists almost surely, and that for $\zeta < \xi$, we have the monotonicity in \eqref{eq:v_mont}.  

Then, 
\be \label{eq:G_dif1}
\begin{aligned}
&\quad \; G(\mbf x_n^\xi,(k+1,t+1)) - G(\mbf x_n^\xi,(k+1,t))  \\
&= \omega_{(k+1,t+1)} + G(\mbf x_n^\xi,(k,t+1)) \vee G(\mbf x_n^\xi,(k+1,t)) - G(\mbf x_n^\xi,(k+1,t)) \\
&= \omega_{(k+1,t+1)} + \bigl(G(\mbf x_n^\xi,(k,t+1)) - G(\mbf x_n^\xi,(k+1,t))\bigr)^+ \\
&= \omega_{(k+1,t+1)} + \Bigl(G(\mbf x_n^\xi,(k,t+1)) - G(\mbf x_n^\xi,(k,t)) - (G(\mbf x_n^\xi,(k+1,t)) - G(\mbf x_n^\xi,(k,t))\Bigr)^+.
\end{aligned}
\ee
By the induction assumption and what we proved above, the limits  
\[
\lim_{n \to \infty} [G\bigl(\mbf x_n^\xi,(k,t+1)\bigr) - G\bigl(\mbf x_n^\xi,(k,t)\bigr)]\quad\text{and}\quad \lim_{n \to \infty}[G\bigl(\mbf x_n^\xi,(k+1,t)) - G\bigl(\mbf x_n^\xi,(k,t)\bigr)]
\]
both exist almost surely, so the limit of the quantity in \eqref{eq:G_dif1} exist almost surely. Furthermore, for $\xi > \zeta$, these almost sure limits satisfy
\begin{align*}
&\quad \;W_\xi^c\bigl((k+1,t),(k+1,t+1)\bigr) \\&= \omega_{(k+1,t+1)} + \Bigl(W_\xi^c\bigl((k,t),(k,t+1)\bigr) - W_\xi^c\bigl((k,t),(k+1,t)\bigr)\Bigr)^+  \\
&\ge \omega_{(k+1,t+1)} + \Bigl(W_\zeta^c\bigl((k,t),(k,t+1)\bigr) - W_\zeta^c\bigl((k,t),(k+1,t)\bigr)\Bigr)^+ \\
&=W_\zeta^c\bigl((k+1,t),(k+1,t+1)\bigr),
\end{align*}
where we have used the monotonicity in \eqref{eq:h_mont} and the induction assumption. 
\end{proof}

We now extend the existence of Busemann functions to a Busemann process, defined simultaneously for all parameters.
\begin{proposition} \label{prop:full_Busemann}
     On the space $(\Omega,\mathcal F,\Pp)$ there exists a process
    \[
    \Bigl(W_\xi^c(\mbf x,\mbf y): \xi \in (0,\xi_{\max}^c], \mbf x,\mbf y \in \ZHS\Bigr)
    \]
    measurable with respect to the weights $(\omega_{\mbf x})_{\mbf x \in \ZHS}$ that satisfies the following:
\begin{enumerate}  [label=(\roman*), font=\normalfont]
    \item \label{it:Buse_dist} For all $\xi \in (0,\xi_{\max}^c)$, and each $t \in \Z$, the process 
    $\bigl(W_\xi^c((t,t),(k + t,t))\bigr)_{k \in \N}$ has law $\mu_{c,s}$, where $s = (X_c \circ T_c)^{-1}(\xi)$. For $\xi = \xi_{\max}^c$, the process has law $\mu_{c,r_c}$.
    \item \label{it:Buse_shift} The law of the Busemann process is invariant under diagonal shifts. That is, for $m \in \Z$,
    \[
    \Bigl(W_\xi^c(\mbf x,\mbf y): \xi \in (0,\xi_{\max}^c], \mbf x,\mbf y \in \ZHS\Bigr) = \Bigl(W_\xi^c(\mbf x +(m,m),\mbf y + (m,m)): \xi \in (0,\xi_{\max}^c], \mbf x,\mbf y \in \ZHS\Bigr).
    \]
    \item{\rm{(Continuity at fixed $\xi$)}} \label{it:cts_fixed_xi} For each $\zeta \in (0,\xi_{\max}^c]$ and $\mbf x,\mbf y \in \ZHS$, the function $\xi \mapsto W_\xi^c(\mbf x,\mbf y)$ is continuous at $\xi = \zeta$ on a $\zeta$-dependent event of probability one.
\end{enumerate}
Furthermore, there is an event $\Omega_1$ of probability one on which the following hold. 
    \begin{enumerate} [resume, label=(\roman*), font=\normalfont]
    \item \label{it:left_continuity} {\rm{(Left continuity)}}
        For each $\mbf x,\mbf y \in \ZHS$, the function $\xi \mapsto W_\xi^c(\mbf x,\mbf y)$ is left-continuous with right limits defined to be $W_{\xi +}^c(\mbf x,\mbf y)$.
        \item{\rm{(Additivity)}} \label{it:Additivity} For $\mbf x,\mbf y,\mbf z \in \ZHS$ and $\xi \in (0,\xi_{\max}^c]$,
        \[
        W_\xi^c(\mbf x,\mbf y) + W_\xi^c(\mbf y,\mbf z) = W_\xi^c(\mbf x,\mbf z).
        \]
        \item \label{it:Buse_full_eternal} {\rm{(Eternal solution)}} For each $\xi \in (0,\xi_{\max}^c]$ and $\mbf x \in \ZHS$, the functions $(m,n) \mapsto W_\xi^c\bigl(\mbf x,(m,n)\bigr)$ and $(m,n) \mapsto W_{\xi+}^c\bigl(\mbf x,(m,n)\bigr)$ are eternal solutions. 
     \item \label{it:full_Buse_mont} {\rm{(Monotonicity)}} For $\zeta < \xi$ and integers $k \ge t$, 
    \begin{align}
    &W_\xi^c\bigl((k,t),(k+1,t)\bigr) \le  W_\zeta^c\bigl((k,t)(k+1,t)\bigr),\qquad\text{and}  \\
    &W_\xi^c\bigl((k,t),(k,t+1)\bigr) \ge W_\zeta^c\bigl((k,t),(k,t+1)\bigr).
    \end{align}
        \item \label{it:Buse_direction} {\rm{(Direction of semi-infinite geodesics)}}
        For $\mbf y \in \ZHS$ and $\xi \in (0,\xi_{\max}^c]$, let $\gamma_{\mbf y}^\xi$ be the semi-infinite geodesic associated to the eternal solution $(m,n) \mapsto W_\xi^c\bigl((0,0),(m,n)\bigr)$, rooted at the point $\mbf y$. Then, for each $\mbf y \in \ZHS$ and $\xi \in (0,\xi_{\max}^c)$, we have 
        \be \label{eq:lim_geod}
        \lim_{k \to \infty} \f{\gamma_{\mbf y}^\xi(-k)\cdot \mbf e_1}{\gamma_{\mbf y}^\xi(-k)\cdot \mbf e_2} = \xi,
        \ee
        where $\mbf e_1= (1,0)$ and $\mbf e_2 = (0,1)$ are the standard basis vectors. 
        For $\xi = \xi_{\max}^c$, we have
        \be \label{eq:lim_bigger_max}
        \lim_{k \to \infty}  \f{\gamma_{\mbf y}^{\xi_{\max}^c}(-k)\cdot \mbf e_1}{\gamma_{\mbf y}^{\xi_{\max}^c}(-k)\cdot \mbf e_2} = 1.
        \ee
        \item \label{it:Buse_limits} {\rm{(Busemann limits)}} If $\xi \in (0,\xi_{\max}^c)$ and if $\mbf v_n = (i_n,j_n)$ is a sequence satisfying $\lim_{n \to \infty} j_n = -\infty$ and $\lim_{n \to \infty} \f{i_n}{j_n} = \xi$, then for all integers $\ell \ge m \ge t$, we have 
        \begin{align*}
        W_{\xi +}^c\bigl((m,t),(\ell,t)\bigr) &\le  \liminf_{n \to \infty} \bigl[G\bigl(\mbf v_n,(\ell,t)\bigr) - G\bigl(\mbf v_n,(m,t)\bigr) \bigr]   \\
        &\le  \limsup_{n \to \infty} \bigl[G\bigl(\mbf v_n,(\ell,t)\bigr) - G\bigl(\mbf v_n,(m,t)\bigr) \bigr]  \le W_{\xi }^c\bigl((m,t),(\ell,t)\bigr).
        \end{align*}
        In particular, the limit in the middle exists whenever $W_{\xi+}^c\bigl((m,t),(\ell,t)\bigr) = W_{\xi}^c\bigl((m,t),(\ell,t)\bigr) $. 
        \item \label{it:Buse_limits_boundary} {\rm{(Busemann Limits at the boundary)}} For $\xi = \xi_{\max}^c$, if a sequence $\mbf v_n = (i_n,j_n)$ satisfies $\lim_{n \to \infty} j_n = -\infty$ and $\liminf_{n \to \infty} \f{i_n}{j_n} \ge \xi_{\max}^c$, then for all integers $\ell \ge m \ge t$,
        \[
        \limsup_{n \to \infty} \bigl[G\bigl(\mbf v_n,(\ell,t)\bigr) - G\bigl(\mbf v_n,(m,t)\bigr) \bigr]  \le W_{\xi_
        {\max}^c}^c\bigl((m,t),(\ell,t)\bigr).
        \]
        \item \label{it:Buse_attractive} \rm{(Attractivity)} If $\xi \in (0,\xi_{\max}^c)$ and $f_n:\Z_{\ge -n} \to \R$ is a sequence of functions such that 
        \[
        \lim_{n \to \infty} \f{1}{n} \max_{1 \le k \le n}[|f_n(k - n-1) - X_c^{-1}(\xi)k|] = 0,
        \]
        then, for all integers $\ell \ge m \ge t$,
         \begin{align*}
        W_{\xi +}^c\bigl((m,t),(\ell,t)\bigr) &\le  \liminf_{n \to \infty} \bigl[G_{f_n,-n}(\ell,t) - G_{f_n,-n}(m,t)) \bigr]   \\
        &\le  \limsup_{n \to \infty} \bigl[G_{f_n,-n}(\ell,t) - G_{f_n,-n}(m,t)) \bigr] \le W_{\xi }^c\bigl((m,t),(\ell,t)\bigr).
        \end{align*}
        \item \label{it:boundary_attract} \rm{(Attractivity at the boundary)} If $\xi = \xi_{\max}^c$ and a sequence of functions $f_n:\Z_{\ge -n} \to \R$ is such that $\liminf_{n \to \infty} \f{1}{n}f_n(-n) \ge 0$ and
        \[
        \limsup_{n \to \infty} \f{1}{n} \max_{1 \le k \le n}[f_n(k - n-1) - X_c^{-1}(\xi_{\max}^c)k] \le 0,
        \]
        then for all integers $\ell \ge m \ge t$,
        \[
        \limsup_{n \to \infty} \bigl[G_{f_n,-n}(\ell,t) - G_{f_n,-n}(m,t)) \bigr] \le W_{\xi_{\max}^c }\bigl((m,t),(\ell,t)\bigr).
        \]
    \end{enumerate}
\end{proposition}
\begin{remark}
    Items \ref{it:Buse_limits_boundary} and \ref{it:boundary_attract} both hold on a single event of probability one. For a fixed initial condition $f_n$, Proposition \ref{prop:to_Buse_in_prob} upgrades Item \ref{it:boundary_attract} to give convergence in probability to $W_{\xi_{\max}^c}^c\bigl((t,t),(t+k,t)\bigr)$ (see also Corollary \ref{cor:bd_to_Buse}). 
\end{remark}
\begin{proof}[Proof of Proposition \ref{prop:full_Busemann}]
    By Proposition \ref{prop:Buse_limits_1}, on a single event of probability one, for each $\xi \in (0,\xi_{\max}^c) \cap \Q$ and $\mbf x,\mbf y \in \ZHS$, we may define a function $\overline W_\xi^c(\mbf x,\mbf y)$ by the limit \eqref{eq:def_of_W}. We will call this full-probability event $\Omega_1$. Then, for each $\mbf x,\mbf y \in \ZHS$, by writing $\mbf y - \mbf x$ as a telescoping sum of terms of the form $\pm\bigl((k+1,t) - (k,t)\bigr)$ and of the form $\pm \bigl((k,t+1) - (k,t)\bigr)$, the monotonicity in Proposition \ref{prop:Buse_limits_1}\ref{it:prelim_Buse_monotonicity} implies that for each $\xi \in (0,\xi_{\max}^c]$, the limits
    \[
    W_\xi^c(\mbf x,\mbf y) := \lim_{\Q \ni \zeta \nearrow \xi} \overline W_{\zeta}^c(\mbf x,\mbf y),\quad\text{and}\quad  W_{\xi+}^c(\mbf x,\mbf y) := \lim_{\Q \ni \zeta \searrow \xi} \overline W_{\zeta}^c(\mbf x,\mbf y)\quad\text{exist}. 
    \]
    (where we only have limits from the left for $\xi = \xi_{\max}^c$). The law in Proposition \ref{prop:Buse_limits_1}\ref{it:Buse_h_law} along with the weak continuity of the measures $\mu_{c,s}$ in Lemma \ref{lem:mu_continuous} gives us \textbf{Item \ref{it:Buse_dist}}. \textbf{Item \ref{it:Buse_shift}} follows from the shift invariance in Lemma \ref{lem:shift_sym} and the limit definition of the Busemann functions in \eqref{eq:def_of_W}. \textbf{ Items \ref{it:left_continuity}-\ref{it:Additivity}} follow immediately from the definition. Items \ref{it:Buse_full_eternal} and \ref{it:full_Buse_mont} follow immediately from their prelimiting counterparts in Proposition \ref{prop:Buse_limits_1}. 

    To see the almost sure continuity in \textbf{Item \ref{it:cts_fixed_xi}}, since we already know the process is left-continuous by Item \ref{it:left_continuity}, we may assume $\zeta < \xi_{\max}^c$. Then, by the monotonicity in Item \ref{it:full_Buse_mont} for $t \in \Z$ and $k \in \N$,
    \be \label{zeta+le}
    W_{\zeta+}^c((t,t),(k + t,t)) \le W_{\zeta}^c((t,t),(k + t,t)).
    \ee
    By Item \ref{it:Buse_dist} and the continuity of the measures in Lemma \ref{lem:mu_continuous}, both sides of the inequality in \eqref{zeta+le} have the same law and are thus equal on a $\zeta$-dependent full-probability event for all integers $t \in \Z$ and $k \in \N$. Since $(m,n) \mapsto W_{\zeta}^c((t,t),(m,n))$ and $(m,n) \mapsto W_{\zeta+}^c((t,t),(m,n))$ are both eternal solutions by Item \ref{it:Buse_full_eternal}, Corollary \ref{cor:eternal_soln_recursion} implies that $W_{\zeta}^c((t,t),(m,n)) = W_{\zeta+}^c((t,t),(m,n))$ for all $m \ge n \ge t$. Then, the additivity in Item \ref{it:Additivity} implies that $W_{\zeta}^c(\mbf x,\mbf y) = W_{\zeta+}^c(\mbf x,\mbf y)$ for all $\mbf x,\mbf y \in \ZHS$. This proves Item \ref{it:cts_fixed_xi}.

    We turn to proving \textbf{Item \ref{it:Buse_direction}.} Fix $\mbf y = (m,n) \in \ZHS$, and for each $k \in \N$, define $(m_k^\xi,n_k^\xi) := \gamma_{\mbf y}^\xi(-k)$. Since the geodesic is an up-right path, we have that 
    \be \label{lxiLxi_bd}
    \begin{aligned}
    \ell_k^\xi &:= \max\{i \in \Z: \gamma_{\mbf y}^\xi(-k) = (i,n_k^\xi-1) \text{ for some }k \in \N\} \\
    &\le m_k^\xi \le \max\{i: \gamma_{\mbf y}^\xi(-k) = (i,n_k^\xi) \text{ for some }k \in \N\} =: L_k^\xi.
    \end{aligned}
    \ee
      Then, for each $k \in \Z_{\ge 0}$, Lemma \ref{lem:geodesics_are_maximizers} implies that 
    \begin{align*}
    &\ell_k^\xi \in \argmax_{i \in \llbracket n_k^\xi, m\rrbracket}\bigl[W_{\xi}^c\bigl((n_k^\xi-1,n_k^\xi-1),(i,n_k^\xi - 1)\bigr) +  G\bigl((i,n_k^\xi),(m,n)\bigr)\bigr],\quad\text{and} \\
    &L_k^\xi  \in \argmax_{i \in \llbracket n_k^\xi + 1, m\rrbracket}\bigl[W_{\xi}^c\bigl((n_k^\xi,n_k^\xi),(i,n_k^\xi)\bigr) +  G\bigl((i,n_k^\xi + 1),(m,n)\bigr)\bigr].
    \end{align*}
    Then, by the distribution in Item \ref{it:Buse_dist} and Lemma \ref{lem:stat_obeys_slope}, the conditions of Proposition \ref{prop:converge_to_max}\ref{it:mx_location} are satisfied so that we can use a squeeze argument via \eqref{lxiLxi_bd} to get that, on a single event of probability one, for all $\xi \in (0,\xi_{\max}^c)\cap \Q$,
    \[
    \lim_{k \to \infty} \f{m_k^\xi}{n_k^\xi} = \xi.
    \]
    However, by the monotonicity in Item \ref{it:full_Buse_mont} and definition of the paths $\gamma_{\mbf y}^{\xi}$ from \eqref{eq:gamma_path}, we see that for $\xi < \zeta$, the path $\gamma_{\mbf y}^\xi$ always lies weakly to the right of $\gamma_{\mbf y}^\zeta$. Hence, we can extend the limits \eqref{eq:lim_geod} from $\xi \in (0,\xi_{\max}^c) \cap \Q$ to all $\xi \in (0,\xi_{\max}^c)$. 
    
    To get the limits \eqref{eq:lim_bigger_max} for $\xi = \xi_{\max}^c$, we consider two cases. In the case $c \le 1$, we have $\xi_{\max}^c = 1$ (see \eqref{Wxi_and_xi_max}), so comparison with the paths $\gamma_{\mbf y}^\xi$ for rational $\xi < \xi_{\max}^c$ as above gives us 
    \[
    \liminf_{k \to \infty} \f{m_k^{\xi_{\max}^c}}{n_k^{\xi_{\max}^c}} \ge \xi_{\max}^c = 1.
    \]
    However, we always have $m_k^{\xi_{\max}^c} \ge n_k^{\xi_{\max}^c}$ since $(m_k^{\xi_{\max}^c},n_k^{\xi_{\max}^c}) \in \ZHS$, and since $n_k^{\xi_{\max}^c} \to -\infty$ as $k \to \infty$, we have 
    \[
    \limsup_{k \to \infty} \f{m_k^{\xi_{\max}^c}}{n_k^{\xi_{\max}^c}} \le 1.
    \]
    In the case $c > 1$, we have $r_c = c$, and the definition \eqref{Wxi_and_xi_max} gives us that $X_c^{-1}(\xi_{\max}^c) = \f{qc}{1-qc} > \f{q}{c-q} = T_c(c)$. Hence, since the map $X_c$ is increasing (Lemma \ref{lem:TX_maps}\ref{it:X_fact}), $(X_c \circ T_c)(c) = 1$ by definition of $X_c$ \eqref{eq:K_func}. Then, by a similar squeeze argument as above, the distribution in Item \ref{it:Buse_dist} and Lemma \ref{lem:stat_obeys_slope} implies that the conditions of Proposition \ref{prop:converge_to_max}\ref{it:mx_location} are satisfied for $\theta = \f{q}{c-q} = T_c(c)$ so that
    \[
    \lim_{k \to \infty} \f{m_k^{\xi_{\max}^c}}{n_k^{\xi_{\max}^c}} = 1.
    \]

   \textbf{ Items \ref{it:Buse_limits}-\ref{it:boundary_attract}} all follow from Items \ref{it:Buse_dist}, \ref{it:left_continuity}, \ref{it:Buse_full_eternal}, and Lemma \ref{lem:comp_to_eternal}. For example, to see Item \ref{it:Buse_limits}, Lemma \ref{lem:comp_to_eternal}\ref{it:upbd1}-\ref{it:lbd1} gives us that for $\zeta < \xi < \eta$,
    \begin{align*}
        W_{\eta}^c\bigl((m,t),(\ell,t)\bigr) &\le  \liminf_{n \to \infty} \bigl[G\bigl(\mbf v_n,(\ell,t)\bigr) - G\bigl(\mbf v_n,(m,t)\bigr) \bigr]   \\
        &\le  \limsup_{n \to \infty} \bigl[G\bigl(\mbf v_n,(\ell,t)\bigr) - G\bigl(\mbf v_n,(m,t)\bigr) \bigr]  \le W_{\zeta }^c\bigl((m,t),(\ell,t)\bigr),
        \end{align*}
        then the result follows from Item \ref{it:left_continuity} by taking $\zeta\nearrow \xi$ and $\eta \searrow \xi$. Similarly, we get Item \ref{it:Buse_limits_boundary} by using Lemma \ref{lem:comp_to_eternal}\ref{it:upbd1}, we get Item \ref{it:Buse_attractive} by using Lemma \ref{lem:comp_to_eternal}\ref{it:both_bd1}, and we get Item \ref{it:boundary_attract} by using Lemma \ref{lem:comp_to_eternal}\ref{it:lbd2}.
\end{proof}

\section{Local convergence of the increments process from the boundary} \label{sec:line_ensemble}
In this section, we prove the following result:
\begin{proposition}\label{prop:recentered_conv}
    As $n \to \infty$, the law of the process
    \[
    \Bigl(G((1,1),(n+k,n)) - G((1,1),(n,n))\Bigr)_{k \in \N}
    \]
    converges weakly to $\mu_{c,r_c}$.
\end{proposition}
Given Proposition \ref{prop:recentered_conv}, the proofs that follow in Section \ref{sec:proofs} can be read independently of this section. The key tools in this section are the Pfaffian Schur process representation of the process and the theory of line ensembles. As mentioned in the introduction, the proof relies on the half-space geometric line ensemble and its Gibbs property, which we introduce in Section \ref{sec:scaling_limits}. The top line of the line of the line ensemble is equal in distribution to passage times from $(1,1)$ to $(n,n+k)$ as $k$ varies over $\Z_{\ge 0}$.  In Section \ref{sec:local convergence}, we establish the recentered convergence for the top layer of this Gibbs law in the case $c \in [0,1)$ (the case $c \in [1,q^{-1})$ is straightforward), assuming that the third and lower curves do not factor into the limit. This gives us the invariant measure as a functional of two independent geometric random walks, as in Definition \ref{def:mucs}. In Section \ref{sec:techlemmas}, we collect separation estimates for the line ensemble, which allow us to show that these lower curves can in fact be ignored in the limit. The proof of Proposition \ref{prop:recentered_conv} is completed in Section \ref{subsec:recentered_conv}.

\subsection{Half-space geometric line ensemble and its Gibbs property} \label{sec:scaling_limits}

In this section, we explain the connection between half-space geometric LPP model, Pfaffian Schur processes, and the half-space geometric line ensemble that were observed in \cite{dy25,ddy26}.

A {\em partition} is a non-increasing sequence of non-negative integers $\lambda=(\lambda_1\geq\lambda_2\geq \dots)$ with finitely many non-zero elements. For any partition $\lambda$, we denote its {\em weight} by $|\lambda|=\sum_{i=1}^{\infty}\lambda_i$. There is a single partition of weight $0$, which we denote by $\emptyset$.  Given two partitions $\lambda$ and $\mu$, we say that $\mu$ and $\lambda$ {\em interlace} if $\lambda_1\geq\mu_1\geq\lambda_2\geq\mu_2\geq\dots$. In this case, we write $\mu\preceq\lambda$ or $\lambda \succeq \mu$.

Given finitely many variables $x_1, \dots, x_n$, we define the {\em skew Schur polynomials} via
\begin{equation}\label{Eq.SkewSchur}
s_{\lambda/ \mu}(x_1, \dots, x_n) = \sum_{\mu = \lambda^{0} \preceq  \lambda^{1} \preceq \cdots \preceq \lambda^{n} = \lambda} \prod_{i = 1}^n x_i^{|\lambda^{i}| - |\lambda^{i-1}|}.
\end{equation}
When $\mu = \emptyset$ we drop it from the notation and write $s_{\lambda}$, which is then the Schur polynomial indexed by $\lambda$. We refer the interested reader to \cite{Mac,BG16} for an introduction to Schur symmetric polynomials.
\begin{definition}\label{Def.SchurProcess} Fix $m, \in \mathbb{N}$, $q \in (0, 1)$, and $c \in [0, q^{-1})$. We define the {\em Pfaffian Schur process} to be the probability distribution on sequences of partitions $\lambda^0,\dots,\lambda^m$, given by
\begin{equation}\label{Eq.SchurProcess}
\begin{split}
&\mathbb{P}\left(\lambda^{0},\dots,\lambda^{m}\right)\propto  c^{\sum_{j=1}^{\infty}(-1)^{j-1}\lambda_j^0} \cdot s_{\lambda^{1}/\lambda^{0}}(q)\cdots s_{\lambda^m/\lambda^{m-1}}(q) \cdot s_{\lambda^{m}}(\underbrace{q,q, \dots, q}_{\text{$m$ times}}).
\end{split}
\end{equation}
\end{definition}
The fact that the above definition gives rise to a honest probability measure (i.e., the normalizing constant is finite) follows from \cite{psp}. The Pfaffian Schur process gives rise to a line ensemble structure defined below.

\begin{definition}[Half-space geometric line ensemble] \label{def:le}
    For each $n \in \mathbb{N}$, consider $(\lambda^0(n),\ldots,\lambda^n(n))$ distributed according to \eqref{Eq.SchurProcess} with $m=n$. For $i\in \mathbb{N}$ and $j\in \llbracket 0,n \rrbracket$ set  $L_i^n(j)=\lambda_i^j(n)$.   We extend these functions to $[0,\infty)$ by linear interpolation on $[0,n]$ and by setting them equal to $L_i^n(n)$ on $[n,\infty)$. The random functions $(L_i^n)_{i\in \N}$ are called the \textit{half-space geometric line ensemble} \cite{dy25,ddy26}, where $L_i^n$ denotes the $i$-th line.
\end{definition}

The half-space geometric line ensemble enjoys certain Gibbs resampling properties
on increasing paths. We next define those probability measures.

\begin{definition}\label{def:gibbslaw} Let $\Omega(m,y)$ be the set of all functions $B : \llbracket 0,m\rrbracket \to \mathbb{Z}$ satisfying  $B(m) =y$  and
    $B(j+1) \ge B(j)$ for $j\in \llbracket 0,m-1\rrbracket$.  For $p\in (0,1)$ we define the \textit{reversed geometric random walk law} $\mathbb{P}_{\mathrm{Geom};p}^{m,y}$ to be the measure on $\Omega(m,y)$ such that  
\begin{align*}
    \mathbb{P}_{\mathrm{Geom};p}^{m,y}(Q = B) \propto p^{-B(0)},
\end{align*}
where $Q$ denotes a $\Omega(m,y)$-valued random variable with this law. 

    Let $\Omega_{\mathrm{Inter}}(m,(y_1,y_2))$ be the set of all functions $(B_1,B_2) \in \Omega(m,y_1)\times \Omega(m,y_2)$ satisfying 
    $B_1(j) \ge B_2(j+1)$ for  $j\in \llbracket 0,m-1\rrbracket$. For $q\in (0,1)$ and $c\in [0,q^{-1})$ we define the \textit{interlacing Gibbs law} $\Pp_{\mathrm{Inter};q,c}^{m,(y_1,y_2)}$ to be the measure on $\Omega_{\mathrm{Inter}}(m,(y_1,y_2))$ such that
\begin{align}\label{def:qc}
  \Pp_{\mathrm{Inter};q,c}^{m,(y_1,y_2)}(Q_1 = B_1,Q_2 = B_2) \propto   c^{B_1(0)-B_2(0)}q^{-B_1(0)-B_2(0)},
\end{align}
where $(Q_1,Q_2)$ denotes a  $\Omega_{\mathrm{Inter}}(m,(y_1,y_2))$-valued random variable with this law. 
\end{definition}

\begin{remark}\label{rem:gibbslaw}
    It is straightforward to check that $\mathbb{P}_{\mathrm{Geom};p}^{m,y}$ is the law of a random walk running backwards, starting from $y$ at time $m$, and at each step makes a negative $\Geo(p)$ jump. However, the fact that the interlacing Gibbs law is well defined (i.e., \eqref{def:qc} gives rise to a honest probability measure) is nontrivial and follows from \cite[Lemma 2.22]{ddy26}.
\end{remark}

The following result from \cite{dy25b,ddy26} explains the connection between the Pfaffian Schur processes and the half-space geometric LPP model and discusses the Gibbs resampling properties of the half-space geometric line ensemble. 

\begin{proposition}\label{prop:laweq} For each $n \in \mathbb{N}$, consider the half-space geometric line ensemble $(L_i^n)_{i\ge 1}$ defined in Definition \ref{def:le}.  We have the following.
    \begin{enumerate} [label=(\roman*), font=\normalfont]
         \item \label{it:G_L_law}  $ (L_1^n(j))_{j\in \llbracket 0,n\rrbracket}$ has the same distribution  as $G\bigl((1,1),(n+j,n)\bigr)_{j\in \llbracket 0,n\rrbracket}$, where the vertex weights are distributed as in Definition \ref{def:GLPPc}.

        \item \label{it:Gibbs1} Fix any $m\in \llbracket 1,n\rrbracket$. When $c\in (q,q^{-1})$, we have
\begin{align} \label{eq:laweq1}
         \Pp\left(L_1^n\llbracket 0,m\rrbracket=B | \mathcal{F}_{\mathrm{ext}}^1\right) \stackrel{a.s.}{=} \frac{\Pp_{\mathrm{Geom};c^{-1}q}^{m,y}\left(\{Q=B\}\cap \{Q(j) \ge g(j+1) \mbox{ for }j\in \llbracket 0,m-1\rrbracket\}\right)}{\Pp_{\mathrm{Geom};c^{-1}q}^{m,y}\left(Q(j) \ge g(j+1) \mbox{ for }j\in \llbracket 0,m-1\rrbracket\right)}
    \end{align}
    where $\mathcal{F}_{\mathrm{ext}}^1=\sigma(L_i^n(j) : (i,j)\not\in \{1\}\times \llbracket0,m-1\rrbracket)$, $y=L_1^n(m)$, and $g=L_2^n\llbracket0,m\rrbracket$.
    \item \label{it:Gibbs2} Fix any $m\in \llbracket 1,n\rrbracket$. When $c\in [0,q^{-1})$ we have
    \begin{align}\label{eq:laweq2}
        & \Pp\left(L_i^n\llbracket 0,m\rrbracket=B_i \mbox{ for }i\in \{1,2\}| \mathcal{F}_{\mathrm{ext}}^{\llbracket1,2\rrbracket}\right) \\ & \hspace{2cm} \stackrel{a.s.}{=} \frac{\Pp_{\mathrm{Inter};q,c}^{m,(y_1,y_2)}\left(\{Q_1=B_1\}\cap\{Q_2=B_2\}\cap \{Q_2(j) \ge g(j+1) \mbox{ for }j\in \llbracket 0,m-1\rrbracket\}\right)}{\Pp_{\mathrm{Inter};q,c}^{m,(y_1,y_2)}\left(Q_2(j) \ge g(j+1) \mbox{ for }j\in \llbracket 0,m-1\rrbracket\right)}
    \end{align}
    where $\mathcal{F}_{\mathrm{ext}}^{\llbracket1,2\rrbracket}=\sigma(L_i^n(j) : (i,j)\not\in \{1,2\}\times \llbracket0,m-1\rrbracket)$, $y_1=L_1^n(m)$, $y_2=L_2^n(m)$, and $g=L_3^n\llbracket0,m\rrbracket$.
    \end{enumerate}
\end{proposition}

\begin{proof}Item \ref{it:G_L_law} appears in \cite[Proposition 4.3]{ddy26}, which is a special case of \cite[Theorem 2.7]{dy25b}. Items \ref{it:Gibbs1} and \ref{it:Gibbs2} follow from  \cite[Lemma 4.6]{ddy26}. 
\end{proof}

\medskip

The interlacing Gibbs law (defined in Definition \ref{def:gibbslaw}) and its scaling limit and properties of that scaling limit were studied extensively in \cite{dy25,ds25b,ddy26}. We conclude this section by stating a Kolmogorov-type maximal inequality for the Gibbs law, whose proof relies on several results from these works and will be useful for our later analysis.

\begin{lemma}\label{intkolm} Fix $\varepsilon\in (0,1)$ and $M_0>0$. For each $m\in \mathbb{N}$ and $M \in (0,\infty)$, let $V_{m,M}$ be the set of all $(y_1,y_2)\in \ZHS$ satisfying
\begin{align}\label{cond}
    y_1-y_2 \le M \sqrt{m}.
\end{align}
\begin{enumerate}[label=(\roman*), font=\normalfont]
    \item \label{it:max1} Suppose $c\in [0,1)$. Then, there exists $\lambda,M_1>0$ depending on $\varepsilon,q,c,M_0$ such that for all $m\ge M_1$ and $(y_1,y_2)\in V_{m,M_0\lambda}$ we have
     \begin{align}\label{bigpr}
        \Pp_{\mathrm{Inter};q,c}^{m,(y_1,y_2)}\left(\inf_{j\in \llbracket 0,m\rrbracket} \left(Q_2(j)-y_2+\frac{(m -j)q}{1-q}\right) \ge -\lambda \sqrt{m}\right) \ge 1-\varepsilon.
    \end{align}
    \item \label{it:max2} Suppose $c\in [1,q^{-1})$. Then, there exists $\lambda>0$ depending on $\varepsilon,q,c$ such that for all $m\in \mathbb{N}$ and $y\in \mathbb{R}$, we have
 \begin{align}\label{bigpr01}
\Pp_{\mathrm{Geom};c^{-1}q}^{m,y}\left(\inf_{j\in \llbracket 0,m\rrbracket} \left(Q(j)-y+\frac{(m -j)q}{c-q}\right) \ge -\lambda \sqrt{m}\right) \ge 1-\varepsilon.
    \end{align}
\end{enumerate}
\end{lemma}

\begin{proof} Recall from Remark \ref{rem:gibbslaw}, that $\mathbb{P}_{\mathrm{Geom};c^{-1}q}^{m,y}$ is the law of a random walk that
starts from $y$ at time $m$, is being run backwards, and at each step makes a negative geometric jump
with parameter $c^{-1}q$. From this interpretation, Item \ref{it:max2} follows from Kolmogorov's maximal inequality. We therefore focus on proving Item \ref{it:max1}. Note that, by translation,
       \begin{align*}
        \mbox{l.h.s.~of~\eqref{bigpr}} & = R_m(y_1-y_2) :=\Pp_{\mathrm{Inter};q,c}^{m,(y_1-y_2,0)}\left(\inf_{j\in \llbracket 0,m\rrbracket} \left(Q_2(j)+\frac{(m -j)q}{1-q}\right) \ge -\lambda \sqrt{m}\right).
    \end{align*}
For each $a\ge 0$, let $U^a$ be a Brownian motion on $[0,1]$ started from $a$ and $V^a$ be an independent 3D Bessel bridge on $[0,1]$ starting at $0$ and ending at $a$. Let $\Pp$ be the probability measure on a probability space containing both of these processes. Let us choose $\lambda > 0$ so that 
\begin{align} \label{eq:bigeps}
    \Pp\left(\inf_{t\in [0,1]} 2^{-1/2}\sigma(U_{1-t}^0-V_t^0) \ge -\lambda\right) \ge 1-\varepsilon/2.
\end{align}
where $\sigma^2=q/(1-q)^2$. We fix this choice of $\lambda$. We next claim that
\begin{align}\label{limit}
    u:=\liminf_{m\to \infty} \inf_{y\in \llbracket0,M_0\lambda\sqrt{m}\rrbracket}  R_m(y)  \ge  \Pp\left(\sup_{t\in [0,1]} 2^{-1/2}\sigma(U_{1-t}^0-V_t^0) \ge -\lambda\right).
\end{align}
Combining \eqref{eq:bigeps} and \eqref{limit} will prove the lemma. 
Since for each finite $m$, the infimum is over a finite collection, we can get $y_m \in \llbracket 0, M_0\lambda \sqrt{m}\rrbracket$ for which the infimum is attained for each $m$. Since $R_m$ is a probability and by \eqref{cond}, we may pass to a subsequence $y_{k_m} \in \llbracket 0, M_0\lambda \sqrt{k_m}\rrbracket$ such that $R_{k_m}(y_{k_m}) \to u$ and
$k_m^{-1/2}y_{k_m} \to v^\ast$. 

Suppose $(Q_1^m,Q_2^m) \sim \Pp_{\mathrm{Inter};q,c}^{k_m,(y_{k_m},0)}$. Since $k_m^{-1/2}y_{k_m} \to v^\ast$, \cite[Lemma 3.5(b)]{ddy26} implies that  $$\Biggl(k_m^{-1/2}\left(Q_1^m(\lfloor tk_m\rfloor)-\frac{q}{1-q}(k_m-\lfloor tk_m\rfloor)\right),k_m^{-1/2}\left(Q_2^m(\lfloor tk_m\rfloor)-\frac{q}{1-q}(k_m-\lfloor tk_m\rfloor)\right)\Biggr)$$ converges weakly to $(\sigma\mathcal{Q}_1^{v^\ast}(t),\sigma\mathcal{Q}_2^{v^\ast}(t))$ as processes in $\mathcal{C}([0,1]^2)$, where $(\mathcal{Q}_1^{v^\ast},\mathcal{Q}_2^{v^\ast})$ are reverse Brownian motions on $[0,1]$ started from $(v^{\ast},0)$ conditioned to non-intersect in $(0,1)$ and $\mathcal{Q}_1^{v^\ast}(0)=\mathcal{Q}_2^{v^\ast}(0)$ (pinning at the origin). We mention that  \cite[Lemma 3.5(b)]{ddy26} is a general result for interlacing geometric random walks law with random endpoints. We have specialized their result to our interlacing Gibbs law setting. In view of our choice of subsequence and the above convergence result, we have
$$u=\lim_{m\to\infty} R_{k_m}(y_{k_m}) =\Pp\left(\sup_{t\in [0,1]}\sigma \mathcal{Q}_2^{v^\ast}(t) \ge -\lambda\right).$$
By stochastic monotonicity of the above law w.r.t.~$v^\ast$ (\cite[Lemma 4.13]{dsy26}) we have
$$u=\Pp\left(\sup_{t\in [0,1]}\sigma \mathcal{Q}_2^{v^\ast}(t) \ge -\lambda\right)\ge \Pp\left(\sup_{t\in [0,1]}\sigma \mathcal{Q}_2^{0}(t) \ge -\lambda\right).$$
But Lemma 2.10 in \cite{ddy26} implies, as a process in $t$, $\mathcal{Q}_2^0(t)$ is equal in distribution to $2^{-1/2}(U_{1-t}^0-V_t^0)$.
This verifies \eqref{limit} and hence completes the proof.
\end{proof}

\subsection{Recentered convergence for the interlacing Gibbs law} \label{sec:local convergence} The goal of this section is to study the recentered convergence for the interlacing Gibbs law defined in \eqref{def:qc}. 
To describe the limiting law, we need a few more pieces of notation. We first extend the definition of skew Schur functions from the space of integer partitions to the space of signatures. Let $\mathsf{Sign}_2$ denote the set of $(\lambda_1,\lambda_2)\in\mathbb{Z}^2$ satisfying $\lambda_1\geq\lambda_2$. Compared to the set of partitions, we allow all integers (not just nonnegative). We call this the set of signatures of length $2$. One can also construct the set of signatures of longer lengths, but we only consider signatures of length $2$ in the present paper. 
For $\lambda,\mu\in\mathsf{Sign}_2$ and a single variable $x$, define the skew Schur function
\begin{equation}\label{eq:def of skew Schur}
s_{\lambda/\mu}(x)={\bf 1}_{\mu\preceq\lambda}\cdot x^{|\lambda|-|\mu|},
\end{equation}
where we write $\mu \preceq\lambda$ or $\lambda\succeq\mu$ to mean $\lambda_1\geq\mu_1\geq\lambda_2\geq\mu_2$, $|\lambda|=\lambda_1+\lambda_2$, and $|\mu|=\mu_1+\mu_2$.
For any finite collection of variables $x=(x_1,\dots,x_n)$, we extend the definition by the branching rule 
 \begin{equation}\label{eq:branching}
 s_{\lambda/\mu}(x_1,\dots,x_n)=\sum_{\nu\in\mathsf{Sign}_2}s_{\lambda/\nu}(x_1,\dots,x_{n-1})s_{\nu/\mu}(x_n).
 \end{equation}
\begin{lemma}\label{lem:translation invariance}
For each $\lambda=(\lambda_1,\lambda_2)\in\mathsf{Sign}_2$ and $a\in\mathbb{Z}$, we write $\lambda+a(1,1)=(\lambda_1+a,\lambda_2+a)$.
Then for each $\lambda,\mu\in\mathsf{Sign}_2$, $a\in\mathbb{Z}$, and $x=(x_1,\dots,x_n)$, we have the translation invariance of skew Schur functions:
    \[
    s_{\lambda+a(1,1)/\mu+a(1,1)}(x)=s_{\lambda/\mu}(x).
    \]
\end{lemma}
\begin{proof}
    The proof follows directly from the definitions \eqref{eq:def of skew Schur} and \eqref{eq:branching} of the skew Schur functions.
\end{proof}
\begin{definition}\label{def:u and h}
    Fix $q\in(0,1)$ and $c\in[0,1)$.  For each $n,k,\ell\in\mathbb{Z}_{\geq0}$, we define 
\begin{equation}\label{eq:definition of u}
    u_n^k(\ell)=(1-q)^{2n}\sum_{\lambda\in\mathsf{Sign}_2}s_{\lambda/(\ell,0)}((q)^n){\bf 1}_{\{\lambda_1-\lambda_2=k\}}.
\end{equation}
Here and below, we write $(x)^n$ to mean the $n$-tuple $(x,\dots,x)$, and use the convention $s_{\lambda/\mu}((x)^0)={\bf 1}_{\{\lambda=\mu\}}$. 
For each $k,\ell\in\mathbb{Z}_{\geq0}$, $n\in\mathbb{N}$, and $x\in\llbracket0,n\rrbracket$, we also define
\begin{equation}\label{eq:definition of h}
   h_{x,n}^k(\ell)= \frac{u_{n-x}^{k}(\ell)}{\sum_{j\geq0}c^j u_n^{k}(j)}.
\end{equation}
The finiteness of \eqref{eq:definition of u} and of the denominator on the right side of \eqref{eq:definition of h} are implied by Lemma \ref{lem:formula u and h} below. Note that \(u_n^k(0)\geq (1-q)^{2n} s_{(k,0)/(0,0)}((q)^n)  >0\), hence the denominator on the right side of \eqref{eq:definition of h} is positive.
\end{definition}

\begin{remark}
The functions $u_n^k(\ell)$ and $h_{x,n}^k(\ell)$ from Definition \ref{def:u and h} are closely related to the quantities introduced in \cite[Section 2.6]{barraquand2024integral} in a different notation. The indicator function ${\bf 1}_{\{\lambda_1-\lambda_2=k\}}$ appearing on the right side of \eqref{eq:definition of u} is replaced in \cite{barraquand2024integral} by the weight $c_2^{\lambda_1-\lambda_2}$ for another fixed parameter $c_2\in[0,1)$.

The parameters $c_1=c$ and $c_2$ can be interpreted as the left and right boundary parameters of the two-layer Gibbs measures for Schur functions introduced in \cite[Section 2]{Barraquand-Corwin-Yang-2023}, which describe the stationary measures for geometric last-passage percolation on a strip. It is natural to expect that, as the width of the strip tends to infinity, the stationary measure for geometric LPP on a strip converges to the stationary measure for half-space geometric LPP with the same left boundary parameter $c_1=c$ and bulk parameter $q$, while the parameter $c_2$ now determines the limiting slope of the stationary measure at spatial infinity.

A sketch of the proof of this convergence was given in the discussion at the end of Section 2.6 of \cite{barraquand2024integral}, based on the Markovian description of the two-layer Gibbs measures developed in \cite[Section 2.6]{barraquand2024integral} together with the explicit integral formula of \cite[Proposition 2.12]{barraquand2024integral}. Our next two results develop the corresponding ingredients in the present setting of the Gibbs measure $\Pp_{\mathrm{Inter};q,c}^{m,(y_1,y_2)}$: Lemma \ref{lem:formula u and h} establishes the integral formula for $u_n^k(\ell)$, while Lemma \ref{lem:MC} gives the corresponding Markovian description involving $h_{x,n}^k(\ell)$. These results are then combined in Proposition \ref{thm:local convergence} to prove the convergence of $\Pp_{\mathrm{Inter};q,c}^{m,(y_{1,m},y_{2,m})}$, under suitable assumptions on the sequence $(y_{1,m},y_{2,m})$ as $m\to\infty$, to the stationary measures of half-space geometric LPP.
\end{remark}

\begin{lemma}\label{lem:formula u and h}
    Fix $q\in(0,1)$ and $c\in[0,1)$.  For each $n,k,\ell\in\mathbb{Z}_{\geq0}$, we have
\begin{align}
 \label{eq:formula u}   u_n^k(\ell)
&=
\frac1\pi\int_0^{2\pi}
\sin((\ell+1)\theta)\sin((k+1)\theta)
\left(1+\frac{4q}{(1-q)^2}\sin^2\frac{\theta}{2}\right)^{-n}
\mathsf{d}\theta ,\\
\label{eq:formula denominator of h}\sum_{j\geq0}c^j u_n^{k}(j)
&=\frac1\pi\int_0^{2\pi}
\frac{\sin\theta}{1-2c\cos\theta+c^2}
\sin((k+1)\theta)
\left(1+\frac{4q}{(1-q)^2}\sin^2\frac{\theta}{2}\right)^{-n}
\mathsf{d}\theta.
\end{align}
\end{lemma}
\begin{proof}
When \(n=0\), we have $u_0^k(\ell)={\bf 1}_{\{k=\ell\}}$,
and \eqref{eq:formula u} follows from the fact that
\[
\frac1\pi\int_0^{2\pi}
\sin((\ell+1)\theta)\sin((k+1)\theta)\,\mathsf d\theta
={\bf 1}_{\{k=\ell\}}.
\]
We next prove \eqref{eq:formula u} for \(n\in\mathbb{N}\).
    For each $b\in(0,1)$,  we have
\begin{equation}\label{eq:use Barraquand}
    \begin{split}
        \sum_{k\geq0}b^ku_n^k(\ell)&=(1-q)^{2n}\sum_{\lambda\in\mathsf{Sign}_2}s_{\lambda/(\ell,0)}((q)^n)b^{\lambda_1-\lambda_2}  \\
        &=\frac{1}{2}\oint_{|z|=1}\frac{\mathsf{d}z}{2\mathtt{i}\pi z}\cdot\left(z^{\ell+1}-\frac{1}{z^{\ell+1}}\right)\cdot\left(\frac{1}{z}-z\right)\cdot\frac{1}{1-bz}\cdot\frac{1}{1-b/z}\cdot\left(\frac{(1-q)^2}{(1-qz)(1-q/z)}\right)^n.
    \end{split}
\end{equation}
The first step uses \eqref{eq:definition of u}; and the second step uses \cite[Proposition 2.12 and (2.36)]{barraquand2024integral} with $M = N =n$, $c_2 = b$, and $a_{N-M+1}=\dots=a_N=q$. 
Observe that 
\[
\frac{1}{1-bz}\cdot\frac{1}{1-b/z}=\sum_{k\geq0}b^k\frac{\frac{1}{z^{k+1}}-z^{k+1}}{\frac{1}{z}-z}.
\]
Since \eqref{eq:use Barraquand} holds for all $b \in (0,1)$, we have
\begin{equation*}
\begin{split}
    u_n^k(\ell)&=\frac{1}{2}\oint_{|z|=1}\frac{\mathsf{d}z}{2\mathtt{i}\pi z}\cdot\left(z^{\ell+1}-\frac{1}{z^{\ell+1}}\right)\cdot\left(\frac{1}{z^{k+1}}-z^{k+1}\right) \cdot\left(\frac{(1-q)^2}{(1-qz)(1-q/z)}\right)^n\\
    &=
\frac1\pi\int_0^{2\pi}
\sin((\ell+1)\theta)\sin((k+1)\theta)
\left(1+\frac{4q}{(1-q)^2}\sin^2\frac{\theta}{2}\right)^{-n}
\mathsf{d}\theta.
\end{split}
\end{equation*}
The first step compares the coefficients of $b^k$ in \eqref{eq:use Barraquand}, and the second step uses the change of variables $z=e^{\mathtt{i}\theta}$, $\theta\in[0,2\pi]$. 
This concludes the proof of \eqref{eq:formula u}. Next, note that, since $c\in[0,1)$,
\[
\sum_{j\ge0}c^j\sin((j+1)\theta)
=
\frac{\sin\theta}{1-2c\cos\theta+c^2},
\]
which completes the proof of \eqref{eq:formula denominator of h}.
\end{proof}

\begin{lemma}\label{lem:MC}
     Fix $q\in(0,1)$ and $c\in[0,1)$.  Then for each $k\in\mathbb{Z}_{\geq0}$ and $n\in\mathbb{N}$,
\begin{equation}\label{eq:Markov chain n}
        \begin{split}
            \mathsf{p}_0(\lambda)& =  c^{\lambda_1-\lambda_2} h_{0,n}^k (\lambda_1-\lambda_2) {\bf 1}_{\{\lambda_2=0\}},\\
            \mathsf{p}_{x,x+1}(\mu,\lambda)&=(1-q)^{2 }s_{\lambda/\mu}(q )\frac{h_{x+1,n}^k (\lambda_1-\lambda_2)}{h_{x,n}^k (\mu_1-\mu_2)} \quad\mbox{ for }x \in \llbracket 0,n-1 \rrbracket,
        \end{split}
    \end{equation} 
are the initial and transitional laws of a Markov chain $\{\lambda^j\}_{j=0}^n$ on state space $\mathsf{Sign}_2$, which has joint law
\[  \mathbb{P}  (\lambda^{0},\dots,\lambda^{n}) \propto c^{\lambda^{0}_1-\lambda^{0}_2}s_{\lambda^{1}/\lambda^{0}}(q)\dots s_{\lambda^{n}/\lambda^{n-1}}(q) {\bf 1}_{\{\lambda_2^0=0\}} {\bf 1}_{\{\lambda^{n}_1-\lambda^n_2=k\}}.\]
\end{lemma}
\begin{proof}
We need to verify that
\begin{align}
    &\label{eq:initial sum to one}\sum_{\lambda\in\mathsf{Sign}_2}\mathsf{p}_0(\lambda)=1, \\
    &\label{eq:transition sum to one}\sum_{\lambda\in\mathsf{Sign}_2}\mathsf{p}_{x,x+1}(\mu,\lambda)=1
\quad \mbox{ for each }\mu\in\mathsf{Sign}_2 \mbox{ and } x \in \llbracket 0,n-1 \rrbracket,
\end{align}
and that for each $\lambda^{0},\dots,\lambda^{n}\in\mathsf{Sign}_2$,
\begin{equation}\label{eq:distribution equals to Markov chain}
    \mathsf{p}_0\left(\lambda^0\right)\mathsf{p}_{0,1}\left(\lambda^0,\lambda^1\right)\dots \mathsf{p}_{n-1,n}\left(\lambda^{n-1},\lambda^n\right)\propto c^{\lambda^{0}_1-\lambda^{0}_2}s_{\lambda^{1}/\lambda^{0}}(q)\dots s_{\lambda^{n}/\lambda^{n-1}}(q) {\bf 1}_{\{\lambda_2^0=0\}} {\bf 1}_{\{\lambda^{n}_1-\lambda^n_2=k\}}.
\end{equation}
Indeed, by the definitions \eqref{eq:Markov chain n} and \eqref{eq:definition of h}, we have
\begin{equation*}
    \sum_{\lambda\in\mathsf{Sign}_2}\mathsf{p}_0(\lambda)= \sum_{\ell\geq0}c^{\ell} h_{0,n}^k (\ell)=1,
\end{equation*}
and hence we conclude \eqref{eq:initial sum to one}. 
To show \eqref{eq:transition sum to one}, note that
\begin{equation*}
    \begin{split}
        \sum_{\lambda\in\mathsf{Sign}_2}\mathsf{p}_{x,x+1}(\mu,\lambda)&=\sum_{\lambda\in\mathsf{Sign}_2}(1-q)^{2 }s_{\lambda/\mu}(q )\frac{h_{x+1,n}^k (\lambda_1-\lambda_2)}{h_{x,n}^k (\mu_1-\mu_2)} \\
        &=\sum_{\lambda\in\mathsf{Sign}_2}(1-q)^{2 }s_{\lambda/\mu}(q )\frac{u_{n-x-1}^{k} (\lambda_1-\lambda_2)}{u_{n-x}^{k} (\mu_1-\mu_2)}\\
                &=\sum_{\lambda\in\mathsf{Sign}_2}s_{\lambda/\mu}(q)
        \frac{\sum_{\nu\in\mathsf{Sign}_2}s_{\nu/(\lambda_1-\lambda_2,0)}((q)^{n-x-1}){\bf 1}_{\{\nu_1-\nu_2=k\}}}
        {\sum_{\nu\in\mathsf{Sign}_2}s_{\nu/(\mu_1-\mu_2,0)}((q)^{n-x}){\bf 1}_{\{\nu_1-\nu_2=k\}}}\\
        &=\sum_{\lambda\in\mathsf{Sign}_2}s_{\lambda/\mu}(q)\frac{\sum_{\nu\in\mathsf{Sign}_2}s_{\nu/\lambda}((q)^{n-x-1}){\bf 1}_{\nu_1-\nu_2=k}}{\sum_{\nu\in\mathsf{Sign}_2}s_{\nu/\mu}((q)^{n-x}){\bf 1}_{\nu_1-\nu_2=k}}\\
        &=\frac{\sum_{\nu\in\mathsf{Sign}_2}{\bf 1}_{\nu_1-\nu_2=k}\sum_{\lambda\in\mathsf{Sign}_2}s_{\lambda/\mu}(q)s_{\nu/\lambda}((q)^{n-x-1}) }{\sum_{\nu\in\mathsf{Sign}_2}{\bf 1}_{\nu_1-\nu_2=k}s_{\nu/\mu}((q)^{n-x}) }=1.
    \end{split}
\end{equation*}
The first step uses  \eqref{eq:Markov chain n}; the second step uses  \eqref{eq:definition of h}; the third step uses \eqref{eq:definition of u}; the fourth step uses the translation invariance as in Lemma \ref{lem:translation invariance} together with the
bijective changes of variables \(\nu\mapsto \nu+\lambda_2(1,1)\) in the
numerator and \(\nu\mapsto \nu+\mu_2(1,1)\) in the denominator, both of which
preserve \(\nu_1-\nu_2\); the fifth step is a change of order of summation; and the final step uses the branching rule \eqref{eq:branching}.

Finally, we verify \eqref{eq:distribution equals to Markov chain}. This follows from 
\begin{equation*}
    \begin{split}
    &\mathsf{p}_0\left(\lambda^0\right)\prod_{x=0}^{n-1} \mathsf{p}_{x,x+1}\left(\lambda^{x},\lambda^{x+1}\right) \\
    &=c^{\lambda_1^0-\lambda_2^0} h_{0,n}^k \left(\lambda_1^0-\lambda_2^0\right) {\bf 1}_{\{\lambda_2^0=0\}}\prod_{x=0}^{n-1}(1-q)^{2 }s_{\lambda^{x+1}/\lambda^x}(q )\frac{h_{x+1,n}^k \left(\lambda_1^{x+1}-\lambda_2^{x+1}\right)}{h_{x,n}^k \left(\lambda_1^x-\lambda_2^x\right)}\\
    &=(1-q)^{2n}c^{\lambda^{0}_1-\lambda^{0}_2}s_{\lambda^{1}/\lambda^{0}}(q)\dots s_{\lambda^{n}/\lambda^{n-1}}(q) {\bf 1}_{\{\lambda_2^0=0\}}h_{n,n}^k\left(\lambda_1^n-\lambda_2^n\right)\\
    &\propto c^{\lambda^{0}_1-\lambda^{0}_2}s_{\lambda^{1}/\lambda^{0}}(q)\dots s_{\lambda^{n}/\lambda^{n-1}}(q) {\bf 1}_{\{\lambda_2^0=0\}} {\bf 1}_{\{\lambda^{n}_1-\lambda^n_2=k\}}.
    \end{split}
\end{equation*}
The first step follows from \eqref{eq:Markov chain n}; the second step follows from telescoping in the product; and the third step follows from 
$h_{n,n}^k(\ell)\propto u_{0}^{k}(\ell)={\bf 1}_{\ell=k}$, in view of $s_{\lambda/(\ell,0)}((q)^0)={\bf 1}_{\lambda=(\ell,0)}$.
\end{proof}

The next lemma defines an infinite Markov chain on $\mathsf{Sign}_2$.
\begin{lemma}\label{lem:limit MC}
Fix $q\in(0,1)$ and $c\in[0,1)$.  Then
    \begin{equation}\label{eq:initial and transition of Markov chain}
        \begin{split}
            \overline{\mathsf{p}}_0\left(\lambda\right)& = (1-c)^2c^{\lambda_1-\lambda_2}\left(\lambda_1-\lambda_2+1\right){\bf 1}_{\{\lambda_2=0\}},\\
            \overline{\mathsf{p}}\left(\mu,\lambda\right)&=(1-q)^2s_{\lambda/\mu}(q)\frac{\lambda_1-\lambda_2+1}{\mu_1-\mu_2+1}\quad\mbox{ for }x \in \mathbb{Z}_{\geq0},
        \end{split}
    \end{equation}
are the initial and transitional laws of a time-homogeneous Markov chain $\{\lambda^j\}_{j=0}^{\infty}$ on state space $\mathsf{Sign}_2$.
\end{lemma}
\begin{proof} We need to show that the right sides of \eqref{eq:initial and transition of Markov chain} are normalized.
     Indeed,
\[
\sum_{\lambda\in\mathsf{Sign}_2}\overline{\mathsf p}_0(\lambda)
=(1-c)^2\sum_{d\ge0}(d+1)c^d=1.
\]
For the transition probabilities, fix $\mu\in\mathsf{Sign}_2$ and denote $\mu_1-\mu_2=d\in\mathbb{Z}_{\geq0}$, $\lambda_1-\mu_1=a$ and $\lambda_2-\mu_2=b$. Then $\lambda\in\mathsf{Sign}_2$ satisfy $s_{\lambda/\mu}(q)>0$ precisely when $a\in\mathbb{Z}_{\geq0}$ and $b\in\llbracket0,d\rrbracket$. Hence,
\[
\sum_{\lambda\in\mathsf{Sign}_2}\overline{\mathsf p}(\mu,\lambda)
=
\frac{(1-q)^2}{d+1}
\sum_{b=0}^{d}\sum_{a=0}^{\infty}
q^{a+b}(d+a-b+1)
=1,
\]
where the last equality follows by evaluating the geometric sums.
\end{proof}

\begin{proposition}\label{thm:local convergence} Fix $q\in (0,1)$, $c\in [0,1)$, and $M>1$. For each $m\in \mathbb{N}$, let $(y_{1,m},y_{2,m})\in \Z^2$ satisfy
\begin{align}\label{cond2}
    y_{1,m}-y_{2,m} \in  [M^{-1}\sqrt{m},M\sqrt{m}].
\end{align}
Suppose $Q_{m}:=(Q_{1,m},Q_{2,m})$ has the law $\Pp_{\mathrm{Inter};q,c}^{m,(y_{1,m},y_{2,m})}$ defined in Definition \ref{def:gibbslaw}.
Extend $Q_m$ to all $j\ge0$ by setting $Q_m(j)=Q_m(m)$ for $j>m$.
Then, as $m\rightarrow\infty$, the law on
$\mathsf{Sign}_2^{\mathbb Z_{\ge0}}$ of
\begin{equation}\label{eq:btrgertg}
\left(Q_{1,m}(j)-Q_{2,m}(0),\,Q_{2,m}(j)-Q_{2,m}(0)\right)_{j\in \Z_{\ge 0}}
\end{equation}
converges weakly to the law of the Markov chain $\{\lambda^j\}_{j=0}^{\infty}$ on state space $\mathsf{Sign}_2$ defined as in Lemma \ref{lem:limit MC}.
\end{proposition} 
\begin{proof}
    Let $\lambda^{j,m}=(\lambda_1^{j,m},\lambda_2^{j,m})\in\mathsf{Sign}_2$ where $\lambda_i^{j,m}=Q_{i,m}(j)$ for $i\in \{1,2\}$ and $j\in\mathbb{Z}_{\geq0}$. In view of Definition \ref{def:gibbslaw} of the measure $\Pp_{\mathrm{Inter};q,c}^{m,(y_{1,m},y_{2,m})}$, we have
\begin{equation} \begin{split}
    \mathbb{P} \left(\lambda^{0,m},\dots,\lambda^{m,m}\right)&\propto c^{\lambda^{0,m}_1-\lambda^{0,m}_2} \cdot q^{-\lambda^{0,m}_1-\lambda^{0,m}_2} \cdot   
    {\bf 1}_{{\left\{\left(\left\{\lambda_1^{j,m}\right\}_{j=0}^m,\left\{\lambda_2^{j,m}\right\}_{j=0}^m\right)\in\Omega_{\mathrm{Inter}}\left(m,\left(y_{1,m},y_{2,m}\right)\right)\right\}}}  \\
    &\propto  c^{\lambda^{0,m}_1-\lambda^{0,m}_2} \cdot  q^{\lambda_1^{m,m}+\lambda_2^{m,m}-\lambda^{0,m}_1-\lambda^{0,m}_2}   \cdot {\bf 1}_{\left\{\lambda^{0,m}\preceq\dots\preceq\lambda^{m,m} =\left(y_{1,m},y_{2,m}\right)\right\}}\\
    &=  c^{\lambda^{0,m}_1-\lambda^{0,m}_2}  \cdot s_{\lambda^{1,m}/\lambda^{0,m}}(q)\dots s_{\lambda^{m,m}/\lambda^{m-1,m}}(q) \cdot {\bf 1}_{\left\{\lambda^{m,m}=\left(y_{1,m},y_{2,m}\right)\right\}}.
\end{split}\end{equation}
Let $k_m=y_{1,m}-y_{2,m}$.
By the translation invariance in Lemma \ref{lem:translation invariance}, 
the process \eqref{eq:btrgertg} has the law
 \begin{equation}\label{eq:sample again}
       \mathbb{P}  (\lambda^{0},\dots,\lambda^{m}) \propto c^{\lambda^{0}_1-\lambda^{0}_2}s_{\lambda^{1}/\lambda^{0}}(q)\dots s_{\lambda^{m}/\lambda^{m-1}}(q) {\bf 1}_{\{\lambda_2^0=0\}} {\bf 1}_{\{\lambda^{m}_1-\lambda^m_2=k_m\}},
    \end{equation}
which, by Lemma \ref{lem:MC}, can be described as a Markov chain with initial and transition laws
\begin{equation}\label{eq:Markov chain m}
        \begin{split}
            \mathsf{p}_0(\lambda)& =  c^{\lambda_1-\lambda_2} h_{0,m}^{k_m} (\lambda_1-\lambda_2) {\bf 1}_{\{\lambda_2=0\}},\\
            \mathsf{p}_{x,x+1}(\mu,\lambda)&=(1-q)^{2 }s_{\lambda/\mu}(q )\frac{h_{x+1,m}^{k_m} (\lambda_1-\lambda_2)}{h_{x,m}^{k_m} (\mu_1-\mu_2)} \quad\mbox{ for }x \in \llbracket 0,m-1 \rrbracket.
        \end{split}
\end{equation} 
To show that this process converges weakly to the Markov chain as in Lemma \ref{lem:limit MC}, it suffices to show that
\begin{equation}\label{eq:limit of h only need to show}
    \lim_{m\rightarrow\infty}h_{x,m}^{k_m}(\ell)=(\ell+1)(1-c)^2\quad\mbox{ for each }x,\ell\in\mathbb{Z}_{\geq0}.
\end{equation}
Recall $k_m/\sqrt m\in[M^{-1},M]$. It suffices to prove
\eqref{eq:limit of h only need to show} along each subsequence. For each subsequence, there is a further subsequence (which we will still denote as $k_m$ for notational convenience), such that $k_m/\sqrt m\to h\in[M^{-1},M]$.
By the definition of $h_{x,n}^k(\ell)$ as in \eqref{eq:definition of h}, it suffices to show that
\begin{align}
 \label{eq:gverge1}   \lim_{m\to\infty}\frac{m-x}{\ell+1}\,u_{m-x}^{k_m}(\ell)&= C(h):=\frac{h(1-q)^3}{2\sqrt{\pi}\,q^{3/2}}
\exp\!\left(-\frac{h^2(1-q)^2}{4q}\right),\quad \mbox{ for each }x,\ell\in\mathbb{Z}_{\geq0},\\
\label{eq:gverge2}\text{and   }\lim_{m\to\infty}m Z_m&=\frac{C(h)}{(1-c)^2},\quad\mbox{ where we denote }\quad Z_m:=\sum_{j\ge0}c^j u_m^{k_m}(j).
\end{align}
We begin by establishing \eqref{eq:gverge1}.
By Lemma \ref{lem:formula u and h}, we have
\begin{equation}\label{eq:brbternyr5}
\frac{m-x}{\ell+1}u_{m-x}^{k_m}(\ell)
=
\frac{m-x}{\ell+1} \frac1\pi\int_0^{2\pi}
\sin((\ell+1)\theta)\sin((k_m+1)\theta)
\left(1+\frac{4q}{(1-q)^2}\sin^2\frac{\theta}{2}\right)^{-m+x}
\mathsf{d}\theta .
\end{equation}
Fix $\delta\in(0,\pi)$. On the interval $\theta\in[\delta,2\pi-\delta]$, observe that $1+\frac{4q}{(1-q)^2}\sin^2\frac{\theta}{2}$ is uniformly bounded below by a constant strictly larger than $1$, therefore contribution to \eqref{eq:brbternyr5} from $[\delta,2\pi-\delta]$ converges exponentially fast to $0$.
On the interval $[0,\delta]$, after making the change of variable $\theta=\alpha/\sqrt{m-x}$, 
the contribution to \eqref{eq:brbternyr5} from $\theta\in[0,\delta]$ equals
\[
\frac1\pi
\int_0^{\delta\sqrt{m-x}}
\frac{\sqrt{m-x}\,\sin((\ell+1)\alpha/\sqrt{m-x})}{\ell+1}
\sin\!\left(\frac{(k_m+1)\alpha}{\sqrt{m-x}}\right)
\left(1+\frac{4q}{(1-q)^2}
\sin^2\frac{\alpha}{2\sqrt{m-x}}\right)^{-m+x}
\mathsf{d}\alpha .
\]
For each fixed \(\alpha\ge0\), as $m\to\infty$, we have
\begin{align}
        &\label{eq:nrtre1}\frac{\sqrt{m-x}\,\sin\left((\ell+1) \alpha/\sqrt{m-x}\right)}{\ell+1}\to \alpha,\\
        &\label{eq:nrtre2}\sin\!\left(\frac{(k_m+1)\alpha}{\sqrt{m-x}}\right)\to \sin(h\alpha),\quad \text{and }
\left(1+\frac{4q}{(1-q)^2}
\sin^2\frac{\alpha}{2\sqrt{m-x}}\right)^{-m+x}
\to \exp\left(-\frac{q}{(1-q)^2}  \alpha^2\right).
\end{align}
Since $\delta<\pi$, there are constants $C,c_*>0$, independent of $m$ and
$\alpha$, such that for all $0\le \alpha\le \delta\sqrt{m-x}$,
\begin{equation*}
    \begin{split}
        &\left(1+\frac{4q}{(1-q)^2}
\sin^2\frac{\alpha}{2\sqrt{m-x}}\right)^{-m+x}
\le C e^{-c_*\alpha^2},\\
&\left|
\frac{\sqrt{m-x}\,\sin((\ell+1)\alpha/\sqrt{m-x})}{\ell+1}
\right|\le \alpha,
\qquad \text{and}\qquad
\left|\sin\!\left(\frac{(k_m+1)\alpha}{\sqrt{m-x}}\right)\right|\le 1.
    \end{split}
\end{equation*}
Thus the rescaled integrand is dominated by \(C_1\alpha e^{-c_1\alpha^2}\),
which is integrable on \([0,\infty)\). By the dominated convergence theorem, the contribution to \eqref{eq:brbternyr5} from \(\theta\in[0,\delta]\) converges to
\begin{equation}\label{eq:bsrtnrwt}
\frac1\pi\int_0^\infty \alpha\sin(h\alpha)\exp\left(-\frac{q}{(1-q)^2}  \alpha^2\right)\mathsf{d}\alpha .
\end{equation}
Since the integrand of \eqref{eq:brbternyr5} is invariant under $\theta\mapsto2\pi-\theta$,
the contribution to \eqref{eq:brbternyr5} from \(\theta\in[2\pi-\delta,2\pi]\) converges to the same quantity \eqref{eq:bsrtnrwt}. 
In conclusion, we have
\[
 \lim_{m\to\infty}\frac{m-x}{\ell+1}\,u_{m-x}^{k_m}(\ell)=\frac2\pi\int_0^\infty \alpha\sin(h\alpha)\exp\left(-\frac{q}{(1-q)^2}  \alpha^2\right)\mathsf{d}\alpha=C(h),
\]
and we conclude the proof of \eqref{eq:gverge1}. The proof of \eqref{eq:gverge2} is largely parallel. By Lemma \ref{lem:formula u and h}, 
\[mZ_m=\frac m\pi\int_0^{2\pi}
\frac{\sin\theta}{1-2c\cos\theta+c^2}
\sin((k_m+1)\theta)
\left(1+\frac{4q}{(1-q)^2}\sin^2\frac{\theta}{2}\right)^{-m}
\mathsf{d}\theta.\]
The contribution from $\theta\in[\delta,2\pi-\delta]$ is exponentially small. On $\theta\in[0,\delta]$, we change $\theta=\alpha/\sqrt{m}$. Observe
\[
\sqrt m\,
\frac{\sin(\alpha/\sqrt m)}
{1-2c\cos(\alpha/\sqrt m)+c^2}
\to
\frac{\alpha}{(1-c)^2},
\]
and, as in \eqref{eq:nrtre2}, the remaining two factors converge pointwise to $\sin(h\alpha)\exp\left(-\frac{q}{(1-q)^2}  \alpha^2\right)$.
The integrand is dominated by \(C_2\alpha e^{-c_2\alpha^2}\) for constants $C_2,c_2>0$.
Hence the contribution from \(\theta\in[0,\delta]\) converges to
\[
\frac1{\pi(1-c)^2}
\int_0^\infty \alpha\sin(h\alpha)\exp\left(-\frac{q}{(1-q)^2}  \alpha^2\right)\mathsf{d}\alpha.
\]
The contribution from $\theta\in[2\pi-\delta,2\pi]$ converges to the same limit. Hence 
\[
\lim_{m\to\infty}mZ_m
=
\frac{2}{\pi(1-c)^2}
\int_0^\infty \alpha\sin(h\alpha)\exp\left(-\frac{q}{(1-q)^2}  \alpha^2\right)\mathsf{d}\alpha
=
\frac{C(h)}{(1-c)^2},
\]
and we conclude the proof of \eqref{eq:gverge2}.
\end{proof}

We next prove the recentered convergence for the first layer of the interlacing Gibbs law. 
\begin{proposition}\label{intconv} Fix $q\in (0,1)$, $c\in [0,1)$, and $M>1$. For each $m\in \mathbb{N}$, let $(y_{1,m},y_{2,m})\in \Z^2$ satisfy 
\[
    y_{1,m}-y_{2,m} \in  [M^{-1}\sqrt{m},M\sqrt{m}].
\]
Suppose $Q_{m}:=(Q_{1,m},Q_{2,m})$ has the law $\Pp_{\mathrm{Inter};q,c}^{m,(y_{1,m},y_{2,m})}$ defined in Definition \ref{def:gibbslaw}. Then as $m\to \infty$, $$(Q_{1,m}(i)-Q_{1,m}(0))_{i\in \N} \stackrel{d}{\to} (f(i))_{i\in \N},$$ where $(f(i))_{i\in \N}$ has the law of $\mu_{c,1}$ defined in Definition \ref{def:mucs}.
    \end{proposition}

\begin{proof}
     Given Proposition \ref{thm:local convergence}, it suffices to the show that the marginal law of the increments along the first layer of the Markov chain defined in \eqref{eq:initial and transition of Markov chain} matches with $\mu_{c,1}$. This is essentially known; see the discussion at the end of Section 2 in \cite{barraquand2024integral}. The proof in that paper follows from results in O'Connell \cite{o2003conditioned} which uses the Robinson–Schensted–Knuth (RSK) correspondence.
    We give an independent proof for completeness.\\

We recall from Definition \ref{def:mucs} that $\mu_{c,1}$ is the law of $\bigl(f(k)\bigr)_{k\in \N}$, where
\[
f(k) = S_2(k) + \Bigl(\max_{1 \le \ell \le k} [S_1(\ell) - S_2(\ell - 1)] - Y\Bigr)^+.
 \]
Here, $S_1$ and $S_2$ are $\Geo(q)$ random walks starting at $0$, $Y \sim \Geo (c)$, and $S_1,S_2,Y$ are independent.

      Now, assume $U,Y\sim\Geo(c)$ and $A_j,B_j\sim\Geo(q)$ for $j\in\mathbb{N}$, and all of these random variables are mutually independent. Define an auxiliary Markov process $\bigl((\widetilde\lambda^j,V_j)\bigr)_{j=0}^{\infty}$ on the state space $\mathsf{Sign}_2\times\mathbb{Z}_{\geq0}$ by
    \begin{equation}\label{eq:MC}
    \begin{aligned} 
        &\widetilde\lambda^0=(U+W,0),\quad V_0=Y,\\
        &\widetilde\lambda^{j+1}_2-\widetilde\lambda^j_2=\min\{A_{j+1},V_j\},\quad
\widetilde\lambda^{j+1}_1-\widetilde\lambda^j_1=B_{j+1}+(A_{j+1}-V_j)^+,\quad \text{and} \\
&V_{j+1}=B_{j+1}+(V_j-A_{j+1})^+\mbox{ for }j\in \Z_{\ge 0}.
    \end{aligned}\end{equation}
We split the rest of the proof into two steps. In the first step, we show that the marginal law of the increments along the first layer $(\widetilde\lambda_1^j-\widetilde\lambda_1^0)_{j=1}^{\infty}$ is the measure $\mu_{c,1}$. In the second step, we show that the process $(\widetilde\lambda^j)_{j=0}^{\infty}$ is equal in distribution to the process $(\lambda^j)_{j = 0}^\infty$ defined by \eqref{eq:initial and transition of Markov chain}.\\

{\bf \raggedleft Step 1.} In this step, we show that $(\widetilde\lambda_1^j-\widetilde\lambda_1^0)_{j=1}^{\infty}$ has law $\mu_{c,1}$.
Define 
\[
S_1(j):=\sum_{i=1}^j A_i,\quad\text{and}\quad S_2(j):=\sum_{i=1}^j B_i\quad\mbox{ for }j\in \Z_{\ge 0},
\]
where $S_1(0)=S_2(0)=0$. Then $S_1$ and $S_2$ are independent $\Geo(q)$ random walks starting at $0$, and $W\sim\Geo(c)$ is independent of $(S_1,S_2)$. Let
\begin{equation}\label{eq:bwtgetrg0}
R_j:=\widetilde\lambda_1^j-\widetilde\lambda_1^0-S_2(j)\quad\mbox{ for }j\in \Z_{\ge 0}.
\end{equation}
Then $R_0=0$, and from the recursion for $\widetilde\lambda_1$ in \eqref{eq:MC}, we have 
\begin{equation}\label{eq:bwtgetrg}
R_j-R_{j-1}=(A_j-V_{j-1})^+\quad\mbox{ for }j\in \N.
\end{equation}
Therefore, by the recursion for $V$ in \eqref{eq:MC}, we have 
\begin{equation*} 
V_j=B_j+(V_{j-1}-A_j)^+=V_{j-1}+B_j-A_j+(A_j-V_{j-1})^+=V_{j-1}+B_j-A_j+R_j-R_{j-1}.
\end{equation*}
Then, subtracting $V_{j-1}$ from both sides and summing over $j\in\llbracket1,k\rrbracket$ (recall $V_0 = Y$), we have
\begin{equation}\label{eq:bwtgetrg3}
V_k=Y+S_2(k)-S_1(k)+R_k.
\end{equation}
In view of \eqref{eq:bwtgetrg} for $j=k$ and \eqref{eq:bwtgetrg3} for $k-1$, we have
\[
R_k=R_{k-1}+\left(S_1(k)-S_2(k-1)-Y-R_{k-1}\right)^+=\max\{R_{k-1},\,S_1(k)-S_2(k-1)-Y\}\quad\mbox{ for }k\geq1.
\]
Since $R_0=0$, iteration yields
\[R_k=\left(\max_{\ell \in \llbracket 1,k \rrbracket}[S_1(\ell)-S_2(\ell-1)]-Y\right)^+.\]
Therefore, for every $k\in \N$, by applying \eqref{eq:bwtgetrg0} for $j=k$, we have
\[\widetilde\lambda_1^k-\widetilde\lambda_1^0=S_2(k)+R_k=S_2(k)+\Bigl(\max_{\ell \in \llbracket 1,k \rrbracket}[S_1(\ell)-S_2(\ell-1)]-Y\Bigr)^+.\]
Hence, the law of $(\widetilde\lambda_1^j-\widetilde\lambda_1^0)_{j=1}^{\infty}$ equals $\mu_{c,1}$.\\

{\bf \raggedleft Step 2.} In this step, we show $(\widetilde\lambda^j)_{j=0}^{\infty}$ is equal in distribution to the Markov chain $(\lambda^j)_{j = 0}^\infty$ defined by \eqref{eq:initial and transition of Markov chain}. First, we prove that $\wt \lambda^j\in\mathsf{Sign}_2$ for $j\in \Z_{\ge 0}$. Since $V_j \ge 0$ (see \eqref{eq:MC}), it suffices to prove that $d_j:=\widetilde\lambda^j_1-\widetilde\lambda^j_2 \ge V_j$ for all $j \in \Z_{\ge 0}$. We do this by induction.  When $j=0$, this follows from $V_0=Y\leq U+Y=d_0$. Assume this inequality is true for $j$, then from the recursive definitions in \eqref{eq:MC},
\begin{align*}
d_{j+1} - V_{j+1} &= \wt \lambda_1^{j+1} - \wt \lambda_2^{j+2} - (B_{j+1} + (V_j - A_{j+1})^+) \\
&= \wt \lambda_1^{j+1} - \wt \lambda_1^j - (\wt \lambda_2^{j+1} - \wt \lambda_2^j) + \wt \lambda_1^j - \wt \lambda_2^j - (B_{j+1} + (V_j - A_{j+1})^+) \\
&= B_{j+1} + (A_{j+1} - V_j)^+ - (\min\{A_{j+1},V_j\}) + d_j - (B_{j+1} + (V_j - A_{j+1})^+) \\
&= d_j - V_j + (A_{j+1} - V_j)^+ \ge 0.
\end{align*}
and hence it is holds for $j+1$. We conclude $V_j\le d_j$ for all $j\in \Z_{\ge 0}$ by induction, as desired. 

Next, we claim that for every $m\geq0$, every $\lambda^0,\dots,\lambda^m\in\mathsf{Sign}_2$ and every $v\ge0$,
\begin{equation}\label{eq:path-identity-aux}
\mathbb P\left(\widetilde\lambda^0=\lambda^0,\dots,\widetilde\lambda^m=\lambda^m,\ V_m=v\right)
=
\frac{{\bf 1}_{\{v \in \llbracket 0,\lambda^m_1-\lambda^m_2\rrbracket\}}}{\lambda^m_1-\lambda^m_2+1}\,
\overline{\mathsf p}_0(\lambda^0)\prod_{r=0}^{m-1}\overline{\mathsf p}(\lambda^r,\lambda^{r+1}).
\end{equation}
Note that, if \eqref{eq:path-identity-aux} holds, then, by summing over 
$v\in\llbracket0,\lambda_1^m-\lambda_2^m\rrbracket$, we see that
$(\widetilde\lambda^j)_{j=0}^{\infty}$ is equal in distribution to 
$(\lambda^j)_{j=0}^{\infty}$.

In the rest of the proof we establish \eqref{eq:path-identity-aux} by induction. For $m=0$, if \(\lambda^0=(n,0)\) and \(v \in \llbracket 0,n \rrbracket\), then
\[\mathbb P(\widetilde\lambda^0=\lambda^0,\ V_0=v)=
\mathbb P(U=n-v,\ Y=v)=(1-c)^2c^n=\frac{\overline{\mathsf p}_0(\lambda^0)}{n+1},
\]
which is exactly \eqref{eq:path-identity-aux}; otherwise both sides vanish. 

Assume that \eqref{eq:path-identity-aux} holds for $m$. Fix $\lambda^0,\dots,\lambda^{m+1}\in\mathsf{Sign}_2$ and write
\[
\lambda^m=(x,y),\quad \lambda^{m+1}=(x+a,y+b),\quad d=x-y.
\]
We claim that both sides of \eqref{eq:path-identity-aux} at time \(m+1\) vanish unless \(a,b\ge0\) and \(b\le d\). Indeed, from the recursion for $\widetilde\lambda_1$ and $\widetilde\lambda_2$ in \eqref{eq:MC}, we have
\[0\leq\widetilde\lambda^{m+1}_2-\widetilde\lambda^m_2=\min\{A_{m+1},V_m\}\leq V_m\leq d_m=\widetilde\lambda^m_1-\widetilde\lambda^m_2,\quad\mbox{and}\quad
0\leq\widetilde\lambda^{m+1}_1-\widetilde\lambda^m_1,\]
which means that the left side of \eqref{eq:path-identity-aux} vanishes unless \(a,b\ge0\) and \(b\le d\). The right side of \eqref{eq:path-identity-aux} vanishes unless $\lambda^m\preceq \lambda^{m+1}$, which is also equivalent to \(a,b\ge0\) and \(b\le d\). 

We next assume \(a,b\ge0\) and \(b\le d\). Since $V_{m+1}\leq d_{m+1}=d+a-b$, we fix $v'\in\llbracket0,d+a-b\rrbracket$. 
By the recursion \eqref{eq:MC}, given $\wt \lambda^m = \lambda^m$, the event
$\{\widetilde\lambda^{m+1}=\lambda^{m+1},\ V_{m+1}=v'\}$
holds if and only if we have the following equalities:
\[
\min\{A_{m+1},V_m\}=b,\quad
B_{m+1}+(A_{m+1}-V_m)^+=a,\quad
B_{m+1}+(V_m-A_{m+1})^+=v'.
\]
This system of equations has exactly one solution, namely 
\[
(V_m,A_{m+1},B_{m+1})= (a^\star,b^\star,v^\star) :=
\begin{cases}
(b,\ b+a-v',\ v'), & \mbox{ if } 0\le v'\le a,\\ 
(b+v'-a,\ b,\ a), & \mbox{ if } a<v'\le d+a-b.
\end{cases}
\]
Using the induction hypothesis and the independence of 
$(A_{m+1},B_{m+1})$ from $(\widetilde\lambda^0,\dots,\widetilde\lambda^m,V_m)$, we obtain
\begin{equation*}
    \begin{split}
        &\quad \;\mathbb P\left(\widetilde\lambda^0=\lambda^0,\dots,\widetilde\lambda^{m+1}=\lambda^{m+1},\ V_{m+1}=v'\right) \\
        &=\mathbb P\left(\widetilde\lambda^0=\lambda^0,\dots,\wt \lambda^m = \lambda^m,V_m = v^\star\right)\Pp(A_{m+1} = a^\star,B_{m+1} = b^\star) \\
        &=\frac{1}{d+1}\overline{\mathsf p}_0(\lambda^0)\prod_{r=0}^{m-1}\overline{\mathsf p}(\lambda^r,\lambda^{r+1})\cdot (1-q)^2q^{a+b}\\
        &=\frac{1}{d+a-b+1} \overline{\mathsf p}_0(\lambda^0)\prod_{r=0}^{m}\overline{\mathsf p}(\lambda^r,\lambda^{r+1}),
    \end{split}
\end{equation*}
which is exactly \eqref{eq:path-identity-aux} at time \(m+1\). In the last equality, we have used the fact that
\[
\overline{\mathsf p}(\lambda^m,\lambda^{m+1})
=
(1-q)^2q^{a+b} (d+a-b+1)/(d+1).
\]
By induction, this completes the proof of \eqref{eq:path-identity-aux} for all $m\in\Z_{\geq0}$. 
\end{proof}

\subsection{Separation estimates} \label{sec:techlemmas} 
Recall the half-space geometric line ensemble from Definition \ref{def:le}. In this section, we establish separation estimates between the curves of the line ensemble. This essentially follows from the scaling limit results for $(L_i^n)_{i \in \N}$, recently established in \cite{dy25,ddy26}.

\begin{proposition}\label{prop:scale_limits} For each $n\in \mathbb{N}$, consider the half-space geometric line ensemble $(L_i^n)_{i\in \N}$ from Definition \ref{def:le}. We have the following.
\begin{enumerate} [label=(\roman*), font=\normalfont]
\item \label{it:scale1} If $c\in [0,1]$, then as $n\to \infty$ we have the following weak convergence  
        \begin{align}\label{convsub}
            \left(n^{-1/3}L_i^n(tn^{2/3})-\frac{2qn^{2/3}}{1-q}-\frac{qtn^{1/3}}{1-q}\right)_{i\in \N} \stackrel{d}{\to} \frac{\sqrt{q}}{1-q}\left(\mathcal{P}_i^{1-\lfloor c\rfloor}(t)\right)_{i\in \N}
        \end{align}
  under the topology of uniform convergence on compact sets. Here, $\left(\mathcal{P}_i^0(t)\right)_{i\in \N}$ and $\left(\mathcal{P}_i^1(t)\right)_{i\in\N}$ are certain line ensembles (defined in \cite{dy25,dsy26,ddy26}) that arises as universal scaling limits
  in the half-space KPZ universality class.

\item \label{it:scale2} If $c\in (1,q^{-1})$, then as $n\to \infty$ we have
  \begin{align}\label{llnconv}
      n^{-1}L_1^n(0) \stackrel{p}{\to} \frac{q(c^2-2qc+1)}{(c-q)(1-qc)}, \ \mbox{ and } \ n^{-1}L_{2}^n(\lfloor n^{2/3}\rfloor) \stackrel{p}{\to} \frac{2q}{1-q}.
  \end{align}
\end{enumerate}
\end{proposition}

\begin{proof} Item \ref{it:scale1} for $c\in [0,1)$ and $c=1$ follows from \cite[Theorem 1.5]{ddy26} and \cite[Theorem 1.4]{dy25} and its proof respectively. Item \ref{it:scale2} follows from stronger process-level convergence established in  \cite[Theorems 1.5(b) and 1.11]{ddy26}.
\end{proof}

The line ensembles $\left(\mathcal{P}_i^0(t)\right)_{i\in \N}$ and $\left(\mathcal{P}_i^1(t)\right)_{i\in \N}$, after a parabolic shift and appropriate rescaling, are known as the critical and pinned half-space Airy line ensembles, defined in \cite{dy25} and \cite{dsy26}, respectively. These ensembles admit a Pfaffian point process structure with an explicit correlation kernel, as well as certain Brownian resampling properties. However, none of these features will be needed in the present work. Instead, we require only a certain ordering property of these line ensembles, which we record below.

\begin{proposition}\label{prop:order} $\mathcal{P}_i^j(\cdot)$ is a random continuous function on $[0,\infty)$ for each $i\in \mathbb{N}$ and $j=0,1$. We have the following ordering property.
\begin{align}\label{eq:orderbulk}
    \mathbb{P}\left(\mathcal{P}_1^j(t) > \mathcal{P}_2^j(t) > \mathcal{P}_3^j(t) \mbox{ for all }t\in (0,\infty)\right)=1, \mbox{ for }j=0,1,
\end{align}
\begin{align}\label{eq:orderbdy}
    \mathbb{P}\left(\mathcal{P}_1^0(0) > \mathcal{P}_2^0(0) > \mathcal{P}_3^0(0)\right)=1,\mbox{ and } \mathbb{P}\left(\mathcal{P}_1^1(0) = \mathcal{P}_2^1(0) > \mathcal{P}_3^1(0)\right)=1.
\end{align}
    \end{proposition}

 \begin{proof}\eqref{eq:orderbulk} follows from \cite[Theorem 1.4]{dy25} when $j=1$ and from \cite[Theorem 1.20]{dsy26} when $j=0$. The first part of \eqref{eq:orderbdy} also follows from \cite[Theorem 1.4]{dy25}. The second part of \eqref{eq:orderbdy} is a consequence of \cite[Theorem 1.26]{dsy26} and \cite[Proposition 3.5]{rrv11}.    
 \end{proof}   

We use the above two propositions to deduce the following separation estimates for the top three curves of the line ensemble.

\begin{lemma}\label{bnevent} Fix $\varepsilon\in (0,1)$.
    There exists $\delta,\rho_0,R>0$ and $N\in \mathbb{N}$ all depending on $\varepsilon,q,c$ such that
 for all $n\ge N$ and $m \in \llbracket 0,\rho_0 n^{2/3}\rrbracket$, the following holds.
 \begin{enumerate}[label=(\roman*), font=\normalfont]
 \item \label{it:sep_it1} When $c\in [0,1)$,
     \begin{align}\label{probest01}
        \Pp\left( L_{2}^n(m)+\frac{(j-m)q}{1-q} \ge L_{3}^n(j)+n^{1/3}\delta \mbox{ for all }j\in \llbracket0,m\rrbracket\right) \ge 1-\varepsilon,
    \end{align}
    and
    \begin{align}\label{probest00}
        \Pp\left( \sup_{j\in \llbracket 0, n^{2/3}\rrbracket} L_{2}^n(j)-L_3^n(j) \le R n^{1/3}\right) \ge 1-\varepsilon.
    \end{align}
     \item \label{it:sep_it2} When $c\in [1,q^{-1})$,
     \begin{align}\label{probest02}
        \Pp\left( L_{1}^n(m)+\frac{(j-m)q}{c-q} \ge L_{2}^n(j)+n^{1/3}\delta \mbox{ for all }j\in \llbracket0,m\rrbracket\right) \ge 1-\varepsilon.
    \end{align}
 \end{enumerate}
\end{lemma}

\begin{proof} By the weak convergence in Proposition \ref{prop:scale_limits}\ref{it:scale1}, \eqref{probest00} is immediate. By continuity of the processes $\mathcal{P}_{j+1}^j, \mathcal{P}_{j+2}^j$ for $j=0,1$ and strict ordering of the processes $\mathcal{P}_{j+1}^j,\mathcal{P}_{j+2}^j$  at the origin (Proposition \ref{prop:order}), there exists $\delta,\rho_0$ depending on $\varepsilon$ such that
 for all $\rho \in [0,\rho_0]$,
\begin{align*}
        \Pp\left(\sup_{y\in [0,\rho]}\mathcal{P}_{j+1}^j(\rho)-\mathcal{P}_{j+2}^j(y) \ge 2\delta\right) \ge 1-\varepsilon/2.
    \end{align*}
In view of the weak convergence in Proposition \ref{prop:scale_limits}\ref{it:scale1}, for $c\in [0,1)$ the above estimate implies that for all large enough $n$, we have
\begin{align*}
    \Pp\left(L_{2}^n(m)+\frac{(j-m)q}{1-q} \ge L_{3}^n(j)+n^{1/3}\delta \mbox{ for all }j\in \llbracket0,m\rrbracket\right)\ge 1-\varepsilon
\end{align*}
for all $m\in \llbracket 0,\rho_0n^{2/3}\rrbracket$. Whereas when $c=1$, for all large enough $n$ we have
\begin{align*}
    \Pp\left( L_{1}^n(m)+\frac{(j-m)q}{1-q} \ge L_{2}^n(j)+n^{1/3}\delta \mbox{ for all }j\in \llbracket0,m\rrbracket\right)\ge 1-\varepsilon
\end{align*}
 for all $m\in \llbracket 0,\rho_0n^{2/3}\rrbracket$.  This verifies \eqref{probest01} and the $c=1$ case of Item \ref{it:sep_it2}. Now, we assume $c\in (1,q^{-1})$. Recall the law of large number result from Proposition \ref{prop:scale_limits}\ref{it:scale1}. Note that the in-probability limit of $n^{-1}L_1^n(0)$ is larger than the in-probability limit of $n^{-1}L_2^n(\lfloor n^{2/3}\rfloor)$.  Thus, there exists $\delta>0$, $N\in \mathbb{N}$ depending on $\varepsilon,q,c$ such that for all $n\ge N$,
 \begin{align*}
        \Pp\left( L_{1}^n(0)-\frac{n^{2/3}q}{c-q} \ge L_{2}^n(\lfloor n^{2/3}\rfloor)+n^{1/3}\delta \right) \ge 1-\varepsilon.
    \end{align*}
    As $L_1^n, L_2^n$ are non-decreasing, the above estimate implies \eqref{probest02} with $\rho_0=1$.
\end{proof}

\subsection{Proof of Proposition \ref{prop:recentered_conv}} \label{subsec:recentered_conv} In this section, we prove Proposition \ref{prop:recentered_conv}. 
While the overarching idea of the proof is same for all $c\in [0,q^{-1})$, there are certain technical differences in the $c\in [0,1)$ and $c\in [1,q^{-1})$ cases due to the pinning phenomena in the former case. We therefore set an additional parameter 
\begin{align*}
    d:=\begin{cases}
        0 & \mbox{ if } c\in [0,1)\\
        1 & \mbox{ if } c\in [1,q^{-1})
    \end{cases}
\end{align*}
for convenience.\\

Fix any $k\in \mathbb{N}$ and $(x_1,\ldots,x_k)\in \mathbb{R}^k$. Recall the probability measure $\mu_{c,r_c}$ from Definition \ref{def:mucs}, and let
\begin{align}\label{def:gammac}
    \gamma := \mu_{c,r_c}\left(\bigcap_{i=1}^k \{f(i) \le x_i\}\right).
\end{align}
Recall from Definition \ref{def:mucs} that when $c \ge 1$, we have $r_c = c$, and $\mu_{c,r_c}$ is a $\Geo(qc^{-1})$ random walk. For any $\Lambda \supset \llbracket 0,k\rrbracket$, and a  process $X : \Lambda \to \mathbb{R}$, let us define
\begin{align}\label{def:ax}
    \mathsf{A}(X):=\bigcap_{i=1}^k\{X(\ell)-X(0)\le x_i\}.
\end{align}
Recall the line ensemble $(L_i^n)_{i\in \N}$ in Definition \ref{def:le} and the identity in distribution therein from Proposition \ref{prop:laweq}\ref{it:G_L_law}. Given the above notations, our goal is to show
\begin{align}\label{toshow}
    \lim_{n\to\infty} \mathbb{P}_c\left(\mathsf{A}(L_1^n)\right) = \gamma,
\end{align}
where $\gamma$ is defined in \eqref{def:gammac}, and $\Pp_c$ is the probability measure of the line ensemble.

Fix any $\varepsilon \in (0,1)$. Take $\delta,\rho_0, R$ from Lemma \ref{bnevent}. Set $M_0=8R\delta^{-1}$. With this choice of $M_0$, get $\lambda$ from  Lemma \ref{intkolm}, and set $$m=\lfloor n^{2/3}\cdot \min(\delta^2/16\lambda^2,\rho_0)\rfloor.$$  
Let $U_{n,d}$ be the set of $(\vec{y},g\llbracket0,m\rrbracket) \in \mathbb{R}^{2-d}\times \mathbb{R}^{m+1}$ such that $g(0) \le g(1) \le \ldots \le g(m)$, and
\begin{align*}
    0 \le y_1-y_{2-d} \le Rn^{1/3}, \quad y_{2-d}+\frac{(j-m)q}{1-q}  \ge g(j)+n^{1/3}\delta, \mbox { for all } j\in \llbracket 0,m\rrbracket.
\end{align*}
Then,
 \begin{itemize}
     \item ($d=0$ case) By Lemma \ref{bnevent}\ref{it:sep_it1}, 
\begin{align}\label{bnkpz}
    \Pp_c(\widetilde{\mathsf{B}}_{n,0})\ge 1-\varepsilon, \mbox{ where }\widetilde{\mathsf{B}}_{n,0}:=\{(L_1^n(m),L_2^n(m),L_3^n\llbracket0,m\rrbracket)\in U_{n,0}\}
\end{align}
for all large $n$. 
By the ordering property in the bulk \eqref{eq:orderbulk} and the weak convergence in Proposition \ref{prop:scale_limits}\ref{it:scale1}, we can find $M>1$ such that, for all large enough $n$,
\begin{align}\label{cnkpz}
    \Pp_c(\mathsf{C}_{n,0})\ge 1-\varepsilon, \mbox{ where }\mathsf{C}_{n,0}:=\left\{L_1^n(m)-L_2^n(m)\in [M^{-1}\sqrt{m},M\sqrt{m}]\right\}.
\end{align}
Set $\mathsf{B}_{n,0}:=\widetilde{\mathsf{B}}_{n,0}\cap\mathsf{C}_{n,0}$.

\item ($d=1$ case) By Lemma \ref{bnevent}\ref{it:sep_it2}, for all large enough $n$, 
\begin{align}\label{bncri}
    \Pp_c(\widetilde{\mathsf{B}}_{n,1})\ge 1-\varepsilon, \mbox{ where }\widetilde{\mathsf{B}}_{n,1}:=\{(L_1^n(m),L_2^n\llbracket0,m\rrbracket)\in U_{n,1}\}.
\end{align}
Set $\mathsf{C}_{n,1}=\Omega$ (the entire sample space) and $\mathsf{B}_{n,1}=\widetilde{\mathsf{B}}_{n,1}\cap \mathsf{C}_{n,1}=\widetilde{\mathsf{B}}_{n,1}$.
 \end{itemize}

In view of \eqref{bnkpz}, \eqref{cnkpz}, and \eqref{bncri} we have that
\begin{align*}
   |\Pp_c(\mathsf{A}(L_1^n)) -  \Pp_c(\mathsf{A}(L_1^n)\cap \mathsf{B}_{n,d})| \le 2\varepsilon
\end{align*}
for all large $n$. 
We claim that for all large $n$,
\begin{align}\label{claim}
\gamma-4\varepsilon  \le \Pp_c(\mathsf{A}(L_1^n)\cap \mathsf{B}_{n,d}) \le  \left(1-\varepsilon\right)^{-1}(\gamma+2\varepsilon). 
\end{align}
As $\varepsilon$ is arbitrary, taking $n\to \infty$ followed by $\varepsilon \downarrow 0$ in the above two equation leads to \eqref{toshow}. Thus, it suffices to prove the above claim. Towards this end, we invoke the resampling property from Proposition \ref{prop:laweq}. By the tower property of conditional expectation, we have 
\begin{align*}
    \Pp_c\left(\mathsf{A}(L_1^n)\cap\mathsf{B}_{n,d} \right) = \Ex_c\left[\ind_{\mathsf{B}_{n,d}}\Ex_c\left[\ind_{\mathsf{A}(L_1^n)}|\mathcal{F}_d\right]\right],
\end{align*}
where $\mathcal{F}_d:=\sigma(L_i^n(j) : (i,j)\notin \llbracket1,2-d\rrbracket \times \llbracket0,m-1\rrbracket)$.  By \eqref{eq:laweq1},  we have
\begin{align}
    & \ind_{\mathsf{B}_{n,d}}\Ex\left[\ind_{\mathsf{A}(L_1^n)}|\mathcal{F}_d\right] \stackrel{a.s.}{=} \ind_{\mathsf{B}_{n,d}}\frac{\mathbb{P}_d^{m,y^\ast}(\mathsf{A}(Q_1^{y^\ast})\cap \{Q_{2-d}^{y^\ast}(j) \ge g^{\ast}(j+1) \mbox{ for }j\in \llbracket 0,m-1\rrbracket\})}{\mathbb{P}_d^{m,y^\ast}(Q_{2-d}^{y^{\ast}}(j) \ge g^{\ast}(j+1) \mbox{ for }j\in \llbracket 0,m-1\rrbracket)} \label{gib}
\end{align}
where $y^{\ast}=(L_{1}^n(m), \ldots, L_{2-d}^n(m))$, $g^{\ast}=L_{3-d}^n\llbracket0,m\rrbracket$, and where
\begin{itemize}
    \item ($d=0$ case) for all $(x,y) \in \ZHS$, $\Pp_0^{m,(x,y)}:=\Pp_{\mathrm{Inter};q,c}^{m,(x,y)}$ (recall Definition \ref{def:gibbslaw}). 
    \item ($d=1$ case) for all $y \in \Z$, $\Pp_1^{m,y}:=\Pp_{\mathrm{Geom};c^{-1}q}^{m,y}$(recall Definition \ref{def:gibbslaw}). 
\end{itemize}

As $m\le n^{2/3}\delta^2/16\lambda^2$, for all $(\vec{y},g\llbracket0,m\rrbracket)\in U_{n,d}$ and for all large $n$  we have
\begin{align*}
    y_{2-d}+\frac{(j-m)q}{1-q}-\lambda\sqrt{m} \ge g(j+1)+n^{1/3}\delta/2-\frac{q}{1-q}-\lambda\sqrt{m} \ge g(j+1)+n^{1/3}\delta/4-\frac{q}{1-q} \ge g(j+1)
\end{align*}
for all $j\in \llbracket 0,m\rrbracket$. Thus for all $(\vec{y},g)\in U_{n,d}$, and for all large $n$ we have 
\begin{equation}
    \label{kolmineq2}
    \begin{aligned}
& \inf_{(\vec{y},g)\in U_{n,d}} \Pp_{d}^{m,\vec{y}}\left(Q_{2-d}^{\vec{y}}(j)  \ge g(j+1) \mbox{ for all }j\in \llbracket 0,m-1\rrbracket\right) \\ & 
\ge \inf_{(\vec{y},g)\in U_{n,d}}\Pp_{d}^{m,\vec{y}}\left(\sup_{j\in\llbracket0,m\rrbracket} Q_{2-d}^{\vec{y}}(j)-y+\frac{(m-j)q}{1-q} \ge -\lambda \sqrt{m} \right) \ge 1-\varepsilon
\end{aligned}
\end{equation}
where the last inequality is due to Lemma \ref{intkolm}. Specifically, Lemma \ref{intkolm}\ref{it:max1} is applicable as by our choice of $M_0,m$, we have $y_1-y_{2} \le M\lambda\sqrt{m}$ for all large enough $n$, whenever $(\vec{y},g)\in U_{n,0}$.

Given \eqref{kolmineq2}, we may estimate the denominator of righthand side of \eqref{gib}. Recalling that $\mathsf{B}_{n,d}=\widetilde{\mathsf{B}}_{n,d}\cap \mathsf{C}_{n,d}$, we thus have
\begin{align}\label{chain}
\ind_{\mathsf{C}_{n,d}}\cdot \mathbb{P}_{d}^{m,y^{\ast}}(\mathsf{A}(Q_1^{y^{\ast}}))-\varepsilon - \ind_{\widetilde{\mathsf{B}}_{n,d}^c}  \le \ind_{\mathsf{B}_{n,d}}\Ex_c\left[\ind_{\mathsf{A}(L_1^n)}|\mathcal{F}_d\right] \le  \left(1-\varepsilon\right)^{-1}\cdot \ind_{\mathsf{C}_{n,d}}\cdot\mathbb{P}_{d}^{m,y^{\ast}}(\mathsf{A}(Q_1^{y^{\ast}})) 
\end{align}
for all large $n$, almost surely. 
\begin{itemize}
\item ($d=0$ case) We may appeal to Lemma \ref{intconv} to get that for large enough $n$,
\begin{align*}
  \ind_{\mathsf{C}_{n,0}}(\gamma-\varepsilon) \le \ind_{\mathsf{C}_{n,0}}\cdot \mathbb{P}_{0}^{m,y^{\ast}}(\mathsf{A}(Q_1^{y^{\ast}})) \le \gamma+\varepsilon
\end{align*}
where $\gamma$ defined in \eqref{def:gammac}
Inserting this bound in \eqref{chain} leads to
\begin{align*}
    \ind_{\mathsf{C}_{n,0}}(\gamma-\varepsilon)-\varepsilon - \ind_{\widetilde{\mathsf{B}}_{n,0}^c}  \le \ind_{\mathsf{B}_{n,0}}\Ex_c\left[\ind_{\mathsf{A}(L_1^n)}|\mathcal{F}_0\right] \le  \left(1-\varepsilon\right)^{-1}(\gamma+\varepsilon). 
\end{align*}
Taking expectation in the above chain of inequalities, in view of \eqref{bnkpz} and \eqref{cnkpz} yields \eqref{claim} for $d=0$.
    \item ($d=1$ case)
By Remark \ref{rem:gibbslaw}, $(Q_1^y(\ell)-Q_1^y(0))_{\ell=1}^k$ is distributed as a $\mathrm{Geo}(c^{-1}q)$ random walk for every $y\in \mathbb{R}$. Thus, $\mathbb{P}_{1}^{m,y^{\ast}}(\mathsf{A}(Q_1^{y^{\ast}}))=\gamma$ defined in \eqref{def:gammac}. Thus taking expectations in \eqref{chain}, and using \eqref{bncri} yields \eqref{claim} for $d=1$ since $\mu_{c,r_c}$ is a $\Geo(qc^{-1})$ random walk in this case. \qed
\end{itemize}


\section{Completing the proofs of the main theorems} \label{sec:proofs}
This section is separated into four subsections. Section \ref{sec:1F1S proof} proves the one force--one solution principle. Section \ref{sec:directions} proves the result about the directions of semi-infinite geodesics in Theorem \ref{thm:geod_direction} and then uses this to prove Lemma \ref{lem:geodesic_directions_stop}, which controls the directions of semi-infinite geodesics associated to a stationary eternal solution. These results are used in Section \ref{sec:slopes_exist} to show that extremal invariant measures are supported on functions satisfying a slope condition (Proposition \ref{prop:slopes_exist}). Theorem \ref{thm:joint_invm_class} (and therefore also Theorem \ref{thm:main_thm}) is proved in Section \ref{sec:proof_complete}.

\subsection{Proof of the one force--one solution principle} \label{sec:1F1S proof}

We start by proving the following two results, which are corollaries to Proposition \ref{prop:recentered_conv}. 
\begin{corollary} \label{cor:bd_to_Buse}
    For any $t \in \Z$ and $k \in \N$, we have 
    \[
    \lim_{n \to \infty} \bigl[G((-n,-n),(t + k,t)) - G((-n,-n),(t,t))\bigr] = W_{\xi_{\max}^c}^c\bigl((t,t),(t+k,t), \quad\text{in probability}.
    \]
\end{corollary}
\begin{proof}
    Since the Busemann process is a family of eternal solutions with laws specified by Proposition \ref{prop:full_Busemann}\ref{it:Buse_dist},\ref{it:Buse_full_eternal}, we may apply Lemma \ref{lem:comp_to_eternal}\ref{it:upbd1} (specifically, Equation \eqref{eq:bd2}), to get that, for each $\xi \in (0,\xi_{\max}^c)$, and $k,n \in \N$,  
    \[
    G((-n,-n),(t +k ,t)) - G((-n,-n),(t,t)) \le W_{\xi}^c\bigl((t,t),(t+k,t)\bigr).
    \]
    Sending $\xi \nearrow \xi_{\max}^c$ and since the Busemann process is left-continuous (Proposition \ref{prop:full_Busemann}\ref{it:left_continuity}), we get 
    \[
    G((-n,-n),(t+k,t)) - G((-n,-n),(t,t)) \le W_{\xi_{\max}^c}^c\bigl((t,t),(t+k,t)\bigr).
    \]
    By Proposition \ref{prop:full_Busemann}\ref{it:Buse_dist}, we have $\Bigl(W_{\xi_{\max}^c}^c\bigl((t,t),(t+k,t)\bigr)\Bigr)_{k \in \N} \sim \mu_{c,r_c}$.  By the diagonal translation-invariance in Lemma \ref{lem:shift_sym}, we have 
    \[
    \Bigl(G\bigl((-n,-n),(t+k,t)\bigr)\Bigr)_{k \in \N} \deq \Bigl(G\bigl((1,1),(n+1+t+k,n + 1+t)\bigr)\Bigr)_{k \in \N} 
    \]
    The convergence in probability then follows from  Lemma \ref{lem:Xn_lims} and the distributional convergence in Proposition \ref{prop:recentered_conv}. 
\end{proof}

\begin{proposition} \label{prop:to_Buse_in_prob}
    Let $(f_n)_{n \ge 0}$ be a sequence of functions $f_n:\Z_{\ge -n} \to \R$ such that $\liminf_{n \to \infty} \f{1}{n}f_n(-n) \ge 0$ and
        \[
        \limsup_{n \to \infty} \f{1}{n} \max_{k \in \llbracket 0,n \rrbracket}[f_n(k - n-1) - X_c^{-1}(\xi_{\max}^c)k] \le 0, \quad\text{almost surely}.
        \]
    Then, for any $t \in \Z$ and $k \in \N$,
    \[
    \lim_{n \to \infty} \bigl[G_{f_n,-n}(t+k,t) - G_{f_n,-n}(t,t)\bigr] =W_{\xi_{\max}^c}^c\bigl((t,t),(t+k,t)\bigr),\quad\text{in probability}.
    \]
\end{proposition}
\begin{proof}
    By Lemma \ref{lem:exit_pt_comp_ptp}, since we always have $Z_{f,-n}(t,t) \ge - n$, for each $k,n \in \N$, we have 
    \be \label{eq:1234}
    G\bigl((-n,-n),(t+k,t)\bigr) -G\bigl((-n,-n),(t,t)\bigr) \le  G_{f_n,-n}(t+k,k) - G_{f_n,-n}(t,t).
    \ee
    On the other hand, Proposition \ref{prop:full_Busemann}\ref{it:boundary_attract} implies that for $k \ge 0$, 
    \be \label{eq:12345}
    \limsup_{n \to \infty} \bigl[G_{f_n,-n}(t,t) - G_{f_n,-n}(t+k,t)) \bigr] \le W_{\xi_{\max}^c }^c\bigl((t,t),(t+k,t)\bigr).
    \ee
    The result then follows from \eqref{eq:1234}-\eqref{eq:12345} and the convergence in Corollary \ref{cor:bd_to_Buse}.  
\end{proof}

We now prove the one force--one solution principle as a corollary of the more general results in Propositions \ref{prop:full_Busemann}\ref{it:Buse_attractive} and \ref{prop:to_Buse_in_prob}.  
\begin{proof}[Proof of Theorem \ref{thm:1F1S}]
For a function $f:\N \to \R$, recall the definition  of $f_n:\Z_{\ge -n} \to \R$ given by $f_n(k) = f(n+k+1)$. 

\medskip \noindent \textbf{Item \ref{it:1F1Spt1}:} Assume that 
\[
\lim_{k \to \infty} \f{f(k)}{k} = \theta
\]
for some $\theta > \f{qr}{1-qr}$. Then, for any $\ve > 0$ there exists $C > 0$ so that $|f(k) - \theta k| \le C +  \ve k$, so 
\[
\max_{k \in \llbracket 1,n \rrbracket }|f_n(k-n-1) - \theta k| = \max_{k \in \llbracket 1,n \rrbracket }|f(k) - \theta k| \le C +\ve n.
\]
Thus, 
\[
\limsup_{n \to \infty} \f{1}{n}\max_{k \in \llbracket 1,n \rrbracket}[|f_n(k-n-1) - \theta k|] \le \ve
\]
but since $\ve > 0$ is arbitrary, the limit is $0$. Thus, the condition of Proposition \ref{prop:full_Busemann}\ref{it:Buse_attractive} is satisfied for $\xi = X_c(\theta)$ (noting that,  because $\theta > \f{qr}{1-qr}$, we have $X_c(\theta) \in (0,\xi_{\max}^c)$ by Lemma \ref{lem:TX_maps}). Since each fixed value of $\xi \in (0,\xi_{\max}^c]$. is a continuity point of the Busemann process by Proposition \ref{prop:full_Busemann}\ref{it:cts_fixed_xi}, Proposition \ref{prop:full_Busemann}\ref{it:Buse_limits} gives us that, for each $t \in \Z$ and $k \in \Z_{\ge 0}$, 
\[
\lim_{n \to \infty}[G_{f_n,-n}(k+t,t) - G_{f_n,-n}(t,t)] = W_{\xi}^c\bigl((t,t),(t+k,t)\bigr),\quad\text{almost surely.}
\]

\medskip \noindent \textbf{Item \ref{it:1F1Spt2}:} Now assume that 
\[
\limsup_{k \to \infty} \f{f(k)}{k} \le \f{qr_c}{1-qr_c} = X_c^{-1}(\xi_{\max}^c)
\]
Then, $f_n(-n) = f(1)$, which is constant in $n$ and therefore bounded from below. Furthermore, by similar reasoning as in Item \ref{it:1F1Spt1},
\[
\limsup_{n \to \infty} \f{1}{n} \max_{k \in \llbracket 1,n \rrbracket}[f_n(k-n-1) - X_c^{-1}(\xi_{\max}^c) k] \le 0.
\]
Then, by Proposition \ref{prop:to_Buse_in_prob}, for $t \in \Z$ and $k \in \N$,
\[
    \lim_{n \to \infty} \bigl[G_{f_n,-n}(t+k,t) - G_{f_n,-n}(t,t)\bigr] =W_{\xi_{\max}^c}^c\bigl((t,t),(t+k,t)\bigr),\quad\text{in probability}. \qedhere
    \]
\end{proof}

\subsection{Directions of semi-infinite geodesics} \label{sec:directions}
 
 We begin by proving Theorem \ref{thm:geod_direction} from the main results section.

\begin{proof}[Proof of Theorem \ref{thm:geod_direction}]
    We first prove a weaker result, namely that for any semi-infinite geodesic, one of the following must hold:
    \[
    \begin{aligned}
    	&\lim_{k \to \infty} \f{\gamma(-k)\cdot \mbf e_1}{\gamma(-k) \cdot \mbf e_2} = \xi \quad \text{for some}\quad \xi \in [0,\xi_{\max}^c),\quad\text{or} \\
    	&\liminf_{k \to \infty} \f{\gamma(-k) \cdot \mbf e_1}{\gamma(-k) \cdot \mbf e_2} \ge \xi_{\max}^c
    \end{aligned}
    \]

    Assume, by way of contradiction, that this conclusion fails. This implies that there exists a semi-infinite geodesic $\gamma$ rooted at some point $\mbf y \in \ZHS$ such that 
    \[
    0 \le \liminf_{k \to \infty} \f{\gamma(-k) \cdot \mbf e_1}{\gamma(-k) \cdot \mbf e_2} < \min\Biggl(\xi_{\max}^c,\limsup_{k \to \infty} \f{\gamma(-k) \cdot \mbf e_1}{\gamma(-k) \cdot \mbf e_2} \Biggr),
    \]
    and thus we may choose $\xi \in (0,\xi_{\max}^c)$ so  that
    \be \label{eq:liminfsup_diff}
\liminf_{k \to \infty} \f{\gamma(-k) \cdot \mbf e_1}{\gamma(-k) \cdot \mbf e_2}  < \xi < \limsup_{k \to \infty} \f{\gamma(-k) \cdot \mbf e_1}{\gamma(-k) \cdot \mbf e_2}.
    \ee
    Let $\gamma_{\mbf y}^\xi$ be the semi-infinite geodesic constructed from the Busemann function. By Proposition \ref{prop:full_Busemann}\ref{it:Buse_direction}, $\gamma$ is to the left of $\gamma_{\mbf y}^\xi$ infinitely often and to the right of  $\gamma_{\mbf y}^\xi$ infinitely often. Thus, $\gamma_{\mbf y}^\xi$ is not the rightmost geodesic between any two of its points, a contradiction to Lemma \ref{lem:hs_geodesics}\ref{it:rightmost}.

    Now, to complete the proof, we show that if
    \be \label{eq:liminf_assmpt}
    \liminf_{k \to \infty} \f{\gamma(-k) \cdot \mbf e_1}{\gamma(-k) \cdot \mbf e_2} \ge \xi_{\max}^c,
    \ee
    then in fact we have that
    \[
    \lim_{k \to \infty} \liminf_{k \to \infty} \f{\gamma(-k) \cdot \mbf e_1}{\gamma(-k) \cdot \mbf e_2} \quad\text{exists and equals }1.
    \]
    When $c \le 1$, we have $\xi_{\max}^c = 1$, and the statement is immediate from the fact that $\gamma(-k) \cdot \mbf e_1 \ge \gamma(-k) \cdot \mbf e_2$ since $\gamma(-k) \in \ZHS$. We may therefore assume that $c > 1$. 
    
     Given such a semi-infinite geodesic $\gamma$ rooted at $\gamma(0) = (m,n)$,  Lemma \ref{lem:Buse_from_geod} implies that the following limit exists for all $\mbf x \le \gamma(0)$ in $\ZHS$:
     \[
     b(\mbf x) := \lim_{k \to \infty} G(\gamma(-k),\mbf x) - G(\gamma(-k),\gamma(0)),
     \]
     and for any $j \in \Z_{\le n}$, 
     \be \label{eq:gamma_is_max}
     \begin{aligned}
     &b(m,n) = \max_{i \in \llbracket j,m \rrbracket}[b(i,j-1) + G\bigl((i,j),(m,n)\bigr)],\quad\text{and} \\
     &\max\{i \in \Z: \gamma(-k) = (i,j-1) \text{for some }k \in \N \} \in \argmax_{i \in \llbracket j,m \rrbracket}\{b(i,j-1) + G\bigl((i,j),(m,n)\bigr)  \}.
     \end{aligned}
     \ee 
 By the assumption \eqref{eq:liminf_assmpt} and Proposition \ref{prop:full_Busemann}\ref{it:Buse_limits}, for $t \in \Z$ and $\ell \in \N$ such that $t+\ell \le m$, we have 
     \be \label{bleW}
     b(t+\ell,t) - b(t,t) \le W_{\xi_{\max}^c}^c\bigl((t,t),(t+\ell,t)\bigr).
     \ee
     On the other hand, for $k \in \N$, define $\gamma(-k) = (m_k,n_k)$. Then, by  Equation \eqref{eq:G_inc_comp} of Lemma \ref{lem:exit_pt_comp_ptp}, for sufficiently large $k \in \N$ so that the passage time below is well-defined, 
     \be \label{eq:Gtllb}
     G\bigl((m_k,n_k),(t+\ell,t)\bigr) - G\bigl((m_k,n_k),(t,t)\bigr) \ge G\bigl((n_k,n_k),(t+\ell,t)\bigr) - G((n_k,n_k),(t,t)).
     \ee
      By Corollary \ref{cor:bd_to_Buse}, on an event of probability one,  for every $t \in \Z$ and $\ell \in \N$, there exists a subsequence $u_p$ (depending on $t,\ell$) along which 
     \[
     \lim_{p \to \infty} G\bigl((-u_p,-u_p),(t+\ell,t)\bigr) - G\bigl((-u_p,-u_p),(t,t)\bigr) = W_{\xi_{\max}^c}^c\bigl((t,t),(t+\ell,t)\bigr),\quad\text{almost surely.} 
     \] 
     Note that the value of $n_k$ either stays the same or decreases by $1$ at each step of $k$. Thus, by passing \eqref{eq:Gtllb} to a subsequence and combining with \eqref{bleW}, we get that 
     \[
      b(t+\ell,t) - b(t,t) = W_{\xi_{\max}^c}^c\bigl((t,t),(t+\ell,t)\bigr),\quad\text{for all }t \in \Z \text{ and }\ell \in \N \text{ such that }t+\ell \le m.
     \]
     
     Since $\ell \mapsto W_{\xi_{\max}^c}^c\bigl((t,t),(t+\ell,t)\bigr)$ has law $\mu_{c,r_c}$ (Proposition \ref{prop:full_Busemann}\ref{it:Buse_dist}), Lemma \ref{lem:stat_obeys_slope} implies that the assumptions of Proposition \ref{prop:converge_to_max}\ref{it:mx_location} are satisfied for $\theta = T_c(c)$. Combined with \eqref{eq:gamma_is_max} and the fact that 
     \[
     \max\{i \in \Z: \gamma(-k) = (i,n_k - 1) \text{ for some } k \in \N \} \le m_k \le \max\{i \in \Z: \gamma(-k) = (i,n_k) \text{ for some } k \in \N \}, 
     \]
     we have that 
     $
     \lim_{k \to \infty} \f{m_k}{n_k} = 1,
     $ 
     as desired. 

     Lastly, Proposition \ref{prop:full_Busemann}\ref{it:Buse_direction} shows that there exists semi-infinite geodesics in every admissible direction $\xi > 0$. For $\xi = 0$, note that downward semi-infinite vertical rays are semi-infinite geodesics. 
 \end{proof}

Given Theorem \ref{thm:geod_direction}, we now define the following function of semi-infinite geodesics (here the parameters $q,c$ are fixed.
    \[
    \Xi(\gamma) :=  \lim_{k \to \infty} \f{\gamma(-k) \cdot \mbf e_1}{\gamma(-k) \cdot \mbf e_2}  \in [0,\xi_{\max}^c) \cup\{1\}.
    \]

\begin{lemma} \label{lem:geodesic_directions_stop}
    Let $b$ be any eternal solution for $\GLPPc$ such that for each $n \in \Z$, $\bigl(b(i,j)\bigr)_{i\ge j; j \le n}$ is independent of the weights on levels $n + 1$ and above, namely $(\omega_{i,j})_{i \ge j \ge n+1}$.  Let $\gamma_{\mbf y}$ be the family of semi-infinite geodesics constructed from this eternal solution as in Lemma \ref{lem:hs_geodesics}. Then, there exists $M \ge 0$ and $\xi \in [0,\xi_{\max}^c) \cup \{1\}$ so that $\Xi(\gamma_{(k,0)}) = \xi$ for all  $k \ge M$.
\end{lemma}
\begin{proof}
    We first observe that the function $k \mapsto \Xi(\gamma_{(k,0)}^b)$ is nonincreasing in $k$. Indeed, by Lemma \ref{lem:hs_geodesics}, if $\gamma_{(k,0)}$ and $\gamma_{(m,0)}$ ever meet, they stay together (recall by Remark \ref{rmk:slope_order} that larger values of $\xi$ correspond to the geodesic being further to the left). Assume, by way of contradiction, that the lemma is false.

 It suffices to prove the stronger statement that
    \[
    \Pp\bigl(\exists\; \text{integers } 0 \le m < k \quad \text{ such that }0 <  \Xi(\gamma_{(k,0)}) < \Xi(\gamma_{(m,0)}) < \xi_{\max}^c\bigr) = 0.
    \]
    
    Suppose, by way of contradiction, that this is not the case. Then, there exists $0 < q_2 < q_1 < \xi_{\max}^c$ with $q_1,q_2 \in \Q$ and $0 \le m < k$ such that 
    \[
    \Pp\bigl(q_2 < \Xi(\gamma_{(k,0)}) < q_1 < \Xi(\gamma_{(m,0)}) < \xi_{\max}^c\bigr) > 0.
    \]
In general, for $n \in \Z$, define the event 
\[
A_n = A_n(m,k,q_1,q_2) := \Bigl\{q_2 < \Xi(\gamma_{(n+k,n)}) < q_1 < \Xi(\gamma_{(n+m,n)})  < \xi_{\max}^c   \Bigr\}.
\]
Our assumption is that $\Pp(A_0) > 0$. We now claim that $\Pp(A_n) = \Pp(A_0)$ for all $n \in \Z$. For $(m,n) \in \ZHS$ and $j \in \Z$, define
\[
I_{(m,n)}(j) := \max\{i \in \Z: \gamma_{(m,n)}(-p) = (i,j-1) \quad\text{for some }p \ge 0 \}.
\]
By Lemma \ref{lem:geodesics_are_maximizers},
\begin{align*}
I_{(m,n)}(j) &= \max \Bigl\{\argmax_{i \in \llbracket j,m \rrbracket }\Bigl[b(i,j-1) + G\bigl((i,j),(m,n)\bigr)\Bigr]\Bigr\} \\
&=\max \Bigl\{\argmax_{i \in \llbracket j,m \rrbracket}\Bigl[b(i,j-1) - b(j-1,j-1) + G\bigl((i,j),(m,n)\bigr)\Bigr]\Bigr\}.
\end{align*}
Then, by the invariance of $b$, the independence assumption on $b$ and $G$ and the diagonal translation-invariance in Lemma \ref{lem:shift_sym}, for each $j \in \Z_{\le 0}$,
\begin{align*}
    &\quad \; \bigl(I_{(n+\ell,n)}(j):\ell \in \Z_{\ge 0} \bigr)  \\
    &\deq \Biggl(\max \Bigl\{\argmax_{i \in \llbracket j,n+\ell \rrbracket }\Bigl[b(i - n,j-1 - n) - b(j-1-n,j-1-n) + G\bigl((i - n,j - n),(\ell,0)\bigr)\Bigr]\Bigr\}\Biggr)  \\
    &= \bigl(I_{(\ell,0)}(j - n): \ell \in \Z_{\ge 0}\bigr).
\end{align*}
Dividing by $j$ and sending $j \to -\infty$, we see that 
\[
\bigl(\Xi(\gamma_{(n+\ell,n)}): \ell \in \Z_{\ge 0}\bigr) \deq \bigl(\Xi(\gamma_{(\ell,0)}): \ell \in \Z_{\ge 0}\bigr)
\]
for all $n \in \Z$. Thus, $\Pp(A_n)  = \Pp(A_0)$ for all $n \in \Z$, as desired.

We now claim that, if $\ind_{A_n} = 1$ for some $n \in \Z$, then there exists $N \le n$ such that $\ind_{A_p} = 0$ for all $p \le N$. For an illustration of the argument below, see Figure \ref{fig:geod_dir}. 

By \eqref{Wxi_and_xi_max}, we have  $\xi_{\max}^c \le 1$, so  we must have $q_1 < 1$. Assume that $\ind_{A_n} = 0$. Then, by definition of $I_{(n+m,n)}(j)$, we have 
\[
q_1 < \lim_{j \to -\infty} \f{I_{(n+m,n)}(j)}{j} = \Xi(\gamma_{(n+m,n)})  < 1.
\]
Hence, there exists $N \le n$ so that for all $p \le N$, 
\[
I_{(n + m,n)}(p) > p + k.
\]
Since the geodesics $\gamma_{\mbf x}$ and $\gamma_{\mbf y}$ stay together if they meet (Lemma \ref{lem:hs_geodesics}), for $j \le p$,
\[
I_{(n,m+n)}(j) \ge I_{(p+k,p)}(j),\quad\text{and therefore}\quad \Xi(\gamma_{(p+k,p)}) \ge \Xi(\gamma_{(n,m+n)}) > q_1 > q_2.
\]
Thus, $\ind_{A_p} = 0$, as desired. 

We have therefore shown that $\lim_{n \to -\infty} \ind_{A_n} = 0$. Indeed, if $\ind_{A_n} = 0$ for all $n$, then this is the case; otherwise, $\ind_{A_n} = 1$ for some $n\in \Z$, then there exists $N \le n$ so that $\ind_{A_p} = 0$ for all $p \le N$. By the bounded convergence theorem, $\lim_{n \to -\infty} \Pp(A_n) = 0$. But $\Pp(A_n) = \Pp(A_0)$ for all $n \in \Z$, so $\Pp(A_0) = 0$, a contradiction to our assumption.
\end{proof}

\begin{figure}
    \centering
    \includegraphics[width=0.5\linewidth]{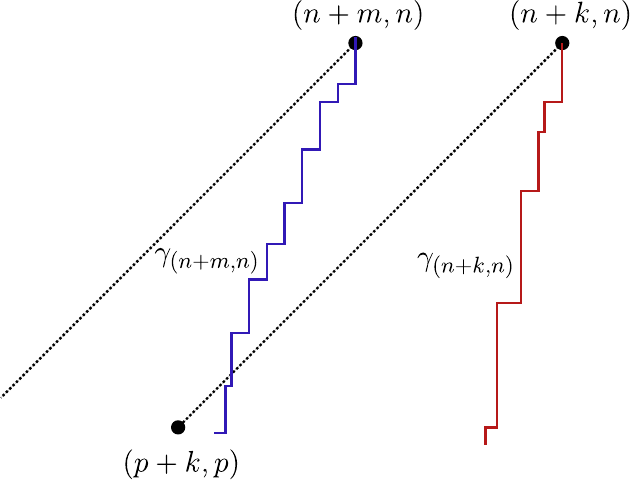}
    \caption{\small A depiction of the argument that, if $A_n$ holds, then $A_p$ fails for all $p$ sufficiently less than $n$. Dashed lines with slope $1$ are drawn from the points $(n+m,n)$ and $(n+k,n)$. Since the geodesic $\gamma_{(n+k,n)}$ has inverse slope less than $1$, the geodesic $\gamma_{(n+m,n)}$ eventually lies to the right of the point $(p+k,p)$. By ordering of geodesics, the inverse-slope of $\gamma_{(p+k,p)}$ is bounded below by the direction of $\gamma_{(n+m,n)}$, which is bounded below by $q_2$, and the event $A_p$ fails.  }
    \label{fig:geod_dir}
\end{figure}

\subsection{Existence of slopes for invariant measures} \label{sec:slopes_exist}
The goal of this subsection is to prove Proposition \ref{prop:slopes_exist}, which states that extremal invariant measures are supported on functions having a given asymptotic slope condition. This comes by combining two intermediate results. The first, Lemma \ref{lem:lim_exist_1} shows that for any invariant measure (not necessarily extremal), there exists a random slope almost surely (or else \eqref{eq:limsup_slope} holds). Then, in Lemma \ref{lem:slopes_pres}, we show that slopes are conserved quantities for the evolution, implying that an extremal invariant measure must be supported on a single (nonrandom) slope condition. 

\begin{lemma} \label{lem:lim_exist_1}
    Let $\nu$ be any invariant measure for $\GLPPc$. Then, if $f \sim \nu$, with probability one, one of the following holds.
    \begin{enumerate} [label=(\roman*), font=\normalfont]
        \item There exists $\theta > \f{qr_c}{1-qr_c}$ such that 
        \be \label{eq:slopefkt}
        \lim_{k \to \infty} \f{f(k)}{k} = \theta.
        \ee
        \item Or, we have 
        \be \label{eq:limsup_slope}
        \limsup_{k \to \infty} \f{f(k)}{k} \le \f{qr_c}{1-qr_c}.
        \ee
    \end{enumerate}
\end{lemma}
\begin{proof}
    By Lemma \ref{lem:eternal_construction}, on some probability space $(\wt \Omega, \wt{\mathcal F},\wt \Pp)$, there exists a family of weights $(\wt \omega_{\mbf x})_{\mbf x \in \ZHS}$ and an eternal solution $b$ so that for each $n \in \Z$,
    \[
    \bigl(b(n+k,n) - b(n,n)\bigr)_{k \in \N} \sim \nu.
    \]
    Furthermore, the eternal solution satisfies the independence property with the weights assumed in Lemma \ref{lem:geodesic_directions_stop}. It then suffices to show that, $\wt \Pp$-almost surely, the function $k \mapsto f(k) := b(k,0) - b(0,0)$ satisfies of of the conclusions of the lemma.
    
    On this probability space, we can construct the Busemann process as a measurable function of the weights $\wt \omega$ by Proposition \ref{prop:full_Busemann}. We will still denote this process as $W$ and we will henceforth denote our weights as $\omega$ instead of $\wt \omega$, keeping in mind that we are actually working on a different probability space. For $\mbf y \in \ZHS$, let $\gamma_{\mbf y}^b$ be the semi-infinite geodesic constructed from Lemma \ref{lem:hs_geodesics}. 
    By Lemma \ref{lem:geodesic_directions_stop}, we know there exists $M \ge 0$ and $\xi \in [0,\xi_{\max}^c]$ so that $\Xi(\gamma_{(k,0)}^b) = \xi$ for all integers $k \ge M$. 
    By Corollary \ref{cor:eternal_soln_recursion}, for all $m,n \in \Z_{\ge 0}$, 
    \[
    b(m,0) = G_{b(\aabullet,-n-1),-n}(m,0),
    \]
    and by Lemma \ref{lem:geodesics_are_maximizers}, 
    \[
    Z_{b(\aabullet,-n-1),-n}(m,0) = \max\Bigl\{i: \gamma_{(0,m)}^b(-\ell) = (\ell,-n-1) \text{ for some }\ell \in \Z_{\ge 0}\Bigr\}.
    \]

    We now consider three cases for the value of $\xi$: 

    \medskip \noindent \textbf{Case 1: $\xi \in (0,\xi_{\max}^c)$}. In this case, choose $\zeta$ and $\eta$ so that $0 < \zeta < \xi < \eta < \xi_{\max}^c$.
    Then, combined with the same argument for the eternal solutions defined by the Busemann functions and their directions from Proposition \ref{prop:full_Busemann}\ref{it:Buse_full_eternal},\ref{it:Buse_direction}, the ordering in Lemma \ref{lem:exit_pt_comp_ptl}  implies that, for all $k \ge M$, if  $0 <\zeta < \xi < \eta < \xi_{\max}^c$,
    \be \label{b_sandwich}
    W_\eta^c\bigl((M,0),(k,0)\bigr) \le b(k,0) - b(M,0)) \le W_\zeta^c\bigl((M,0),(k,0)\bigr).
    \ee
By Proposition \ref{prop:full_Busemann}\ref{it:Buse_dist}, for $\varphi \in \{\eta,\zeta\}$, the process $k \mapsto W_{\varphi}^c\bigl((0,0),(k,0)\bigr)$ has law $\mu_{c,s_\varphi}$ for $s_\varphi =(X_c \circ T_c)^{-1}(\varphi)$. Then, by the additivity of Proposition \ref{prop:full_Busemann}\ref{it:Additivity} and Lemma \ref{lem:inv_meas_slopes}, $\wt\Pp$-almost surely, for $\varphi \in \{\eta,\zeta\}$,
\[
\lim_{k \to \infty} \f{ W_\varphi^c\bigl((M,0),(k,0)\bigr)}{k} = \lim_{k \to \infty} \f{ W_\varphi^c\bigl((0,0),(k,0)\bigr) - W_\varphi^c\bigl((0,0),(M,0)\bigr) }{k} = X_c^{-1}(\varphi). 
\]
Thus, dividing all sides of \eqref{b_sandwich} by $k$, sending $k \to \infty$, then sending $\zeta \nearrow \xi$ and $\eta \searrow \xi$, we get 
\[
\lim_{k \to \infty} \f{b(k,0) - b(0,0)}{k} = \lim_{k \to \infty} \f{b(k,0) - b(M,0)}{k} = X_c^{-1}(\xi)  > \f{qr_c}{1-qr_c},
\]
 where in the last line, we used the assumption that $\xi \in (0,\xi_{\max}^c)$ and Lemma \ref{lem:TX_maps}\ref{it:X_fact}.  This gives us \eqref{eq:slopefkt}.

\textbf{Case 2: $\xi = \xi_{\max}^c$}. We follow a similar procedure, except we only get the upper bound
\[
b(k,0) - b(M,0) \le W_{\zeta}^c\bigl((M,0),(k,0)\bigr)
\]
for $\zeta \in (0,\xi_{\max}^c)$, and then we obtain
\[
\limsup_{k \to \infty} \f{b(k,0) - b(0,0)}{k} \le \f{qr}{1 - qr}. 
\]

\textbf{Case 3: $\xi = 0$.} We show that this case cannot happen. Following the same procedure as above, this would imply that 
\[
b(k,0) - b(M,0) \ge W_\eta^c\bigl((M,0),(k,0)\bigr),\quad\text{for all }k \in \Z_{\ge M},\quad\text{and all }\eta > 0.
\]
By Proposition \ref{prop:full_Busemann}\ref{it:full_Buse_mont}, the quantity is $W_\eta^c\bigl((M,0),(k,0)\bigr)$ is nondecreasing as $\eta \searrow 0$ and thus has a limit. As $\eta \searrow 0$, Lemma \ref{lem:TX_maps}\ref{it:X_fact} implies that $X_c^{-1}(\eta) \to \infty$ and so by Lemma \ref{lem:TX_maps}\ref{it:T_fact}, we have $\lim_{\eta \searrow 0}(X_c \circ T_c)^{-1}(\eta) = q^{-1}$. Hence, so by Lemma \ref{lem:mu_continuous}, we must have $b(k,0) - b(M,0) = \infty$, a contradiction.  

\end{proof}

The following lemma shows that asymptotic slopes are conserved quantities for half-space geometric LPP. 
\begin{lemma} \label{lem:slopes_pres}
    There exists a single event of probability one on which the following holds: For $j \in \Z$, whenever an initial condition $f:\Z_{\ge - j} \to \R$ satisfies 
    \[
    \limsup_{k \to \infty} \f{f(k)}{k} \le \theta ,\quad\text{for some }\theta \in \R, \text{then for all } n \in \Z_{\ge j}, \quad \limsup_{k \to \infty} \f{G_{f,j}(m,n)}{m} \le \theta
    \]
    Similarly, if 
    \[
    \liminf_{k \to \infty} \f{f(k)}{k} \ge \theta\quad \text{ for some $\theta > \f{q^2}{1-q^2}$, then for all $n \in \Z_{\ge j}$,}\quad \liminf_{k \to \infty} \f{G_{f,j}(m,n)}{m} \ge \theta.
    \]
\end{lemma}
\begin{proof}
    By the dynamic programming principle in Lemma \ref{lem:dynam_prog}, it suffices to prove the case where $n = j = 1$ (Recall that we use the shorthand notation $G_f = G_{f,1}$). For $m \in \N$, define
    \[
    S(m) = \sum_{\ell = 1}^m \omega_{(\ell,1)}.
    \]
    Since the weights $\omega_{(\ell,1)}$ for $\ell > 1$ are $\Geo(q^2)$, we have $\mathbb E[\omega_{(\ell,1)}] = \f{q^2}{1-q^2}$. We prove the first statement, and the second follows a similar proof.

   Choose $\ve > 0$. By the law of large numbers and the assumption on $f$, there is $C > 0$ so that 
    \[
    \Bigl|S(k) -\f{q^2}{1-q^2} k \Bigr| \le C  +  \ve k,\quad \text{and}\quad f(k) \le C + (\theta + \ve) k,\quad \text{for all }k \in \N.
    \]
    Then, we have 
    \begin{align*}
    G_{f}(m,1) &= \max_{k \in \llbracket 1,m \rrbracket}\Bigl[f(k) + G\bigl((k,1),(m,1)\bigr)\Bigr] = S(m) + \max_{k \in \llbracket 1,m \rrbracket}\Bigl[f(k)  - S(k-1)\Bigr] \\
    &\le 3C + \Bigl(\f{q^2}{1-q^2} + \ve\Bigr)m +  \max_{k \in \llbracket 1,m \rrbracket}\Bigl[\Bigl(\theta - \f{q^2}{1-q^2} + 2\ve\Bigr)k + \f{q^2}{1-q^2} - \ve\Bigr].
    \end{align*}
    From here, we consider two cases. If $\theta \ge \f{q^2}{1-q^2}$, then the bound becomes
    \[
    G_{f}(m,1) \le 3C  + \Bigl(\f{q^2}{1-q^2} + \ve\Bigr)m + \Bigl(\theta - \f{q^2}{1-q^2} + 2\ve\Bigr)m + \f{q^2}{1-q^2} - \ve = (\theta + 3\ve)m + O(1),
    \]
    and the resulting follows by dividing by $m$, sending $m \to \infty$, then sending $\ve \searrow 0$.

    On the other hand, if $\theta  < \f{q^2}{1-q^2}$,  we may choose $\ve > 0$ sufficiently small so that $\theta - \f{q^2}{1-q^2} + 2\ve < 0$, then our bound above becomes
    \[
    G_{f}(m,1) \le 3C +\Bigl(\f{q^2}{1-q^2} + \ve\Bigr)m + \f{q^2}{1-q^2} - \ve \le (\theta + \ve)m + O(1),
    \]
    and the result follows. 
\end{proof}

For the following proposition, we recall the definition of a jointly invariant measure from Definition \ref{def:joint_invariant}
\begin{proposition} \label{prop:slopes_exist}
    Let $m \in \N$, and assume that $\nu$ is an extremal invariant measure for $\GLPPc$ on the space $(\R^\N)^m$. Then, for each $1 \le i \le m$, one of the following one must hold.
    \begin{enumerate}  [label=(\roman*), font=\normalfont]
        \item There exists $\theta_i > \frac{qr_c}{1 - qr_c}$ such that 
        $
        \nu\Bigl(\Bigl\{\lim_{k \to \infty} \f{f_i(k)}{k} = \theta_i\Bigr\}\Bigr) = 1, \quad\text{or}$
        \item $\nu\Bigl(\Bigl\{\limsup_{k \to \infty} \frac{f_i(k)}{k} \le \f{qr_c}{1-qr_c}\Bigr\}\Bigr) = 1$. 
    \end{enumerate}
\end{proposition}
\begin{proof}
     Define the following random variable on the space $\R^\N$:
    \[
    \Theta(f) = \begin{cases}
        \theta &\lim_{k \to \infty} \f{f(k)}{k} = \theta \text{ for some }\theta > \f{qr_c}{1-qr_c} \\
        \f{qr_c}{1-qr_c} & \limsup_{k \to \infty} \f{f(k)}{k} \le \f{qr_c}{1-qr_c} \\
        \infty &\text{otherwise}.
    \end{cases}
    \]
    Let $\nu$ be an extremal $m$-jointly invariant measure on the space $(\R^\N)^m$. Since every $m$-jointly invariant measure is necessarily a coupling of jointly invariant measures, Lemma \ref{lem:lim_exist_1} implies that $\nu(\mathcal B) = 1$, where
    \[
    \mathcal B = \{(f_1,\ldots,f_m): \Theta(f_i) < \infty \text{ for } i \in \llbracket 1,m \rrbracket\}.
    \]
    We need to show that, for $i \in \llbracket 1,m \rrbracket$, there exists $\theta_1,\ldots,\theta_m \in \R$ so that, for each $i \in \llbracket 1,m \rrbracket$, we have  $\nu( \{\Theta(f_i) = \theta_i\}) = 1$. 
    Suppose, by way of contradiction, that this is not the case. Then, there exists $i \in \llbracket 1,m \rrbracket$ and $\theta  \ge \f{qr_c}{1-qr_c}$ so that $\nu\{\Theta(f_i) \le \theta\} \in (0,1)$ and $\nu\{ \Theta(f_i) > \theta\} \in (0,1)$.  Let $(\Omega,\mathcal F,\Pp)$ be the probability space of the geometric weights $(\omega_{\mbf x})_{\mbf x \in \ZHS}$, and for $n \ge 0$, let $G_f^n$ denote the function $k \mapsto G_{f_i}(n+k,n) - G_{f_i}(n,n)$ on the space $(\R^\N)^m \otimes \Omega$. Letting $\mathcal C$ be the $\Pp$-almost sure event from Lemma \ref{lem:slopes_pres}, we have 
    \begin{align*}
    &\{(f_1,\ldots,f_m): \Theta(f_i) \le \theta\} \cap (\mathcal B \times \mathcal C) =  \{(f_1,\ldots,f_m): \Theta(G_{f_i}^n) \le \theta\} \cap (\mathcal B \times \mathcal C),\quad\text{and} \\
    &\{(f_1,\ldots,f_m): \Theta(f_i) > \theta\}\cap (\mathcal B \times \mathcal C) =  \{(f_1,\ldots,f_m): \Theta(G_{f_i}^n) > \theta\}\cap (\mathcal B \times \mathcal C).
    \end{align*}
    Hence, by invariance of the measure $\nu$, we have for any measurable event $A \subseteq (\R^N)^m$,
    \begin{align*}
    &\quad \;\nu\{(f_1,\ldots,f_m) \in A|\Theta(f_i) \le \theta\}  \\
    &= (\nu \otimes \Pp)\{(G_{f_1}^n,\ldots,G_{f_m}^n) \in A|\Theta(G_{f_i}^n) \le \theta\} \\
    &=   (\nu \otimes \Pp)\{(G_{f_1}^n,\ldots,G_{f_m}^n) \in A|\Theta(f_i) \le \theta\}. 
    \end{align*}
    Thus, the measure defined by $A \mapsto \nu\{(f_1,\ldots,f_m) \in A|\Theta(f_i) \le \theta\}$ is an $m$-jointly invariant measure, and the same holds when we replace the event $\{\Theta(f_i) \le \theta\}$ with $\{\Theta(f_i) > \theta\}$. Thus, $\nu$ is a nontrivial convex combination of two $m$-jointly invariant measures and is therefore not extremal, a contradiction. 
\end{proof}

\subsection{Completing the proof} \label{sec:proof_complete}

\begin{proof}[Proof of Theorems \ref{thm:main_thm} and \ref{thm:joint_invm_class}]
By the distribution of the Busemann functions given in Proposition \ref{prop:full_Busemann}\ref{it:Buse_dist}, Theorem \ref{thm:main_thm} is exactly the $m = 1$ case of Theorem \ref{thm:joint_invm_class}, so we prove the latter. 
    Let $\nu$ be an extremal $m$-jointly invariant measure on the space $(\R^N)^m$. By Proposition \ref{prop:slopes_exist}, for each $i \in \llbracket 1,m \rrbracket$,  either there exists $\theta_i > \frac{qr_c}{1 - qr_c}$ such that 
        \[
        \nu\Bigl(\Bigl\{\lim_{k \to \infty} \f{f_i(k)}{k} = \theta_i\Bigr\}\Bigr) = 1,
        \]
        or
        \[
        \nu\Bigl(\Bigl\{\limsup_{k \to \infty} \frac{f_i(k)}{k} \le \f{qr_c}{1-qr_c}\Bigr\}\Bigr) = 1.
        \]
        Let $(f,\ldots,f_m) \sim \nu$. By invariance,  $
        \bigl(G_{f_i}(n+k,n) - G_{f_i}(n,n)\bigr)_{i \in \llbracket 1,m \rrbracket, k \in \N} \sim \nu$ for each $n \ge 1$. On the other hand, by Theorem \ref{thm:1F1S} and Corollary \ref{cor:shift_to_fn}, the process $\bigl(G_{f_i}(n+k,n) - G_{f_i}(n,n)\bigr)_{i \in \llbracket 1,m \rrbracket, k \in \N}$ converges in distribution to $\Bigl(W_{\xi_i}^c\bigl((0,0),(k,0)\bigr)\Bigr)_{i \in \llbracket 1,m \rrbracket, k \in \N}$
        for some choices of $\xi_1,\ldots,\xi_m \in (0,\xi_{\max}^c]$. Therefore, for these choices of $\xi_i$, we have 
        \[
\Bigl(W_{\xi_i}^c\bigl((0,0),(k,0)\bigr)\Bigr)_{i \in \llbracket 1,m \rrbracket, k \in \N} \sim\nu.\
        \]
        Next, the zero-temperature limit of \cite[Lemma 2.11]{Dauvergne-Zhang-2026} shows that the measures $\mu_{c,s_1,\ldots,s_m}$ are consistent; that is, if $(f_1,\ldots,f_m) \sim \mu_{c,s_1,\ldots,s_m}$, then $f_i \sim \mu_{c,s_i}$ for $i \in \llbracket 1,m \rrbracket$. Therefore, by Lemma \ref{lem:inv_meas_slopes},  for all $i \in \llbracket 1,m \rrbracket$,
        \[
        \lim_{k \to \infty} \f{f_i(k)}{k} = T_c(s_i),\quad \mu_{c,s_1,\ldots,s_m}-\text{almost surely},
        \]
        Then, the argument above shows that the law of 
         $\bigl(G_{f_i}(n+k,n) - G_{f_i}(n,n)\bigr)_{i \in \llbracket 1,m \rrbracket, k \in \N}$ converges in distribution to $\Bigl(W_{\xi_i}^c\bigl((0,0),(k,0)\bigr)\Bigr)_{i \in \llbracket 1,m \rrbracket, k \in \N}$, where  $s_i = (X_c \circ T_c)^{-1}(\xi_i)$ of $\xi_i < \xi_{\max}^c$ and $s_i = r_c$ if $\xi = \xi_{\max}^c$. But $\mu_{c,s_1,\ldots,s_m}$ is an $m$-jointly invariant measure, so for all $n \in \N$,
         \[
         \bigl(f_1(k),\ldots,f_m(k)\bigr)_{k \in \N} \deq \bigl(G_{f_i}(n+k,n) - G_{f_i}(n,n)\bigr)_{i \in \llbracket 1,m \rrbracket, k \in \N} \deq \Bigl(W_{\xi_i}^c\bigl((0,0),(k,0)\bigr)\Bigr)_{i \in \llbracket 1,m \rrbracket, k \in \N}.
         \]
         Hence, the law of $$\Bigl(W_{\xi_i}^c\bigl((0,0),(k,0)\bigr)\Bigr)_{i \in \llbracket 1,m \rrbracket, k \in \N}$$ must be $\mu_{c,s_1,\ldots,s_m}$, as desired.  
\end{proof}

\appendix
\section{Some auxiliary results and proofs}
We prove the following standard probabilistic facts.
\begin{lemma} \label{lem:Geo_RW_LD}
    For $q \in (0,1)$, let $S$ be an i.i.d. random walk with $\mathbb E[e^{tS(1)}] < \infty$ for all $t$ in a neighborhood of the origin. Let $\mu = \mathbb E[S(1)]$. Then, for every $\ve > 0$, there exist constants $0 < \alpha < \beta$  so that, for all integers $1 \le k \le n$, 
    \[
    \Pp\Bigl(\Bigl|S(k) - \mu k\Bigr| \ge \ve n\Bigr) \le e^{\alpha k - \beta n}.
    \]
\end{lemma}
\begin{proof}
    We first handle the upper bound. By the Markov inequality, for all sufficiently small $t > 0$, we have 
    \begin{align} \label{eq:Sk_exp}
        \Pp\Bigl(S(k) - \mu k  > \ve n\Bigr) &\le \exp \Bigl(\log \mathbb E[e^{tS(k)}] - t\mu k - t \ve n\Bigr) = \exp\bigl(\varphi(t)k - t\ve n\bigr),
    \end{align}
    where $\varphi$ is the function
    \[
    \varphi(t) = \log \mathbb E[e^{tS(1)}] - \mu t.
    \]
Similarly, for all sufficiently small $t > 0$, we have 
    \be \label{eq:SK_exp2}
    \Pp\Bigl(S(k) - \mu k  < -\ve n\Bigr) \le \exp\Bigl(\log \mathbb E[e^{-t S(k)}] + \mu tk - t\ve n\Bigr) = \exp\bigl(\varphi(-t)k - t\ve n\bigr).
    \ee

    Observe that $\varphi(0) = \varphi'(0) = 0$. Hence, we may choose $t > 0$ sufficiently small (depending on $\ve$) so that $\varphi(t) < t\ve$ and $\varphi(-t) < t\ve$. Then, for that choice of $t$, we set $\alpha = \max(\varphi(t),\varphi(-t))$ and $\beta = t\ve$ and obtain the desired upper bound from \eqref{eq:Sk_exp}-\eqref{eq:SK_exp2}.  
\end{proof}

\begin{lemma} \label{lem:Xn_lims}
    If $(X_n)_{n \ge 0}$ is a sequence of real-valued random variables satisfying 
    \[
    \limsup_{n \to \infty} X_n \le X, \quad \text{a.s.}
    \]
    for some real-valued random variable $X$, and $X_n$ converges in distribution to $X$, then $X_n$ converges to $X$ in probability. 
\end{lemma} 
\begin{proof}
    By assumption, $\limsup_{n \to \infty} \Pp(X_n - X > \ve) = 0$ for each $\ve > 0$. Now, consider the joint vector $(X,X_n)$. This vector is tight as $n \to \infty$, and all subsequential limits $(X,Y)$ satisfy $Y \deq X$ since $X_n \Longrightarrow X$. Furthermore, by the Portmanteau theorem, for any subsequential limit $(X,Y)$ along the subsequence $n_k$ and any $\ve > 0$,
    \[
    \Pp(Y > X + \ve) \le \liminf_{k \to \infty} \Pp(X_{n_k} > X + \ve) \le \limsup_{k \to \infty} \Pp(X_{n_k} > X + \ve) = 0.
    \]
    Sending $\ve \searrow 0$, we have $\Pp(Y > X) = 0$. Hence, $Y \le X$ a.s., so since $Y \deq X$, we have $Y = X$ a.s. Since this holds for any subsequential limit $Y$, we have that $(X,X_n) \Longrightarrow (X,X)$ as $n \to \infty$. Then, again using the Portmanteau theorem, for any $\ve > 0$,
    \[
    \limsup_{n \to \infty} \Pp(X_n - X \le -\ve) \le \Pp(X -X \le -\ve) = 0. \qedhere
    \]
\end{proof}

\subsection{Proof that Definition \ref{def:mucs_extend} matches Definition  \ref{def:mucs} in the case $k = 1$} \label{appx:match}
In this case, we choose $s_1 = s, s_2 = \f{1-\ve}{s}$, and $s_i = q$ for $i \ge 3$. Then, we have that $\mu_{c,s}$ is the law of the function $f:\N \to \R$, where 
\be \label{fkLPP}
f(k)  = \lim_{\ve \searrow 0}[ \mathcal G^\ve\bigl((2,1),(2+k,2)\bigr) - \mathcal G^\ve\bigl((2,1),(2,2)\bigr)].
\ee
We note that 
\[
\mathcal G^\ve\bigl((2,1),(2,2)\bigr) = Y^\ve_{(2,1)} + Y^{\ve}_{(2,2)},
\]
and 
\[
\mathcal G^\ve\bigl((2,1),(2+k,2)\bigr) = \max_{\ell \in \llbracket 2,2+k \rrbracket}\Biggl[\sum_{j = 2}^\ell Y^\ve_{(j,1)} + \sum_{j = \ell}^{2+k}Y^\ve_{(j,2)} \Biggr] = Y_{(2,1)}^\ve + \max_{\ell \in \llbracket 2,2+k \rrbracket}\Biggl[\sum_{j = 3}^\ell Y^\ve_{(j,1)} + \sum_{j = \ell}^{2+k}Y^\ve_{(j,2)} \Biggr], 
\]
where we interpret $\sum_{j = 3}^2 Y_{(j,1)}^\ve  = 0$. We have $Y_{2,1}^\ve \sim \Geo(1- \ve)$, which diverges as $\ve \searrow 0$, but this term is subtracted off in the difference \eqref{fkLPP}. The weight $Y_{(2,2)}^\ve$ has the $\Geo(\f{c(1-\ve)}{s})$ distribution, the weights $Y_{j,1}^\ve$ for $\ell \ge 3$ are i.i.d. $\Geo(qs)$, and the weights $Y_{(j,2)}$ for $\ell \ge 3$ are i.i.d. $\Geo(\f{q(1-\ve)}{s})$. Then, if we set 
\[
S_1^\ve(\ell) = \sum_{j = 3}^{\ell + 2} Y_{(j,1)}^\ve,\quad\text{and} \quad S_2^\ve(\ell) = \sum_{j = 3}^{\ell + 2} Y_{(j,2)}^\ve,\quad\text{for}\quad \ell \in \N,
\]
then $S_1$ is a $\Geo(qs)$ random walk, $S_2$ is a $\Geo(\f{q(1-\ve)}{s})$ random walk, and for $k \in \N$, we get 
\begin{align*}
    &\quad \;\mathcal G^\ve\bigl((2,1),(2+k,2)\bigr)  - \mathcal G^\ve\bigl((2,1),(2,2)\bigr)  \\
    &=  (Y_{(2,2)}^\ve + S_2^\ve(k))  \vee\Bigl(\max_{\ell \in \llbracket 3,2+k \rrbracket}[S_1^\ve(\ell-2) - S_2^\ve(k) - S_2^\ve(\ell-3)] \Bigr) - Y_{(2,2)}^\ve \\
    &=  S_2^\ve(k) + \Bigl(\max_{\ell \in \llbracket 1,k \rrbracket} [S_1^\ve(\ell) - S_2^\ve(\ell - 1)] - Y_{(2,2)}^\ve\Bigr)^+,
\end{align*}
and this matches Definition \ref{def:mucs} after sending $\ve \searrow 0$.

\bibliographystyle{goodbibtexstyle}
\bibliography{refs}

\end{document}